\documentclass[a4paper,11pt]{article}

\usepackage[latin1]{inputenc}
\usepackage[T1]{fontenc}
\usepackage[normalem]{ulem}
\usepackage[english]{babel}
\usepackage{verbatim}
\usepackage{graphicx}
\usepackage{relsize,enumitem,stackrel}
\usepackage{amsmath,amssymb,amsfonts,amsthm}
\usepackage{rotating}
\usepackage{float}
\usepackage{array}

\newtheorem*{thmA}{Theorem A}
\newtheorem*{thmB}{Theorem B}
\newtheorem*{thmC}{Theorem C}
\newtheorem{theorem}{Theorem}[subsection]

\newtheorem{lemma}[theorem]{Lemma}
\newtheorem{prop}[theorem]{Proposition}
\newtheorem{cor}[theorem]{Corollary}

\theoremstyle{definition}
\newtheorem{defn}[theorem]{Definition}
\newtheorem{rem}[theorem]{Remark}
\newtheorem{ex}[theorem]{Example}

\input xy
\xyoption{all}
\newcommand{\sbt}{\,\begin{picture}(-1,1)(0.5,-1)\circle*{1.8}\end{picture}\hspace{.05cm}}

\DeclareMathOperator{\pr}{pr}

\DeclareMathOperator{\id}{id}

\DeclareMathOperator{\map}{Map}

\DeclareMathOperator{\Top}{Top}
\DeclareMathOperator{\Ob}{Ob}

\DeclareMathOperator*{\hocolim}{hocolim}
\DeclareMathOperator*{\holim}{holim}

\DeclareMathOperator{\THH}{THH}

\DeclareMathOperator{\sym}{Sym}
\DeclareMathOperator{\hofib}{hof}
\DeclareMathOperator{\hoc}{hoc}
\DeclareMathOperator{\conn}{Conn}

\DeclareMathOperator{\res}{res}
\DeclareMathOperator{\KR}{KR}

\DeclareMathOperator{\K}{K}
\DeclareMathOperator{\Sp}{Sp}
\DeclareMathOperator{\Simp}{Simp}
\DeclareMathOperator{\GW}{GW}
\DeclareMathOperator{\HR}{HR}

\begin{document}
\begin{center}\LARGE{Equivariant calculus of functors and $\mathbb{Z}/2$-analyticity of Real algebraic $K$-theory}
\end{center}

\begin{center}\large{Emanuele Dotto}\let\thefootnote\relax\footnotetext{ Partially supported by
ERC Adv.Grant No.228082. and by the Danish National Research Foundation through the Centre 
for Symmetry and Deformation (DNRF92)
}\end{center}
\vspace{.3cm}

\begin{quote}
\textsc{Abstract}. We define a theory of Goodwillie calculus for enriched functors
from finite pointed simplicial $G$-sets to symmetric $G$-spectra, where $G$ is a finite group. We extend a notion of $G$-linearity suggested by Blumberg to define stably excisive and $\rho$-analytic
homotopy functors, as well as a $G$-differential, in this equivariant context. A main result of the paper is that analytic functors with
trivial derivatives send highly connected $G$-maps to $G$-equivalences.
It is analogous to the classical result of Goodwillie that ``functors with
zero derivative are locally constant''. As main example we show that
Hesselholt and Madsen's Real algebraic $K$-theory of a split square zero
extension of Wall antistructures defines an analytic functor in the
$\mathbb{Z}/2$-equivariant setting. We further show that the equivariant derivative of this Real $K$-theory functor is $\mathbb{Z}/2$-equivalent to Real MacLane homology.
\end{quote}

\section*{Introduction}

Calculus of functors was developed in Goodwillie's seminal papers \cite{calcI},\cite{calcII} and \cite{calcIII}, and found important applications in algebraic $K$-theory of rings \cite{McCarthy}, $A$-theory \cite{calcIII}, and stable mapping spaces \cite{Arone}. In the current paper, we are interested in developing a theory of equivariant calculus tailored to study the relationship between the Real algebraic $K$-theory of rings with Wall antistructures and their Real topological Hochschild and cyclic homology.

Hesselholt and Madsen define in \cite{IbLars} a functor $\KR$ that associates to a ring $A$ equipped with a Wall antistructure (in the sense of \cite{Wall}) a symmetric $\mathbb{Z}/2$-spectrum, with underlying spectrum equivalent to the algebraic $K$-theory of $A$. Given an $A$-bimodule $M$ with a suitable involutive structure, and a pointed finite simplicial $\mathbb{Z}/2$-set $X$, we define a symmetric $\mathbb{Z}/2$-spectrum $\widetilde{\KR}\big(A\ltimes M(X)\big)$ as the homotopy fiber of the projection map
\[\widetilde{\KR}\big(A\ltimes M(X)\big)=\hofib\Big(\KR\big(A\ltimes M(X)\big)\longrightarrow \KR(A)\Big)\]
Here $M(X)$ is a simplicial $A$-bimodule defined by the Dold-Thom construction, and $\ltimes$ denotes the semi-direct product. This construction is functorial in $X$, defining a functor from finite pointed simplicial $\mathbb{Z}/2$-sets to symmetric $\mathbb{Z}/2$-spectra. The analytic properties, in the sense of Goodwillie calculus, of the corresponding functor $\widetilde{\K}\big(A\ltimes M(-)\big)$ from pointed simplicial sets to spectra are studied extensively in \cite{McCarthy}, and play a crucial role in the fundamental relationship between algebraic $K$-theory and topological cyclic homology. The present paper studies the analogous analytic properties of the Real $K$-theory functor $\widetilde{\KR}\big(A\ltimes M(-)\big)$, in the sense of a suitable theory of equivariant calculus.

In its original form, functor calculus was developed for functors between categories of pointed spaces or spectra, but it was later extended to the generality of model categories in \cite{BO} and \cite{BCO}, or to quasi-categories in \cite{HH}. This homotopy theoretical calculus is however inadequate to study functors between categories of equivariant objects. As an example, let $G$ be a finite group and let $G$-$\Top_\ast$ be the category of pointed spaces with a $G$-action, equipped with the fixed points model structure. The first excisive approximation of functor calculus of a reduced enriched functor $F\colon G$-$\Top_\ast\to G$-$\Top_\ast$ is equivalent to the stabilization
\[P_1F(X)\simeq \hocolim\big(F(X)\longrightarrow \Omega F(\Sigma X)\longrightarrow \Omega^2 F(\Sigma^2 X)\longrightarrow \dots\big)\]
This stabilization produces a ``na\"{i}ve'' homology theory, instead of a ``genuine'' one, as this colimit does not take into account non-trivial representations of $G$. This phenomenon is the consequence of a definition of excision that is not adequate for dealing with equivariant homotopy theory. For enriched functors $F\colon G$-$\Top_\ast\to G$-$\Top_\ast$,  a suitable definition of $G$-excision was suggested by Blumberg in \cite{Blumberg}, where the author adds to excision an extra compatibility condition with equivariant Spanier-Whitehead duality. A similar condition was already present in \cite[1.4]{Shima} in the context of $\Gamma_G$-spaces.

\vspace{.5cm}

We follow Blumberg's idea, and we extend his definition of $G$-excision to homotopy functors enriched in $G$-spaces $F\colon G\mbox{-}\mathcal{S}_\ast\rightarrow \Sp^{\Sigma}_G$ from the category of finite pointed simplicial $G$-sets to symmetric $G$-spectra, for a finite group $G$. We define $F$ to be $G$-excisive if it sends homotopy pushout squares to homotopy pullback squares, and if for every finite $G$-set $J$ the canonical map
\[F(\bigvee_JX)\longrightarrow \prod_JF(X)\]
is an equivalence of $G$-spectra, for every pointed finite simplicial $G$-set $X$. Here $G$ acts on $X$ and permutes the $J$-indexed components, both on the wedge and on the product. The first analogy with classical excision is that a $G$-excisive functor $F\colon G\mbox{-}\mathcal{S}_\ast\rightarrow\Sp^{\Sigma}_G$ that sends the point to the point is $G$-equivalent to a functor of the form $C\wedge (-)$, for some genuine $G$-spectrum $C$ (see Corollary \ref{corlinearsmashspectrum}).

The differential (at a point) of a reduced enriched functor $F\colon G\mbox{-}\mathcal{S}_\ast\rightarrow\Sp^{\Sigma}_G$ is the stabilization
\[D_\ast F(X)=\hocolim_{n\in\mathbb{N}}\Omega^{n\rho}F(X\wedge S^{n\rho})\]
where $\rho$ is the regular representation of $G$. The fundamental property of the construction $D_\ast$ is that it sends stably $G$-excisive functors (Definition \ref{defstablin}) to $G$-excisive functors (Proposition \ref{derstablinislin}). Moreover for a stably $G$-excisive functor $F$ we have control on the connectivity of the canonical map $F\to D_{\ast}F$ (see Proposition \ref{connlinapprox}).

As an example of this construction, we identify the derivative of the Real $K$-theory functor 
\[\widetilde{\KR}\big(A\ltimes M(-)\big)\colon \mathbb{Z}/2\mbox{-}\mathcal{S}_\ast\longrightarrow\Sp^{\Sigma}_{\mathbb{Z}/2}\]
We define the Real MacLane homology $\HR(C;M)$ of an exact category with duality $C$ with coefficients in a bimodule with duality $M$ (Definition  \ref{defMacLane}). Its construction is analogous to the model of $\THH$ used in \cite[3.2]{DM}. In particular a bimodule $M$ over a ring with Wall antistructure $A$ induces a bimodule with duality on the category of finitely generated projective $A$-modules $\mathcal{P}_A$, with associated MacLane homology $\HR(A;M)$. The following is proved in Section \ref{splitting}.

\begin{thmA} Let $A$ be a ring with Wall antistructure and $M$ a bimodule over it (Definition \ref{bimodule}).
For every finite pointed simplicial $\mathbb{Z}/2$-set $X$, there is a natural $\pi_{\ast}$-equivalence of symmetric $\mathbb{Z}/2$-spectra
\[D_\ast\widetilde{\KR}\big(A\ltimes M(X)\big)\simeq \HR\big(A;M(S^{1,1})\big)\wedge|X|\]
where $M(S^{1,1})$ is the equivariant Dold-Thom construction of the sign-representation sphere $S^{1,1}$.
\end{thmA}
We mention for completeness that the author's thesis \cite{thesis} contains a theory of Real topological Hochschild homology, and it identifies the Real MacLane homology $\HR\big(A;M)$ with the Real topological Hochschild homology of $A$ with coefficients in $M$, at least in the case where $2$ is invertible in $A$. 

We further develop the theory of equivariant calculus by defining a notion of $G$-$\rho$-analytic functors. We do this in analogy with \cite[4.2]{calcII}, by introducing connectivity ranges in the definition of $G$-excision. In equivariant homotopy the connectivity of a map is a function on the subgroups of $G$, hence the $\rho$ above is a function $\rho\colon \{H\leq G\}\to \mathbb{Z}$. 
There is a construction of the derivative which is relative to a pointed finite simplicial $G$-set $B$. It gives rise to a functor $D_BF$ defined on the category of finite retractive simplicial $G$-sets over $B$ (see Definition \ref{defdifferential}). Theorem $B$ below is proved in Section \ref{secanalytic} and it is analogous to Goodwillie's result from \cite[5.4]{calcII} that ``functors with zero derivative are locally constant''.

\begin{thmB} A suitable $G$-$\rho$-analytic functor $F\colon G\mbox{-}\mathcal{S}_\ast\rightarrow\Sp^{\Sigma}_G$ whose derivatives are weakly $G$-contractible sends $\rho$-connected split-surjective equivariant maps to $\pi_\ast$-equivalences of symmetric $G$-spectra. In particular if $X$ is $\rho$-connected, $F(X)$ is $G$-contractible.
\end{thmB}

Our main example of a $G$-analytic functor is Hesselholt and Madsen's Real algebraic $K$-theory functor $\widetilde{\KR}\big(A\ltimes M(-)\big)$, for the group $G=\mathbb{Z}/2$. The following is proved in Section \ref{splitting}.

\begin{thmC} Let $M$ be a bimodule over a ring $A$ with Wall-antistructure (Definition \ref{bimodule}). The relative Real $K$-theory functor 
\[\widetilde{\KR}\big(A\ltimes M(-)\big)\colon \mathbb{Z}/2\mbox{-}\mathcal{S}_\ast\longrightarrow
\Sp^{\Sigma}_{\mathbb{Z}/2}\]
is $\mathbb{Z}/2$-$\rho$-analytic, where $\rho$ is the function on the subgroups of $\mathbb{Z}/2$ with values $\rho(1)=-1$ and $\rho(\mathbb{Z}/2)=0$.
\end{thmC}

%

In later work we will develop a theory of Real topological cyclic homology, receiving a trace map from Real algebraic $K$-theory. Theorems $A$, $B$ and $C$ will be crucial tools in establishing a relationship between the two theories, analogous to the Dundas-Goodwillie-McCarthy theorem of \cite{McCarthy} and \cite{DGM}.

\tableofcontents


\section{Preliminaries on enriched functors}

\subsection{Conventions about symmetric $G$-spectra}\label{spectrasec}

Let $G$ be a finite group. By a $G$-space, we will always mean a compactly generated Hausdorff space with a continuous action of the group $G$. We choose the category of symmetric $G$-spectra $\Sp^{\Sigma}_G$ of \cite{Mandell} as a model for stable equivariant homotopy theory, since Real algebraic $K$-theory fits naturally in this framework. 

\begin{defn}[\cite{Mandell}]
A symmetric $G$-spectrum consists of a well pointed $(\Sigma_n\times G)$-space $E_n$ for every $n$ in $\mathbb{N}$, and pointed $(\Sigma_n\times\Sigma_m\times G)$-equivariant maps
\[\sigma_{n,m}\colon E_n\wedge S^{m\rho}\longrightarrow E_{n+m}\]
satisfying the classical compatibility conditions. Here $S^{m\rho}$ is the one point compactification of the direct sum of $m$-copies of the regular representation $\rho=\mathbb{R}[G]$, and $G$ acts diagonally on the source of $\sigma_{n,m}$.
A $G$-map of $G$-spectra $E\to W$ is a collection of $(\Sigma_n\times G)$-equivariant maps $E_n\to W_n$ which respect the structure maps. The resulting category of symmetric $G$-spectra and $G$-maps is denoted $\Sp^{\Sigma}_G$.
\end{defn}

Let us explain which kind homotopical structure we consider on the category $\Sp^{\Sigma}_G$. We recall from \cite[5.1]{Mandell} that for every subgroup $H\leq G$, the $H$-homotopy groups of a $G$-spectrum $E$ are defined as
\[\pi^{H}_k E=\pi_0\hocolim_n(\Omega^{n\rho+k}E_n)^H\]
Here $\Omega^{n\rho+k}$ is the space of pointed maps from the sphere $S^{n\rho}\wedge S^k$ for $k\geq 0$, and from the sphere $S^{-k\overline{\rho}}\wedge S^{(n+k)\rho}$ for $k<0$, where the first smash factor is the representation sphere of $-k$ copies of the reduced regular representation $\overline{\rho}$ of $G$. 
The maps in the homotopy colimit system are induced by the adjoints of the structure maps $\sigma_{n,m}$ of $E$.
In order to carry out connectivity arguments, we will need to take as equivalences between $G$-spectra the $\pi_\ast$-equivalences, which do not coincide with stable equivalences of the stable model structure of $\Sp^{\Sigma}_G$. Because of this discrepancy between stable and $\pi_\ast$-equivalences, we avoid talking about model structures on $\Sp^{\Sigma}_G$ all together. By the homotopy limit and colimit of a diagram $X\colon I\rightarrow \Sp^{\Sigma}_G$ we will mean the raw Bousfield-Kan formulas
\[\holim_IX=\hom\big(NI/_{(-)},X\big) \ \ \ \ \ \ \hocolim_IX=\big((-)/NI\big)^{op}\otimes X\]
(see e.g. \cite[18.1.2-18.1.8]{Hirschhorn}). Here the cotensor and the tensor structures of $\Sp^{\Sigma}_G$ over simplicial sets are levelwise, and therefore so are homotopy limits and colimits.
In particular homotopy pullbacks and homotopy pushouts are formed levelwise. As level fibrations and level cofibrations induce long exact sequences in equivariant homotopy groups (see \cite[5.7-5.8]{Mandell}), homotopy pullbacks and homotopy pushouts preserve $\pi_\ast$-equivalences of symmetric $G$-spectra. This is all the homotopical information we are going to need and use about  $\pi_\ast$-equivalences in $\Sp^{\Sigma}_G$. In a context where the maps with arbitrarily high connectivity are the stable equivalences, for example the category of orthogonal $G$-spectra of \cite{ManMay} or \cite{Schwede}, all the results of the present paper can be interpreted in model categorical terms.

A fiber sequence of symmetric $G$-spectra is a sequence of $G$-equivariant maps of symmetric $G$-spectra  $F\rightarrow E\stackrel{f}{\rightarrow }W$ together with a $\pi_\ast$-equivalence $F\rightarrow\hofib(f)$ over $E$, where $\hofib(f)$ is the homotopy fiber of $f$. Similarly a cofiber sequence is a sequence of $G$-equivariant maps of symmetric $G$-spectra  $ E\stackrel{f}{\rightarrow }W\rightarrow C$ together with a $\pi_\ast$-equivalence under $W$ from the homotopy cofiber $\hoc(f)$ of $f$ to $C$.

\begin{rem}\label{fibhofib}
The canonical map $\hofib(f)\rightarrow \Omega\hoc (f)$ induced by taking horizontal homotopy fibers in the square
\[\xymatrix{E\ar[r]^-f\ar[d]&W\ar[d]\\
\ast\ar[r]&\hoc(f)
}\]
is a $\pi_\ast$-equivalence (see e.g. \cite[5.7-5.8]{Mandell}). This shows that every fiber sequence is canonically a cofiber sequence, and vice versa. This has the consequence that every homotopy cocartesian square of $G$-maps in $\Sp^{\Sigma}_G$ is also homotopy cartesian, and the other way around.
\end{rem}

We define an enrichment of $\Sp^{\Sigma}_G$ in the category $G$-$\Top_\ast$ of pointed $G$-spaces. Given two symmetric $G$-spectra $E$ and $W$ let $\Sp^\Sigma(E,W)$ be the set of collections of $\Sigma_n$-equivariant pointed maps $\{f_n\colon E_n\rightarrow W_n\}_{n\geq 0}$ that commute with the structure maps of $E$ and $W$. Endow $\Sp^\Sigma(E,W)$ with the subspace topology of the product of the mapping spaces $\map_\ast(E_n,W_n)$. The space  $\Sp^\Sigma(E,W)$ inherits a $G$-action by conjugation, defining a $G$-$\Top_\ast$-enrichment on $\Sp^{\Sigma}_G$. The $G$-fixed set $\Sp^\Sigma(E,W)^G$ is the set of morphisms of symmetric $G$-spectra from $E$ to $W$. For a subgroup $H\leq G$, an element of $\Sp^\Sigma(E,W)^H$ is called an $H$-equivariant map of symmetric $G$-spectra.

\begin{defn}
Let $E$ be a symmetric $G$-spectrum and let $\nu\colon \{H\leq G\}\to\mathbb{Z}$ be a function which is invariant on conjugacy classes. We say that $E$ is $\nu$-connected if $\pi^{H}_k E=0$ for every $k\leq \nu(H)$ and every subgroup $H$ of $G$. Let $\conn E$ be the largest of these functions, and let us denote its value at a subgroup $H$ by $\conn_H E$.
 A $G$-spectrum $E$ is weakly $H$-contractible if $\pi^{K}_\ast E=0$ for every subgroup $K$ of $H$.
\end{defn}

\begin{defn}\label{spGconn} Let  $f\colon E\to W$ be a $G$-equivariant map in $\Sp^{\Sigma}_G$, and let $\nu\colon\{H\leq G\}\rightarrow \mathbb{Z}$ be a function which is invariant on conjugacy classes. We say that $f$ is $\nu$-connected if its homotopy fiber is $(\nu-1)$-connected.
We say that $f$ is an $H$-equivalence if it induces an isomorphism in $\pi^{K}_\ast$ for every subgroup $K$ of $H$. In particular a $G$-equivalence is a $\pi_\ast$-equivalence.
\end{defn}

Categorical limits and colimits in $\Sp^{\Sigma}_G$ are also degreewise, in particular products and coproducts. The inclusion of the coproduct into the product of symmetric $G$-spectra is a $G$-equivalence, essentially by Remark \ref{fibhofib} above. In our context of equivariant calculus it is going to be a key point to consider coproducts and products which are indexed on finite sets with a non-trivial $G$-action.
Given a finite $G$-set $J$ and a well pointed $G$-space $X$  define $\bigvee_JX$ to be the coproduct of one copy of $X$ for every element in $J$, with $G$-action defined by $g(x,j)=(gx,gj)$. Define similarly a $G$-action on the product $\prod_JX$, by sending a $J$-tuple $\underline{x}$ to the $J$-tuple with $j$-component
\[g(\underline{x})_j=gx_{g^{-1}j}\]
The inclusion of wedges into products induces a $G$-equivariant map $\bigvee_JX\longrightarrow \prod_JX$.
As limits and colimits of $G$-spectra are levelwise, the analogous constructions for a symmetric $G$-spectrum $E$ gives a $G$-equivariant map of symmetric $G$-spectra
\[\bigvee_JE\longrightarrow \prod_JE\]
This map is a $\pi_\ast$-equivalence as a consequence of the Wirthm\"{u}ller isomorphism theorem (see e.g. \cite[\S 4]{Schwede}).
We conclude the section by proving an analogous result for a relative version of this map (Proposition \ref{wedgesintoprod} below), under some extra connectivity assumptions. Let $p\colon X\rightarrow B$ be an equivariant map of well-pointed $G$-spaces, and suppose that it has a $G$-equivariant section $s\colon B\rightarrow X$. Define $\bigvee_{B}^{J}X$ and $\widetilde{\prod}_{B}^{J}X$ respectively as the pushout and homotopy pullback squares
\[\xymatrix@=17pt{\bigvee_{\!\!J}B\ar[r]^-{\bigvee_J s}\ar[d]_{fold}&\bigvee_{\!\!J}X\ar[d]^{}\\
B\ar[r]&{\bigvee\limits_{B}}^{J}X}
\ \ \ \ \ \ \ \ \ \ \ \ \ \ \ \ \ \
\xymatrix@=15pt{\widetilde{\prod\limits_{B}}^{J}X\ar[r]&\prod_JX\ar[d]^{\prod_J p}\\
B\ar[r]_-{\Delta}\ar@{<-}[u(.8)]&\prod_JB
}\]
The canonical inclusion of wedge into products induces a $G$-equivariant map
\[{\bigvee\limits_{B}}^{J}X\longrightarrow \widetilde{\prod\limits_{B}}^{J}X\]
which factors through the categorical pullback.
The asymmetry between the homotopy pullback on the right and the categorical pushout on the left becomes homotopically irrelevant when either $s\colon B\rightarrow X$ is a cofibration or $p\colon X \rightarrow B$ is a fibration of $G$-spaces (in particular when $B$ is a point).
This construction extends levelwise to split maps of $G$-spectra $p\colon E\rightarrow B$, resulting in a $G$-equivariant map of $G$-spectra
${\bigvee\limits_{B}}^{J}E\longrightarrow \widetilde{\prod\limits_{B}}^{J}E$.

\begin{prop}\label{wedgesintoprod} Let $p\colon E\rightarrow B$ be a split $G$-map of symmetric $G$-spectra. Suppose that for every positive integer $n$ the spaces $E^{H}_n$ and $B^{H}_n$ are $(n|G/H|-c)$-connected for some integer $c$ independent on $n$. Then
the inclusion of wedges into products
\[{\bigvee\limits_{B}}^{J}E\longrightarrow \widetilde{\prod\limits_{B}}^{J}E\]
is a $G$-equivalence.
In particular for $B=\ast$ the map $\bigvee_JE\rightarrow \prod_JE$ is a $G$-equivalence.
\end{prop}

\begin{rem} An example of a $G$-spectrum which satisfies the connectivity hypothesis above is the suspension spectrum $(X\wedge \mathbb{S}^G)_n=X\wedge S^{n\rho}$ of a well-pointed $G$-space $X$.
Proposition \ref{wedgesintoprod} holds without any connectivity assumption, by a relative version of the Wirthm\"{u}ller isomorphism theorem. In the present paper, we will use this result only in the presence of this strong connectivity hypothesis. Moreover the connectivity range of Lemma \ref{connwedgesitoprodtop} below will be used throughout the paper, and it is the motivation for ``the equivariant part'' of the definition of stably $G$-linear functors \ref{defstablin}.
\end{rem}

\begin{proof}[Proof of \ref{wedgesintoprod}]
In lemma \ref{connwedgesitoprodtop} below we show that the inclusion in spectrum degree $n$
\[\iota_n\colon{\bigvee\limits_{B}}^{J}E_n\longrightarrow \widetilde{\prod\limits_{B}}^{J}E_n\]
is $\nu_n$-connected, for $\nu_n(H)=\min\{2\conn p^{H}_n-1,\min_{K\lneqq H}\conn p^{K}_n\}$.
The homotopy fiber of $p$ satisfies the same connectivity hypothesis of $E$ and $B$. Thus there is a constant $c$ such that 
\[\nu_n(H)\geq \min\{2n|G/H|,\min_{K\lneqq H}n|G/K|\}-c\]
The the connectivity of the map induced by $\iota_n$ at the $n$-th stage of the homotopy colimit defining $\pi_{\ast}^H$ is therefore
\vspace{-.3cm}

\[\begin{array}{lll}\min_{K\leq H}\big(\nu_n(K)-\dim (S^{n\rho})^K\big)\\
\geq\min_{K\leq H}\big(\min\{2n|G/K|,\min_{L\lneqq K}n|G/L|\}-c-n|G/K|\big)\\
\geq \min_{K\leq H}\min\{n|G/K|,n\}-c\geq n-c
\end{array}\]
This goes to infinity with $n$, showing that the homotopy fiber of $\iota$ is weakly $G$-contractible.
\end{proof}

\begin{lemma}\label{connwedgesitoprodtop}
For every well-pointed $G$-space $X$, the canonical map ${\bigvee\limits_{B}}^{J}X\rightarrow \widetilde{\prod\limits_{B}}^{J}X$ is $\nu$-connected, for
\[\nu(H)=\min\{2\conn p^H-1,\min_{K\lneqq H}\conn p^K\}\]
\end{lemma}

\begin{proof}[Proof of \ref{connwedgesitoprodtop}]
Let us describe the map $\iota$ on $H$-fixed points
\[\iota^{H}\colon({\bigvee\limits_{B}}^{J}X)^H\longrightarrow (\widetilde{\prod\limits_{B}}^{J}X)^H\]
for every subgroup $H$ of $G$. Homotopy limits (which are defined by the Bousfield-Kan formula) and pushouts commute with fixed points. The source of $\iota^H$ is then homeomorphic to the space ${\bigvee\limits_{B^H}}^{J^H}X^H$.
The target of $\iota^H$ is the homotopy pullback
\[\xymatrix@R=13pt{(\widetilde{\prod\limits_{\substack{B\\ \phantom{a}\\ \phantom{a} }}}^{J}X)^H \ar[r]&(\prod_JX)^{H}\cong\prod\limits_{[j]\in J/H}X^{H_j}\cong (\prod\limits_{J^H}X^{H})\times\prod\limits_{\substack{[j]\in J/H\\|[j]|\geq 2}}X^{H_j}\\
B^{H}_{\substack{\phantom{B}\\\phantom{B}\\ \phantom{B}\\ \phantom{B} }}\ar[r]_-{\Delta}\ar@{<-}[u(.7)]&(\prod_JB)^{H}\cong\prod\limits_{[j]\in J/H}B^{H_j} \cong (\prod\limits_{J^H}B^{H})\times\prod\limits_{\substack{[j]\in J/H\\|[j]|\geq 2}}B^{H_j}\ar@{<-}@<-12ex>[u(.7)]
}\]
where $H_j=\{h\in H|hj=j\}$ is the stabilizer group of $j$ in $H$. The products on the right-hand side range over a set of representatives for the equivalence classes in the quotient $J/H$. We choose the same set of representatives for the products of $B$ and $X$. The map $\iota^H$ factors as
\[{\bigvee\limits_{B^H}}^{J^H}X^H\rightarrow
\widetilde{\prod\limits_{B^H}}^{J^H}\!\!\!X^H=
\holim\left(
\vcenter{\hbox{\xymatrix@R=13pt{ B^{H}\ar[d]\\\prod\limits_{J^H}B^H\\\prod\limits_{J^H}X^H\ar[u]}}}\right)\rightarrow
\holim\left(
\vcenter{\hbox{\xymatrix@R=10pt{B^{H}\ar@<-1ex>[d]\\(\prod\limits_{J^H}B^{H})\times\!\!\!\!\prod\limits_{\substack{[j]\in J/H\\|[j]|\geq 2}}B^{H_j}\\(\prod\limits_{J^H}X^{H})\times\!\!\!\!\prod\limits_{\substack{[j]\in J/H\\|[j]|\geq 2}}X^{H_j}\ar@<1ex>[u(.6)]}}}\right)
\cong
(\widetilde{\prod\limits_{B}}^{J}X)^H
\]
The first map is the canonical inclusion. It is $2\conn p^H-1$ connected by the Blakers-Massey theorem (see e.g. \cite[2.3]{calcII}), applied inductively to the homotopy pushout square of spaces (without any group action)
\[\xymatrix{X^H\vee_{B^H}X^H\ar[r]\ar[d]&X^H\ar[d]\\
X^H\ar[r]&B^H
}\]
Notice that the pinch map $X^H\vee_{B^H}X^H\rightarrow X^H$ is at least as connected as $p^H\colon X^H\rightarrow B^H$. The second map is induced on homotopy limits by the inclusions in the first product factors. Its homotopy fiber is equivalent to the homotopy limit
\[\holim\bigg( \ast\longrightarrow\prod\limits_{\substack{[j]\in J/H\\|[j]|\geq 2}}B^{H_j}\longleftarrow \prod\limits_{\substack{[j]\in J/H\\|[j]|\geq 2}}X^{H_j}\bigg)\cong\prod\limits_{\substack{[j]\in J/H\\|[j]|\geq 2}}\hofib p^{H_j}\]
which is $(\min_{K\lneqq H}\conn p^K-1)$-connected. Then $\iota^H$ is $\nu$-connected, where $\nu$ is the minimum of the two connectivities $\nu(H)=\min\{2\conn p^H-1,\min_{K\lneqq H}\conn p^K\}$.
\end{proof}


\subsection{Enriched homotopy functors and assembly maps}\label{secsetup}

Let $G$ be a finite group, and let $B$ be a pointed finite simplicial $G$-set. Consider the category $G$-$\mathcal{S}_B$ of finite retractive simplicial $G$-sets over $B$. An object of $G$-$\mathcal{S}_B$ is a triple $(X,p,s)$ of a finite simplicial $G$-set $X$, an equivariant simplicial map $p\colon X\rightarrow B$ and an equivariant section $s\colon B\rightarrow X$ of $p$. We remark that $s$ is in particular a cofibration in the fixed points model structure of pointed simplicial $G$-sets, see e.g. \cite[1.2]{Shipley}. A morphism of $G$-$\mathcal{S}_B$ is an equivariant map that commutes with both the projections and the sections. The category $G$-$\mathcal{S}_B$ admits an enrichment in $G$-$\Top_\ast$.
The space of morphisms $\map_B(X,Y)$ from $(X,p_X,s_X)$ to $(Y,p_Y,s_Y)$ is the geometric realization of the simplicial subset $\map_B(X,Y)_{\sbt}$ of the simplicial mapping space $\map(X,Y)_{\sbt}$ of relative maps. Its $p$-simplices are the simplicial maps $f\colon X\times\Delta[p]\to Y$ for which the squares
\[\xymatrix{X\times \Delta[p]\ar[r]^-f\ar[d]_{p_X\times \id}&Y\ar[d]^{p_Y}\\
B\times \Delta[p]\ar[r]&B
} \ \ \ \ \ \ \ \ \ \ \ \ \ \ \ \ \ \xymatrix{X\times \Delta[p]\ar[r]^-f&Y\\
B\times \Delta[p]\ar[u]^{s_X\times \id}\ar[r]&B\ar[u]_{s_Y}
}\]
commute. Here the bottom horizontal maps are projections, and $\Delta[p]$ is the simplicial $p$-simplex. The group $G$ acts simplicially on $\map_B(X,Y)_{\sbt}$ by conjugation, inducing a $G$-action on the geometric realization $\map_B(X,Y)$. If $G=\{1\}$ is the trivial group, we write $\mathcal{S}_B$ for the category $\{1\}$-$\mathcal{S}_B$.

We are interested in studying ($G$-$\Top_ \ast$)-enriched functors
\[\Phi\colon G\mbox{-}\mathcal{S}_B\longrightarrow \Sp^{\Sigma}_G\]
which arise from $\Top_ \ast$-enriched functors $\Psi\colon \mathcal{S}_B\rightarrow \Sp^{\Sigma}_G$ via the following construction. Given an object $(X,p,s)$ of $G$-$\mathcal{S}_B$, regard the action maps as endomorphisms $g\colon (X,p,s)\rightarrow (X,p,s)$ in $\mathcal{S}_B$. By functoriality, the $G$-spectrum $\Psi(X)$ inherits an extra $G$-action by the maps $\Psi(g)$, and the diagonal actions
\[\Psi(X)_n\stackrel{g}{\longrightarrow}\Psi(X)_n\stackrel{\Psi(g)}{\longrightarrow}\Psi(X)_n\]
define a new $G$-spectrum $\overline{\Psi}(X)$. This construction extends $\Psi$ to a functor $\Phi=\overline{\Psi}\colon G\mbox{-}\mathcal{S}_B\rightarrow \Sp^{\Sigma}_G$. The technical advantage of a functor $\Phi$ of this form, is that it can be evaluated at retractive spaces over $B$ that have only an action of a subgroup $H$ of $G$. Compose $\Psi\colon \mathcal{S}_B\rightarrow \Sp^{\Sigma}_G$ with the restriction functor to $H$-spectra $\Sp^{\Sigma}_H$ and extend it to a functor $H\mbox{-}\mathcal{S}_B\rightarrow \Sp^{\Sigma}_H$ in the above fashion. Since $\Psi$ is enriched, this sends $H$-equivariant simplicial homotopy equivalences to $H$-equivariant homotopy equivalences of symmetric $H$-spectra. An $H$-equivariant map in $H\mbox{-}\mathcal{S}_B$ is an $H$-equivalence if the underlying map of simplicial $H$-sets is a weak equivalence in the fixed point model structure.

\begin{defn}\label{htpyfctr}
An enriched functor $\Phi\colon G\mbox{-}\mathcal{S}_B\rightarrow \Sp^{\Sigma}_G$ is a homotopy functor if it is extended from a functor $\Psi\colon \mathcal{S}_B\rightarrow \Sp^{\Sigma}_G$ as explained above, such that the corresponding extension $H\mbox{-}\mathcal{S}_B\rightarrow \Sp^{\Sigma}_H$ of $\Psi$ sends weak equivalence of simplicial $H$-sets to $H$-equivalences of symmetric $H$-spectra, for every subgroup $H$ of $G$.
\end{defn}

\begin{ex}
The following are examples of homotopy functors $G\mbox{-}\mathcal{S}_\ast\rightarrow \Sp^{\Sigma}_G$.
\begin{itemize}
\item For a fixed $G$-spectrum $E$ in $\Sp^{\Sigma}_G$, the functor $E\wedge|-|\colon G\mbox{-}\mathcal{S}_\ast\rightarrow \Sp^{\Sigma}_G$ that sends a finite pointed simplicial set $X$ to the spectrum $E\wedge|X|$ with diagonal action, where $|X|$ is the geometric realization of $X$.
\item For a fixed simplicial $G$-set $K$, the functor $\map_\ast(K,-)\wedge\mathbb{S}^G\colon G\mbox{-}\mathcal{S}_\ast\rightarrow \Sp^{\Sigma}_G$ that sends $X$ to the suspension spectrum of the pointed mapping space $\map_\ast(K,X)$ with conjugation action.
\end{itemize}
An example of a functor $G\mbox{-}\mathcal{S}_\ast\rightarrow \Sp^{\Sigma}_G$ that is not the extension of a functor $\mathcal{S}_\ast\rightarrow \Sp^{\Sigma}_G$ is the the functor that sends a based simplicial $G$-set $X$ to the suspension spectrum of the orbit space of $X$ with trivial $G$-action $(X/G)\wedge \mathbb{S}^G$.
\end{ex}

\begin{ex}\label{extendfctrs} Most of the enriched homotopy functors that we will encounter are induced from functors $Set_{\ast}^f\to \Sp^{\Sigma}_G$ from the category $Set_{\ast}^f$ of finite pointed sets. Extend a functor $F\colon Set_{\ast}^f\to \Sp^{\Sigma}_G$ to an enriched functor $F\colon \mathcal{S}_\ast\rightarrow \Sp^{\Sigma}_G$ as follows. Given a finite pointed simplicial set $X$, we denote $F(X_k)_n$ the $n$-th space of the symmetric $G$-spectrum $F(X_k)$. The simplicial structure on $X$ induces a simplicial space structure on $F(X_{\sbt})_n$ by functoriality of $F$, that respects the $(\Sigma_n\times G)$-action. Denote $F(X)_n$ its thick geometric realization. The structure maps of the spectra $F(X_k)$ induce the structure maps of a $G$-spectrum $F(X)$, defined by
\[F(X)_n\wedge S^{m\rho}\cong|F(X_{\sbt})_n\wedge S^{m\rho}|\stackrel{|\sigma_n|}{\xrightarrow{\hspace*{1cm}}}F(X)_{n+m}\]
The symmetric structure is defined in a similar way.
This defines the functor $\mathcal{S}_\ast\rightarrow \Sp^{\Sigma}_G$ on objects.
The components of the map $F\colon |\map_\ast(X,Y)_{\sbt}|{\to} \Sp_{G}^{\Sigma}\big(F(X),F(Y)\big)$
defining $F$ on morphism spaces are the composites
\[ |\map_\ast(X,Y)_{\sbt}|{\longrightarrow} |G\mbox{-}\map_\ast\big(F(X)_n,F(Y)_n\big)_{\sbt}|\stackrel{|-|}{\xrightarrow{\hspace*{1cm}}} G\mbox{-}\Top_\ast\big(F(X)_n,F(Y)_n\big)\]
where the second map takes a map of simplicial spaces to its geometric realization, and the first map is the geometric realization of the map of simplicial spaces \[F_n\colon\map_\ast(X,Y)_{\sbt}\longrightarrow G\mbox{-}\map_\ast\big(F(X)_n,F(Y)_n\big)_{\sbt}\] given in simplicial degree $k$ by sending $f_{\sbt}\colon X\times\Delta[k]\rightarrow Y$ to the simplicial $G$-map $F_{n}(f_{\sbt})\colon F(X)_n\times\Delta[k]\rightarrow F(Y)_n$ defined in degree $p$ by
\[F_{n}(f_{\sbt})\big(z\in F(X_p)_n,\sigma\in \Delta[k]_p\big)=F\big(f_p(-,\sigma)\big)(z)\]
This defines a functor $F\colon\mathcal{S}_\ast\rightarrow \Sp^{\Sigma}_G$ which is further extended to $F\colon G\mbox{-}\mathcal{S}_\ast\rightarrow \Sp^{\Sigma}_G$ as explained before Definition \ref{htpyfctr}.
\end{ex}

We end the section by discussing the assembly map of an enriched homotopy functor. The category $G\mbox{-}\mathcal{S}_B$ has a symmetric monoidal structure defined by an internal smash product. The smash product of two objects $(X,p_X,s_X)$ and $(Y,p_Y,s_Y)$ is the retractive space $(X\wedge_B Y,p,s)$ defined as the pushout of simplicial $G$-sets
\[\xymatrix{X\vee_BY\ar[r]\ar[d]&X\times_BY\ar[d]\\
B\ar[r]&X\wedge_B Y
}\]
with the obvious maps $p$ and $s$ to and from $B$.
Notice that the coproduct $X\vee_B Y$ is defined using the sections $s_X$ and $s_Y$.
An enriched functor $\Phi\colon G\mbox{-}\mathcal{S}_B\rightarrow \Sp^{\Sigma}_G$ has an associated assembly map 
\[A^{X}_K\colon \Phi(X)\wedge |K|\longrightarrow \Phi(X\wedge_B(K\times B))\]
where $(X,p,s)$ is an object of $G\mbox{-}\mathcal{S}_B$ and $K$ is a finite pointed simplicial $G$-set (the smash product of a spectrum with a space is levelwise). It is adjoint to the composite
\[|K|\longrightarrow| \map_B\big(X,X\wedge_B(K\times B)\big)_{\sbt}|\stackrel{\Phi}{\longrightarrow} \Sp^{\Sigma}\big(\Phi(X),\Phi(X\wedge_B(K\times B))\big)\]
where the first map is the realization of the adjoint of the identity on $X\wedge_B (K\times B)$. It sends a simplex $k$ of $K$ to the map that sends $x$ to the element of the smash product determined by  $(x,k,p(x))\in X\times_B K\times B$.

\begin{rem}\label{assandposquare}
Consider the particular case where $K=S^1=\Delta[1]/\partial$ is the simplicial circle with trivial $G$-action, and suppose that $\Phi(B)$ is a weakly contractible $G$-spectrum. The adjoint of the assembly map is a map
\[\widetilde{A}^{X}_{S^1}\colon\Phi(X)\longrightarrow \Omega\Phi(X\wedge_B(S^1\times B))\]
Let $C_{p}=X\wedge_B(B\times I)$ be the mapping cylinder of the projection $p\colon X\rightarrow B$.
The universal map from $\Phi(X)$ into the homotopy limit of the rest of the square
\[\xymatrix{\Phi(X)\ar[d]\ar[r]&\Phi(C_{p})\ar[d]\\
\Phi(C_{p})\ar[r]&\Phi(X\wedge_B(S^1\times B))
}\]
is the top horizontal map in the $G$-homotopy commutative diagram
\[\xymatrix{
\Phi(X)\ar[r]\ar[d]_{\widetilde{A}^{X}_{S^1}}&\holim\big(\Phi(C_{p})\rightarrow \Phi(X\wedge_B\!(S^1\times B))\leftarrow \Phi(C_{p})\big)\\
\Omega\Phi(X\wedge_B(S^1\times B))\ar@{=}[r] &\holim\big(\ast\rightarrow \Phi(X\wedge_B\!(S^1\times B))\leftarrow \ast\big)\ar[u]_{\simeq}
}\]
A $G$-homotopy between the two maps in induced by a $G$-contraction $C_p\simeq B$. Therefore the square above is homotopy cartesian precisely when the adjoint assembly map for the circle is a $G$-equivalence. Moreover iterating this homotopy limit construction gives a map corresponding to the assembly map $A^{X}_{S^n}\colon \Phi(X)\wedge S^n\rightarrow \Phi(X\wedge_B(S^n\times B))$ for the $n$-sphere.
\end{rem}

\begin{rem}
In his first calculus paper \cite{calcI} Goodwillie works in the category $\mathcal{S}/_B$ of spaces (or simplicial sets) over $B$. The objects of $\mathcal{S}/_B$ are maps $p\colon X\to B$ that do not necessarily admit a section. The category $\mathcal{S}/_B$ does not have a zero object, and its suspension functor is not adjoint to an internal $\hom$ object. We find it technically convenient to work with the category of retractive spaces $\mathcal{S}_B$ (and its equivariant analogue $G\mbox{-}\mathcal{S}_B$) which enjoys these extra categorical properties. A disadvantage will emerge in Theorem $B$, that will apply only to split-surjective maps. This restriction will however not affect the applications, as we are ultimately interested in the map $X\to \ast$ (see Corollary \ref{corA}).
\end{rem}


\section{Elements of equivariant calculus}
Let $G$ be a finite group, let $G\mbox{-}\mathcal{S}_B$ be the category of finite retractive simplicial $G$-sets over $B$ from \S\ref{secsetup}, and let $\Sp^{\Sigma}_G$ be the category of symmetric $G$-spectra introduced in \S\ref{spectrasec}. Both categories are enriched over the category of pointed $G$-spaces $G$-$\Top_\ast$. In this section we develop a theory of calculus for ($G$-$\Top_\ast$)-enriched homotopy functors $\Phi\colon G\mbox{-}\mathcal{S}_B\to \Sp^{\Sigma}_G$ (see Definition \ref{htpyfctr}) based on Blumberg's definition of equivariant linearity \cite[3.3]{Blumberg}. The main result is Theorem $B$ proved in Section \ref{secanalytic}, analogous to Goodwillie's Corollary \cite[5.4]{calcII} that ``functors with zero derivative are locally constant''.

\subsection{$G$-linear functors}
The following definition of $G$-linearity for functors $G\mbox{-}\mathcal{S}_B\rightarrow \Sp^{\Sigma}_G$ is analogous to Blumberg's definition of $G$-linearity for endofunctors of pointed $G$-spaces from \cite[3.3]{Blumberg}, in the case when the group $G$ is finite.

\begin{defn}\label{Glin}
A homotopy functor $\Phi\colon G\mbox{-}\mathcal{S}_B\rightarrow \Sp^{\Sigma}_G$ is $G$-linear if
\begin {enumerate}[label=\emph{\arabic*})]
\item It is reduced, that is $\Phi(B)$ is a weakly $G$-contractible spectrum,
\item It sends homotopy cocartesian squares to homotopy cartesian squares,
\item For every finite $G$-set $J$ and $(X,p,s)$ in $G\mbox{-}\mathcal{S}_B$, the canonical map
\[\Phi(X\wedge_B(J_+\times B))\longrightarrow \prod_{J}\Phi(X)\]
is a $G$-equivalence.
\end{enumerate}
\end{defn}
The above map is adjoint to the composite
\[J\rightarrow |\map_B\big({\bigvee_{B}}^JX,X\big)_{\sbt}|\cong|\map_B\big(X\wedge_B(J_+\times B),X\big)_{\sbt}| \stackrel{\Phi}{\longrightarrow} \Sp^{\Sigma}\big(\Phi(X\wedge_B(J_+\times B)),\Phi(X)\big) \]
where the first map sends $j$ to the projection
\[\pr_j(x,i)=\left\{\begin{array}{ll}
x&\mbox{ if $i=j$}\\
sp(x)&\mbox{ if $i\neq j$}
\end{array}\right.\]

\begin{ex} For every symmetric $G$-spectrum $E$, the functor $E\wedge|-|\colon G\mbox{-}\mathcal{S}_\ast\rightarrow \Sp^{\Sigma}_G$ is $G$-linear. The condition on squares holds by Remark \ref{fibhofib}, and because the smash product of well-pointed spaces preserves cofibrations and pushouts. The last condition is a consequence of the Wirthm\"{u}ller isomorphism theorem (see also \ref{wedgesintoprod}).
\end{ex}

\begin{ex}\label{DoldThom}
Given a topological Abelian group $M$ with additive $G$-action and a finite pointed set $X$, define a pointed $G$-space
\[M(X)=\big(\bigoplus_{x\in X} (M\cdot x)\big)/_{M\cdot\ast}\]
where $G$ acts diagonally on the $M$-components.
This construction is functorial in $X$ by sending a map of sets $f\colon X\to Y$ to the map $f_\ast\colon M(X)\to M(Y)$ defined by the formula
\[f_\ast(\{m_x\})_y=\sum_{x\in f^{-1}(y)}m_x\]
It extends levelwise to a $(G\mbox{-}\Top_\ast)$-enriched homotopy functor $M(-)\colon G\mbox{-}\mathcal{S}_\ast\rightarrow G\mbox{-}\Top_\ast$ by the construction of Example \ref{extendfctrs}. The functor $M(-)$ is the equivariant Dold-Thom construction, and its equivariant homotopy groups are Bredon homology.
Adapting the definition of $G$-linearity to $G$-space valued functors, we say that a reduced homotopy functor  $F\colon G\mbox{-}\mathcal{S}_\ast\rightarrow G\mbox{-}\Top_\ast$ is $G$-linear if it sends homotopy cocartesian squares to homotopy cartesian squares, and if the canonical map
\[F(\bigvee_J X)\longrightarrow \prod_JF(X)\]
is a weak equivalence of $G$-spaces (in the fixed points model structure) for every finite $G$-set $J$. We show in Appendix \ref{confGlin} that $M(-)\colon G\mbox{-}\mathcal{S}_\ast\rightarrow G\mbox{-}\Top_\ast$ is a $G$-linear homotopy functor.
\end{ex}

\begin{prop}\label{linearsmashspectrum}
Let $\Phi\colon G\mbox{-}\mathcal{S}_B\rightarrow  \Sp^{\Sigma}_G$ be a $G$-linear homotopy functor. The assembly map
\[A^{X}_{K}\colon \Phi(X)\wedge |K|\longrightarrow \Phi(X\wedge_B(K\times B))\]
is a $G$-equivalence for every object $K$ of $G$-$\mathcal{S}_\ast$. In particular for $X=S^0\times B$ there is a  $G$-equivalence
\[\Phi(S^0\times B)\wedge |K|\stackrel{\simeq}{\longrightarrow} \Phi(K\times B)\]
\end{prop}

\begin{proof}
The proof is by induction on the skeleton of $K$. The base inductive step is when $K=J_+$ is a finite pointed $G$-set. In this case there is a commutative diagram
\[\xymatrix{\Phi(X)\wedge J_+\ar[r]^-{A^{X}_{J_+}}\ar@{=}[d]&\Phi(X\wedge_B(J_+\times B))\ar[d]^\simeq\\
\bigvee_J\Phi(X)\ar[r]_{\simeq}&\prod_J\Phi(X)
}\]
where the bottom map is a $G$-equivalence by the Wirthm\"{u}ller isomorphism theorem (see also Proposition \ref{wedgesintoprod}).
Suppose inductively that the assembly map is an equivalence for the $(n-1)$-skeleton of $K$. The inclusion of the $(n-1)$-skeleton into the $n$-skeleton gives a cofiber sequence 
\[K^{(n-1)}\longrightarrow K^{(n)}\longrightarrow S^n\wedge J_+\]
for some finite $G$-set $J$, inducing a homotopy cocartesian square
\[\xymatrix{K^{(n-1)}\times B\ar[d]\ar[r]& K^{(n)}\times B\ar[d]\\
B\ar[r]&(S^n\wedge J_+)\times B
}\]
in $ G\mbox{-}\mathcal{S}_B$ (the horizontal maps are cofibrations with isomorphic cofibers). Since $\Phi$ is $G$-linear, the image of this square defines a fiber sequence
\[\Phi\big(X\wedge_B(K^{(n-1)}\times B)\big)\longrightarrow \Phi\big(X\wedge_B(K^{(n)}\times B)\big)\longrightarrow \Phi\big(X\wedge_B((S^n\wedge J_+)\times B)\big)\]
This sequence receives an assembly map from the (co)fiber sequence of symmetric $G$-spectra
\[\Phi(X)\wedge|K^{(n-1)}|\longrightarrow \Phi(X)\wedge |K^{(n)}|\longrightarrow \Phi(X)\wedge |S^n\wedge J_+|\]
In view of the long exact sequence in homotopy groups induced by these fiber sequences, it is enough by the inductive hypothesis to show that the assembly map for an indexed wedge of spheres $S^n\wedge J_+$ is an equivalence. This is the top horizontal map in the diagram
\[\xymatrix{ \Phi(X)\wedge |S^n\wedge J_+|\ar@{=}[d]\ar[r]&\Phi\big(X\wedge_B((S^n\wedge J_+)\times B )\big)\ar@{=}[d]\\
\bigvee_J\big( \Phi(X)\wedge |S^n|\big)\ar[r]\ar[d]^{\simeq}&
\Phi\big(X\wedge_B(S^n\times B)\wedge_B(J_+\times B)\big)\ar[d]^\simeq\\
\prod_J\big( \Phi(X)\wedge |S^n|\big)\ar[r]_-{\prod_J A}&\prod_J \Phi\big(X\wedge_B(S^n\times B)\big)
}\]
The bottom map $\prod_J A$ is the product of the assembly maps for $S^n$. It is an equivalence as $A$ fits in a commutative diagram
\[\xymatrix@C=30pt{ \Omega^n(\Phi(X)\wedge |S^n|)\ar[r]^-{\Omega^n A}& \Omega^n\Phi(X\wedge_B(S^n\times B))\\
\Phi(X)\ar[u]^\simeq\ar[ur]_\simeq
}\]
The diagonal map is adjoint to the assembly map, and it is a $G$-equivalence by linearity of $\Phi$ (see Remark \ref{assandposquare}).
\end{proof}
In particular when $B$ is the point, a $G$-linear functor $\Phi\colon G\mbox{-}\mathcal{S}_\ast\rightarrow  \Sp^{\Sigma}_G$ is determined by its value on the $0$-sphere.

\begin{cor}\label{corlinearsmashspectrum}
A $G$-linear homotopy functor $\Phi\colon G\mbox{-}\mathcal{S}_\ast\rightarrow  \Sp^{\Sigma}_G$ is naturally equivalent to the functor $\Phi(S^0)\wedge|-|\colon G\mbox{-}\mathcal{S}_\ast\rightarrow  \Sp^{\Sigma}_G$.
\end{cor}


\subsection{Stable $G$-excision and the $G$-differential}

We generalize the definitions of stable excision and of the differential from Goodwillie calculus (\cite{calcI} and \cite{calcII}) to our equivariant setting. We prove that the differential of a stably $G$-linear functor is $G$-linear, and we compute the connectivity of the $G$-linear approximation map. Here $G$ denotes as usual a finite group.

We remind the reader that an $n$-cube in a category $C$ is a functor $\chi\colon\mathcal{P}(\underline{n})\rightarrow C$, where $\underline{n}=\{1,\dots,n\}$ is the set with $n$-elements and $\mathcal{P}(\underline{n})$ is the poset category of subsets of $\underline{n}$ ordered by inclusion. When $C$ is the category $G\mbox{-}\mathcal{S}_B$, we say that $\chi$ is strongly homotopy cocartesian if all of its $2$-dimensional faces are homotopy cocartesian squares. When $C$ is the category of $G$-spectra $\Sp^{\Sigma}_G$, we say that 
$\chi$ is $\nu$-homotopy cartesian, with respect to a function $\nu\colon \{H\leq G\}\rightarrow \mathbb{Z}$, if the canonical map
\[\chi_{\emptyset}\longrightarrow\holim_{\mathcal{P}(\underline{n})\backslash{\emptyset}}\chi\]
is $\nu$-connected (see Definition \ref{spGconn}).
We denote $e_i\colon\chi_{\emptyset}\rightarrow\chi_{\{i\}}$ the initial maps of $\chi$, for all $i$ in $\underline{n}$.

Given a finite $G$-set $J$ and a $G$-map $p\colon E\rightarrow B$ of symmetric $G$-spectra we recall from \S\ref{spectrasec} that  $\widetilde{\prod\limits_{B}}^{J}E$ is defined as the homotopy pullback
\[\xymatrix@=15pt{\widetilde{\prod\limits_{B}}^{J}E\ar[r]\ar[d]&\prod_JE\ar[d]^{\prod_J p}\\
B\ar[r]_-{\Delta}&\prod_JB
}\]

\begin{defn}\label{defstablin}
Let $\Phi\colon G\mbox{-}\mathcal{S}_B\rightarrow \Sp^{\Sigma}_G$ be a homotopy functor, and let $\kappa,v\colon \{H\leq G\}\rightarrow\mathbb{Z}$ and $c\colon \{H\leq G\}\rightarrow\mathbb{Q}$ be functions which are invariant on conjugacy classes.
\begin{enumerate}[label=\emph{\arabic*})]
\item We say that $\Phi$ satisfies $E^{G}_n(c,\kappa)$ if it sends strongly homotopy cocartesian $(n+1)$-cubes $\chi\colon\mathcal{P}(\underline{n+1})\rightarrow \Sp^{\Sigma}_G$ with $\kappa(H)\leq\conn e_{i}^H$ to $\nu$-homotopy cartesian cubes, where $\nu$ is the function
\[\nu(H)=\sum_{i=1}^{n+1}\min_{K\leq H}\left(\conn e_{i}^K-c(K)\right)\]
\item We say that $\Phi$ satisfies $W(v,\kappa)$ if for every finite $G$-set $J$ and $(X,p,s)$ in $G\mbox{-}\mathcal{S}_B$ with connectivity  $\kappa(H)\leq\conn p^H$, the canonical map
\[\Phi(X\wedge_B(J_+\times B))\longrightarrow \widetilde{\prod\limits_{\Phi(B)}}^{J}\Phi(X)\]
is $\vartheta$-connected, where $\vartheta$ is the function
\[\vartheta(H)=\min\{2\conn p^H,\min_{K\lneqq H}\conn p^K\}-v(H)\]
\end{enumerate}
We say that $\Phi$ is stably $G$-excisive if it satisfies $E^{G}_1(c,\kappa)$ and $W(v,\kappa)$ for some functions  $c,\kappa,v$. We call $v$ the additivity function of $\Phi$. We say that $\Phi$ is stably $G$-linear if it is stably $G$-excisive and $\Phi(B)$ is weakly $G$-contractible.
\end{defn}

The conditions $E^{G}_n(c,\kappa)$ for $n\geq 2$ will play a role later in the paper, when we will consider $G$-analytic functors.

\begin{rem}\label{remstablin}\begin{enumerate}[label=\emph{\arabic*})]
\item If $G=\{1\}$ is the trivial group, condition $E^{\{1\}}_n(c,\kappa)$ is equivalent to the classical condition $E_n((n+1)c,\kappa)$ of \cite[1.8]{calcI}. The difference in the constants comes form the fact that in Goodwillie's definition $c$ appears outside the sum. In our context, it is going to be convenient to allow $c$ to vary with the subgroup. This is why $c$ can take rational values.
\item When $G=\{1\}$ is the trivial group, the property $W(v,\kappa)$ follows directly from $E^{\{1\}}_1(c,\kappa)$ (with $v=2c$). When $J$ has two elements, it follows from $E^{\{1\}}_1(c,\kappa)$ for the homotopy cocartesian square
\[\xymatrix@C=35pt{X\vee_BX\ar[r]^-{\id\vee_Bp}\ar[d]_{p\vee_B\id}&X\ar[d]^{p}\\
X\ar[r]_-{p}&B
}\]
The condition for lager $J$ can be proved inductively on the cardinality of $J$. This shows that stable $G$-excision for the trivial group agrees with Goodwillie's definition of stable excision from \cite{calcI} and \cite{calcII}.
\item The connectivity range in the condition $W(v,\kappa)$ is the same as the connectivity of the inclusion
\[{\bigvee\limits_{B}}^{J}X\longrightarrow \widetilde{\prod\limits_{B}}^{J}X\]
of spaces, in Lemma \ref{connwedgesitoprodtop}. Together with the Blakers-Massey theorem (cf. \cite[2.3]{calcII}) applied on fixed points, this shows that the forgetful functor $G\mbox{-}\mathcal{S}_B\rightarrow G\mbox{-}\mathcal{S}_\ast$ is stably $G$-excisive. 
\end{enumerate}
\end{rem}

Given a homotopy functor $\Phi\colon G\mbox{-}\mathcal{S}_B\rightarrow \Sp^{\Sigma}_G$, let
$\widetilde{\Phi}\colon G\mbox{-}\mathcal{S}_B\rightarrow \Sp^{\Sigma}_G$ be the associated reduced homotopy functor
\[\widetilde{\Phi}(X)=\hofib\big(\Phi(X)\stackrel{\Phi(p)}{\longrightarrow}\Phi(B)\big)\]
where $p$ is seen as a morphism $(X,p,s)\rightarrow (B,\id,\id)$ in $ G\mbox{-}\mathcal{S}_B$. On morphisms $\widetilde{\Phi}$ sends a map $f\colon (X,p_X,s_X)\rightarrow (Y,p_Y,s_Y)$ to the map induced on homotopy fibers
\[\xymatrix{\widetilde{\Phi}(X)\ar[r]\ar@{-->}[d]_{\widetilde{\Phi}(f)}&\Phi(X)\ar[d]^{\Phi(f)}\ar[r]^-{\Phi(p_X)}& \Phi(B)\ar@{=}[d]\\
\widetilde{\Phi}(Y)\ar[r]&\Phi(Y)\ar[r]_-{\Phi(p_Y)}& \Phi(B)
}\]

\begin{lemma}\label{reducedphi}
If $\Phi$ is stably $G$-excisive, $\widetilde{\Phi}$ is stably $G$-linear.
\end{lemma}

\begin{proof}
It is immediate to see that $E^{G}_1(c,\kappa)$ for $\Phi$ implies $E^{G}_1(c,\kappa)$ for $\widetilde{\Phi}$. The map $\widetilde{\Phi}(X\wedge_B(J_+\times B))\rightarrow \prod_J\widetilde{\Phi}(X)$ is the map of vertical homotopy fibers in the diagram
\[\xymatrix@R=13pt{
\widetilde{\Phi}(X\wedge_B(J_+\times B))\ar[d] \ar[r] &\holim\bigg(\ar@<-3ex>[d(.8)]\ast\ar[r]&\ast&\prod_J\widetilde{\Phi}(X)\ar[l] \bigg)
\\
\Phi(X\wedge_B(J_+\times B))\ar[d] \ar[r] &\holim\bigg(\Phi(B)\ar@<-3ex>[d(.8)]\ar[r]^-\Delta&\prod_J\Phi(B)&\prod_J\Phi(X)\ar[l] \bigg)
\\
\Phi(B) \ar[r]^-\simeq& \holim\bigg(\Phi(B)\ar[r]^-\Delta&\prod_J\Phi(B)&\prod_J\Phi(B)\ar@{=}[l]\bigg)
}
\]
The middle map is $\vartheta$-connected by condition $W(v,\kappa)$ for $\Phi$, and therefore so is the map on homotopy fibers.
\end{proof}

For a finite $G$-set 
$I$, let $S^{\mathbb{R}[I]}$ be the $I$-fold smash product of simplicial circles
\[S^{\mathbb{R}[I]}=S^1\wedge S^1\wedge\dots\wedge S^1 \]
where $G$ acts by permuting the smash components.
This is our simplicial model for the one point compactification of the permutation representation $\mathbb{R}[I]$ spanned by $I$. We denote its suspension in the category $G\mbox{-}\mathcal{S}_B$ by
\[S^{\mathbb{R}[I]}_BX=X\wedge_B(S^{\mathbb{R}[I]}\times B)\]
Let $nG$ be the disjoint union of $n$ copies of $G$ with diagonal action by left multiplication. The corresponding suspension is denoted
\[S^{\mathbb{R}[nG]}=S^{n\rho}\]

\begin{defn}\label{defdifferential}
The differential of a reduced enriched homotopy functor $\Phi\colon G\mbox{-}\mathcal{S}_B\rightarrow \Sp^{\Sigma}_G$  is the functor $D\Phi\colon G\mbox{-}\mathcal{S}_B\rightarrow \Sp^{\Sigma}_G$
defined by
\[D\Phi(X)=\hocolim_{n\in \mathbb{N}}\Omega^{n\rho}\Phi(S^{n\rho}_BX)\]
with structure maps adjoint to the assembly map
\[\Phi\big(S^{n\rho}_BX\big)\wedge |S^\rho|\longrightarrow \Phi\big(S^{n\rho}_BX\wedge_B( S^\rho\times B)\big)\cong \Phi\big(S^{(n+1)\rho}_BX\big)  \]
of \S\ref{secsetup}. If $\Phi$ is not reduced, we define $D\Phi=D\widetilde{\Phi}\colon G\mbox{-}\mathcal{S}_B\rightarrow \Sp^{\Sigma}_G$.
\end{defn}

\begin{prop}\label{derstablinislin}
The differential of a stably $G$-excisive functor is $G$-linear.
\end{prop}

\begin{proof} By Lemma \ref{reducedphi} we can assume that $\Phi$ is reduced. Moreover $D\Phi$ is obviously a reduced homotopy functor. Let $\chi$ be a homotopy cocartesian square in $ G$-$\mathcal{S}_B$ with initial maps $e_i\colon\chi_\emptyset\rightarrow \chi_i$ for $i=0,1$. Notice that if $e_i$ is $k_i$-connected and $H$ is a subgroup of $G$, the relative suspension $S^{\mathbb{R}[I]}_Be_i$ on $H$-fixed points is
\[k_i(H)+\conn (S^{\mathbb{R}[I]})^H+1\]
connected.
Moreover the $H$-fixed points space of the sphere $S^{\mathbb{R}[I]}$ is isomorphic to $S^{\mathbb{R}[I/H]}$,
which is $(|I/H|-1)$-connected. In particular $(S^{n\rho})^H$ is $(n|G/H|-1)$-connected. Thus by choosing $n$ sufficiently large the initial maps of $S^{n\rho}_B\chi$ become $\kappa_n$-connected, for
\[\kappa_n(H)=k_{i}(H)+n|G/H|\geq \kappa(H)\]
Condition $E^{G}_1(c,\kappa)$ ensures that the square $\Phi(S^{n\rho}_B\chi)$ is $\nu_n$-homotopy cartesian for
\[\begin{array}{lll}\nu_n(H)&=\sum\limits_{i=0,1}\min_{K\leq H}\big(k_{i}(K)+n|G/K|-c(K)\big)
\\&\geq 2n|G/H|+\sum\limits_{i=0,1}\min_{K\leq H}\big(k_{i}(K)-c(K)\big)\\
&\geq 2n|G/H|-2(\overline{c}(H)+1)
\end{array}\]
where $\overline{c}(H):=\max_{K\leq H}c(K)$. The $n$-th stage of the homotopy colimit defining $D\Phi(\chi)$ is then $\nu$-homotopy cartesian, for
\[\begin{array}{lll}\nu(H)&=\min_{K\leq H}\big(\nu_n(K)-\dim(S^{n\rho})^K\big)
\\&\geq \min_{K\leq H}\big(2n|G/K|-2(\overline{c}(K)+1)-n|G/K|\big)
\\&\geq n|G/H|-2\min_{K\leq H}\big((\overline{c}(K)+1\big)
\end{array}\]
This becomes arbitrarily large with $n$, and $D\Phi(\chi)$ is homotopy cartesian.

By a similar argument, we can choose $n$ sufficiently large so that the map
\[\Phi(S^{n\rho}_BX\wedge_B(J_+\times B))\longrightarrow \prod_J\Phi(S^{n\rho}_BX)\]
is $\vartheta_n$-connected, for
\[\begin{array}{lll}\vartheta_n(H)&\!\!\!\!\!=\min\big\{2(\conn p^H+n|G/H|),\min_{K\lneqq H}(\conn p^K+n|G/K|)\big\}-v(H)
\\&\!\!\!\!\!\geq\min\big\{2n|G/H|,\min_{K\lneqq H}n|G/K|\big\}-v(H)
\end{array}\]
Let us denote $\overline{v}=\max_{K\leq H}v(K)$.
The map $\Omega^{n\rho}\Phi\big(S^{n\rho}_BX\wedge_B(J_+\times B)\big)\to \prod_J\Omega^{n\rho}\Phi(S^{n\rho}_BX)$
is then $\nu_n$-connected, for
\[\begin{array}{llll}\nu_n(H)&=\min_{L\leq H}(\vartheta_n(L)-n|G/L|)\\
&\geq\min_{L\leq H}\big(\min\{2n|G/L|,\min_{K\lneqq L}n|G/K|\}-v(L)-n|G/L|\big)\\
&\geq\min_{L\leq H}\big(\min\{n|G/L|,\min_{K\lneqq L}n|G/K|-n|G/L|\}\big)-\overline{v}\\
&\geq\min_{L\leq H}\big(\min\{n|G/L|,n\}\big)-\overline{v}\geq n-\overline{v}
\end{array}\]
Since homotopy colimits preserve connectivity, the map $D\Phi(X\wedge_B(J_+\times B))\to \prod_JD\Phi(X)$
is $(n-\overline{v})$-connected, for every $n$.
\end{proof}

There is a canonical map $\Phi\to D\Phi$, as the functor $\Phi$ is the first term of the homotopy colimit sequence defining $D\Phi$. In order to estimate the connectivity of this ``approximation map'' $\Phi\to D\Phi$ we need an extra connectivity assumption on $\Phi$.

\begin{defn}\label{presconn}
Let $\Phi\colon G\mbox{-}\mathcal{S}_B\rightarrow \Sp^{\Sigma}_G$ be a homotopy functor.
 We say that $\Phi$ preserves connectivity above a function  $\kappa\colon \{H\leq G\}\rightarrow \mathbb{N}$  if there is a function $\lambda\colon \{H\leq G\}\rightarrow \mathbb{Z}$ such that
\[\conn_H\Phi(X)\geq \min_{K\leq H}\big(\conn p^K\big)+\lambda(H)\]
for every subgroup $H$ of $G$ and every $(X,p,s)$ in $G\mbox{-}\mathcal{S}_B$ with $\kappa(H)\leq \conn p^H$. We say that $\lambda$ is the connectivity function of $\Phi$.
\end{defn}

\begin{prop}\label{connlinapprox}
 Let  $\Phi\colon G\mbox{-}\mathcal{S}_B\rightarrow \Sp^{\Sigma}_G$ be a reduced homotopy functor that satisfies $E^{G}_1(c,\kappa)$ and that preserves connectivity. For every retractive space $(X,p,s)$ in $G\mbox{-}\mathcal{S}_B$ with $\kappa(H)\leq\conn p^H$ the canonical map $\Phi(X)\rightarrow D\Phi(X)$ is $\nu$-connected, for
\[\nu(H)=\min\left\{2\min_{K\leq H}\big(\conn p^K-c(K)\big),\min_{K\lneqq H}\big(\conn p^K+\lambda(K)\big)\right\}\]
where $\lambda$ is the connectivity function of $\Phi$.
\end{prop}

\begin{rem}\label{rightestimate}
\begin{enumerate}[label=\emph{\arabic*}),leftmargin=.5cm]
\item We are in fact going to prove that the map $\Phi(X)\rightarrow D\Phi(X)$ is $\nu_n$-connected for every sufficiently large $n$, where
\[\nu_n(H)=\min\left\{2\min_{K\leq H}\big(\conn p^K-c(K)\big),\min_{K\lneqq H}\big(\conn_K\Phi(S^{n\rho}_BX)-n|G/K|\big)\right\}\]
The estimate of Proposition \ref{connlinapprox} is obtained by using once more that $\Phi$ preserves connectivity, as
\[\begin{array}{l}\min_{K\lneqq H}\big(\conn_K\Phi(S^{n\rho}_BX)-n|G/K|\big)\geq\\
\min_{K\lneqq H}\big(
(\min_{L\leq K}\conn (S^{n\rho}_Bp)^L)+\lambda(K)-n|G/K|\big)=\\
\min_{K\lneqq H}\big(
\min_{L\leq K}(\conn p^L+n(|G/L|-|G/K|))+\lambda(K)\big)\geq\\
\min_{K\lneqq H}\big(
\min\{\conn p^K,\min_{L\lneqq K}(\conn p^L+n)\}+\lambda(K)\big)=\\
\min_{K\lneqq H}\big(
\conn p^K+\lambda(K)\big)
\end{array}
\]
The last equality holds for $n$ sufficiently large. In the proof of Theorem $B$ below this more refined estimate is going to be a key ingredient.

\item If $\Phi$ is $G$-linear, the map $\Phi\rightarrow D\Phi$ is an equivalence, even though the connectivity range of Proposition \ref{connlinapprox} is finite. For every $n$, there is a commutative diagram
\[\xymatrix{\Phi(X)\ar[dr]_{\simeq}\ar[r]&\Omega^{n\rho}\Phi(S^{n\rho}_BX)\\
&\Omega^{n\rho}(\Phi(X)\wedge S^{n\rho})\ar[u]^{\simeq }_{\ref{linearsmashspectrum}}
}\]
where the diagonal map is an equivalence as representation spheres are invertible in $G$-spectra. Hence $D\Phi$ is a homotopy colimit of equivalences, and the map from the initial object of the sequence $\Phi\to D\Phi$ is an equivalence. A similar statement for $G$-linear functors to $G$-spaces instead of $G$-spectra is proved in \cite[3.5]{Blumberg} and \cite[1.4]{Shima} using different methods (there is no analogue of Proposition \ref{linearsmashspectrum} if the target is not the category of $G$-spectra).

\end{enumerate}
\end{rem}

\begin{proof}[Proof of \ref{connlinapprox}.] We study the connectivity of $\Phi(X)\rightarrow D\Phi(X)$ in $\pi^{H}_\ast$, for a fixed subgroup $H\leq G$.
The map $\Phi(X)\rightarrow \Omega^{n\rho}\Phi(S^{n\rho}_BX)$ fits into a commutative diagram
\[\xymatrix{\Phi(X)\ar[r]\ar[d]& \Omega^{n\rho}\Phi(S^{n\rho}_BX)\ar[d]^{\res}\\
 \Omega^{n|G/H|}\Phi\big(S^{n|G/H|}_BX\big)\ar[r]_-{\iota} & \Omega^{n|G/H|}\Phi(S^{n\rho}_BX)
}\]
where both $\iota$ and $\res$ are induced by the inclusion $S^{n|G/H|}=(S^{n\rho})^H\rightarrow S^{n\rho}$ of the fixed points sphere.
The connectivity of the left vertical map in $\pi^{H}_\ast$ is 
\[2\min_{K\leq H}\left(\conn p^K-c(K)\right)\]
since $\Phi$ satisfies $E^{G}_1(c,\kappa)$ and $S^{n|G/H|}$ has trivial $H$-action.
The right vertical map fits into a fiber sequence
\[\map_\ast\big(S^{n\rho}/S^{n|G/H|},\Phi(S^{n\rho}_BX)\big)\longrightarrow\Omega^{n\rho}\Phi(S^{n\rho}_BX)\longrightarrow\Omega^{n|G/H|}\Phi(S^{n\rho}_BX)\]
of symmetric $G$-spectra (mapping spaces are taken levelwise). The connectivity of the mapping space in $\pi^{H}_\ast$, and hence of the restriction map, is at least
\[\min_{K\in cell(S^{n\rho}/S^{n|G/H|})}\left(\conn_K\Phi(S^{n\rho}_BX)-\dim S^{n|G/K|}/S^{n|G/H|}\right)
\]
where the minimum is taken over the collection of subgroups $K$ of $H$ with the property that the $H$-$CW$-complex $S^{n\rho}/S^{n|G/H|}$ contains a $K$-equivariant cell (one of the form $H/K\times D^k$). Since $S^{n|G/H|}$ is the $H$-fixed points of $S^{n\rho}$, the quotient cannot contain an $H$-equivariant cell (this would be a cell with trivial $H$-action), and the minimum above is greater than
\[\min_{K\lneqq H}\left(\conn_K\Phi(S^{n\rho}_BX)-\dim S^{n|G/K|}/S^{n|G/H|}\right)=\min_{K\lneqq H}\left(\conn_K\Phi(S^{n\rho}_BX)-n|G/K|\right)\]

Let us finally compute the connectivity of $\iota$. The cofiber of the inclusion  $S^{n|G/H|}\rightarrow S^{n\rho}$ is $G$-equivalent to $S^{n\rho}/S^{n|G/H|}$. The homotopy cocartesian square
\[\xymatrix{S^{n|G/H|}_BX\ar[d]\ar[r]&S^{n\rho}_BX\ar[d]\\
B\ar[r]&X\wedge_B\big(S^{n\rho}/S^{n|G/H|}\times B\big)
}\]
induces a sequence $\Phi\big(S^{n|G/H|}_BX\big)\to \Phi(S^{n\rho}_BX)\to \Phi\big(X\wedge_B(S^{n\rho}/S^{n|G/H|}\times B)\big)$. By stable linearity, this induces a long exact sequence in $\pi^{H}_\ast$ up to degree
\[\nu'(H)=2\min_{K\leq H}\left(\conn p^K-c(K)\right)+2n|G/H|\]
Looped down by $\Omega^{n|G/H|}$, it induces a long exact sequence in $\pi^{H}_\ast$ up to degree
\[\nu(H)=2\min_{K\leq H}\left(\conn p^K-c(K)\right)+n|G/H|\]
This range can be made arbitrarily large with $n$. Therefore we can choose $n$ sufficiently large so that $\iota$ is as connected as $\Omega^{n|G/H|}\Phi\big(X\wedge_B(S^{n\rho}/S^{n|G/H|}\times B)\big)$ in $\pi^{H}_\ast$. Since $\Phi$ preserves connectivity this is at least
$\min_{K\leq H}\big(\conn p^K+\conn S^{n|G/K|}/S^{n|G/H|}\big)+\lambda(H)-n|G/H|$ connected. The $K=H$ term of the minimum is infinite, and the connectivity becomes
\[\begin{array}{lll}\min_{K\lneqq H}\big(\conn p^K+n|G/K|-1\big)+\lambda(H)-n|G/H|\\
\geq\min_{K\lneqq H}\big(\conn p^K+n(|G/H|+1)-1\big)+\lambda(H)-n|G/H|\\
=\min_{K\leq H}\big(\conn p^K-1\big)+n+\lambda(H)
\end{array}\]
This diverges with $n$, and therefore we can choose $n$ sufficiently large so that $\iota$ does not contribute to the connectivity of $\Phi(X)\rightarrow D\Phi(X)$.
\end{proof}

We end the section by discussing differentials for non-relative functors.
Let $F\colon G\mbox{-}\mathcal{S}_\ast\rightarrow \Sp^{\Sigma}_G$ be a homotopy functor, and let $B$ in $G\mbox{-}\mathcal{S}_\ast$ be a pointed finite simplicial $G$-set. Define a homotopy functor $\widetilde{F}_B\colon G\mbox{-}\mathcal{S}_B\rightarrow \Sp^{\Sigma}_G$ by taking the homotopy fiber
\[\widetilde{F}_B(X)=\hofib\big(F(X)\stackrel{F(p)}{\longrightarrow}F(B)\big)\]
on objects, and by sending a morphism $f\colon (X,p_X,s_X)\rightarrow (Y,p_Y,s_Y)$ to the map induced on homotopy fibers
\[\xymatrix{\widetilde{F}_B(X)\ar[r]\ar@{-->}[d]_{\widetilde{F}_B(f)}&F(X)\ar[d]^{F(f)}\ar[r]^-{F(p_X)}& F(B)\ar@{=}[d]\\
\widetilde{F}_B(Y)\ar[r]&F(Y)\ar[r]_-{F(p_Y)}& F(B)
}\]

\begin{defn}\label{reldiff}
Let $F\colon G\mbox{-}\mathcal{S}_\ast\rightarrow \Sp^{\Sigma}_G$ be a homotopy functor. The differential of $F$ at $B$ in $G\mbox{-}\mathcal{S}_\ast$ is the differential of $\widetilde{F}_B$
\[D_BF:=D\widetilde{F}_B\colon G\mbox{-}\mathcal{S}_B\longrightarrow \Sp^{\Sigma}_G\]
\end{defn}

\begin{defn}\label{reladd}
We say that $F\colon G\mbox{-}\mathcal{S}_\ast\rightarrow \Sp^{\Sigma}_G$ is relatively additive if for every $B$ in $G\mbox{-}\mathcal{S}_\ast$ there are functions $v_B,\kappa_B\colon\{H\leq G\}\rightarrow\mathbb{Z}$ such that  for every $(X,p,s)$ in $G\mbox{-}\mathcal{S}_B$ with connectivity $\conn p^H\geq \kappa_B(H)$, the canonical map
\[F(X\wedge_B(J_+\times B))\longrightarrow \stackrel[F(B)]{}{\widetilde{\prod}}^{\!\!\!J} F(X)\]
is $\vartheta_B$-connected, for $\vartheta_B(H)=\min\{2\conn p^H,\min_{K\lneqq H}\conn p^K\}-v_B(H)$.
\end{defn}

\begin{rem}\label{linreladd}
A $G$-linear functor $F\colon G\mbox{-}\mathcal{S}_\ast\rightarrow \Sp^{\Sigma}_G$ is automatically relatively additive. Indeed, the cofiber sequence $B\rightarrow X\wedge_B(J_+\times B)\rightarrow \bigvee_{J}X/B$ induces a map of fiber sequences
\[\xymatrix{
F(B)\ar@{=}[d]\ar[r]& F(X\wedge_B(J_+\times B))\ar[d]\ar[r]& F(\bigvee_{J}X/B)\ar[d]^{\simeq}\\
F(B)\ar[r]&\stackrel[F(B)]{}{\widetilde{\prod}}^{\!\!\!J} F(X)\ar[r]&\prod_JF(X/B)
}\]
where the right vertical map is a $G$-equivalence by $G$-linearity. The right bottom map is
\[\holim\bigg(F(B)\rightarrow \prod_JF(B)\leftarrow\prod_JF(X)\bigg)\longrightarrow \holim\bigg(\ast\rightarrow \ast\leftarrow\prod_JF(X/B)\bigg)\]
whose homotopy fiber is indeed $\holim\bigg(F(B)\rightarrow\prod_JF(B)\leftarrow\prod_J F(B)\bigg)\simeq F(B)$.
\end{rem}

\begin{prop}\label{forgetstablylinear}
Suppose that  $F\colon G\mbox{-}\mathcal{S}_\ast\rightarrow \Sp^{\Sigma}_G$ is a stably $G$-linear and relatively additive homotopy functor. Then
 the functor $\widetilde{F}_B\colon G\mbox{-}\mathcal{S}_B\rightarrow \Sp^{\Sigma}_G$ is stably $G$-linear. In particular $D_BF$ is a $G$-linear functor.
\end{prop}

\begin{proof}
Consider the composite functor $G\mbox{-}\mathcal{S}_B\stackrel{U}{\longrightarrow} G\mbox{-}\mathcal{S}_\ast\stackrel{F}{\longrightarrow} \Sp^{\Sigma}_G$
where $U$ is the forgetful functor. As $\widetilde{F}_B=\widetilde{F\circ U}$, it is enough by Lemma \ref{reducedphi} to show that $F\circ U$ is stably $G$-excisive. The condition on squares for $F\circ U$ follows immediately from the condition on square for $F$, as $U$ preserves homotopy cocartesian squares and connectivity. It remains to show that $F\circ U$  satisfies $W(v_B,\kappa_B)$ for some functions $v_B,\kappa_B\colon\{H\leq G\}\rightarrow \mathbb{Z}$, but this is precisely the relative additivity condition for $F$.
\end{proof}

\begin{prop}\label{Fpresconn}
Let  $F\colon G\mbox{-}\mathcal{S}_\ast\rightarrow \Sp^{\Sigma}_G$ be a homotopy functor that satisfies $E_{1}^G(c,\kappa)$ and $W(v,\kappa)$ for some functions $c,v,\kappa\colon\{H\leq G\}\rightarrow \mathbb{Q}$.
If the reduced functor $\widetilde{F}_\ast\colon G\mbox{-}\mathcal{S}_\ast\rightarrow \Sp^{\Sigma}_G$ preserves connectivity above $\kappa$, so does $\widetilde{F}_B\colon G\mbox{-}\mathcal{S}_B\rightarrow \Sp^{\Sigma}_G$ for every $B$ in $G\mbox{-}\mathcal{S}_\ast$.
\end{prop}

\begin{proof}
For every $(X,p,s)$ in $G\mbox{-}\mathcal{S}_B$ there is a natural $G$-equivalence $\widetilde{F}_B(X)\simeq (\widetilde{\widetilde{F}_\ast})_B(X)$. It is induced by the diagram
\[\xymatrix@=17pt{(\widetilde{\widetilde{F}_\ast})_B(X)\ar[r]\ar[d]&\widetilde{F}_\ast(X)\ar[d]\ar[r]&\widetilde{F}_\ast(B)\ar[d]\\
\widetilde{F}_B(X)\ar[r]\ar[d]&F(X)\ar[d]\ar[r]^-{F(p)}&F(B)\ar[d]\\
\ast\ar[r]&F(\ast)\ar@{=}[r]&F(\ast)
}\]
All the rows and the columns are fiber sequences, and therefore the top left vertical map is a $G$-equivalence. It is then enough to show that $(\widetilde{\widetilde{F}_\ast})_B$ preserves connectivity above $\kappa$. Given a retractive space $(X,p,s)$ in $G\mbox{-}\mathcal{S}_B$ with $\kappa(H)\leq\conn p^H$
the sequence
\[\widetilde{F}_\ast(X)\stackrel{\widetilde{F}_\ast(p)}{\longrightarrow} \widetilde{F}_\ast(B)\longrightarrow \widetilde{F}_\ast(\hoc(p))\]
induces a long exact sequence in $\pi^{H}_\ast$ up to degree
\[\begin{array}{ll}\nu(H)&= \min_{K\leq H}\left(\conn X^K+1-c(K)\right)+\min_{K\leq H}\left(\conn p^K-c(K)\right)\\
&\geq\min_{K\leq H}\conn p^K-2\max_{K\leq H}c(K)
\end{array}\]
 by condition $E^{G}_1(c,\kappa)$. Since $ \widetilde{F}_\ast$ preserves connectivity there is a function $\lambda$ such that
\[\conn_H \widetilde{F}_\ast(\hoc(p))\geq\min_{K\leq H}\conn \hoc(p)^K+\lambda(H)\geq \min_{K\leq H}\conn{p}^K+\lambda(H)\]
By setting $\overline{c}(H)=\max_{K\leq H}c(K)$ we get
\[\begin{array}{ll}\conn_H(\widetilde{\widetilde{F}_\ast})_B(X)&=\conn_H( \widetilde{F}_\ast(X)\longrightarrow  \widetilde{F}_\ast(B))-1\\
 &\geq \min_{K\leq H}\conn{p}^K-1+\min\{\lambda(H),-2\overline{c}(H)\}
\end{array}\]
\end{proof}


\subsection{$G$-analytic functors}\label{secanalytic}

We generalize Goodwillie's notion of analytic functors from \cite{calcII} to the $G$-equivariant setting, for a finite group $G$. We prove Theorem $B$ below, showing that $G$-analytic functors with trivial differentials (in fact derivatives) send highly connected split-surjective maps to $G$-equivalences. This is the equivariant analogue of Corollary \cite[5.4]{calcII}: ``functors with trivial derivative are locally constant''.

\begin{defn}\label{Ganalytic}
Let $\Phi\colon G\mbox{-}\mathcal{S}_B\rightarrow \Sp^{\Sigma}_G$ be a homotopy functor and $\rho\colon \{H\leq G\}\rightarrow \mathbb{Z}$ a function which is invariant on conjugacy classes.
We say that $\Phi$ is $G$-$\rho$-analytic if there are functions $q,v\colon \{H\leq G\}\rightarrow \mathbb{Z}$ such that for every $n\geq 1$ the functor $\Phi$ satisfies conditions $W(v,\rho+1)$ and $E_{n}^{G}(\rho-\frac{q}{n+1},\rho+1)$ of Definition \ref{defstablin}.
\end{defn}

\begin{rem}
Let us point out the difference in the choice of constants from Goodwillie's definition of $\rho$-analytic functors of \cite[4.2]{calcII} to the present definition for the trivial group $G=\{1\}$.
Functors that are $\{1\}$-$\rho$-analytic in our sense are precisely the classical $\rho$-analytic functors, but with a different constant $q$. The comparison with Goodwillie's $E_{n}(c,\kappa)$ condition is
\[E_{n}^{\{1\}}\big(\rho-\frac{q}{n+1},\rho+1\big)=E_{n}\big((n+1)\rho-q,\rho+1\big)=E_{n}\big(n\rho-(q-\rho),\rho+1\big)\]
Let us also remind that $E_{1}^{\{1\}}\big(\rho-\frac{q}{2},\rho+1\big)$ implies $W(2\rho-q,\rho+1)$ for the trivial group (cf. Remark \ref{remstablin}).
\end{rem}

\begin{ex}\label{analyticmap}
Let $Z$ be a finite $G$-CW-complex and suppose that the dimension of the fixed points $Z^G$ is at least one. Then the functor \[\mathbb{S}^G\wedge \map(Z,|-|)\colon G\mbox{-}\mathcal{S}_\ast\longrightarrow \Sp^{\Sigma}_G\]
is $G$-$\rho$-analytic, where $\rho$ is the dimension function $\rho(H)=\dim Z^H$.
An argument completely analogous to the non-equivariant case of \cite[4.5]{calcII} shows that our functor satisfies $E_{n}^{G}(\rho,0)$. Let us show the property $W(2\rho,0)$. Given a finite pointed simplicial $G$-set $X$, there is a commutative diagram
\[\xymatrix{\mathbb{S}^G\wedge\map(Z,|\bigvee_JX|)\ar[rr]\ar[d]&&\prod_J\mathbb{S}^G\wedge\map(Z,|X|)\\
\mathbb{S}^G\wedge\map(Z,\prod_J|X|)\ar[r]_-{\cong}&\mathbb{S}^G\wedge\big(\prod_J\map(Z,|X|)\big)\ar[ur]&\mathbb{S}^G\wedge\big(\bigvee_J\map(Z,|X|)\big)\ar[l]\ar[u]_{\simeq}
}\]
We need to calculate the connectivity of the top horizontal map. Since smashing with $\mathbb{S}^G$ preserves connectivity, it is enough to calculate the connectivity of the maps
\[\bigvee_J\map(Z,|X|)\rightarrow \prod_J\map(Z,|X|)\ \ \ \ \ \mbox{and}\ \ \ \ \ \  \map(Z,\bigvee_J|X|)\rightarrow \map(Z,\prod_J|X|)\]
By Lemma \ref{connwedgesitoprodtop} the first map is $\nu$-connected, for
\[\begin{array}{llll}\nu(H)&=
\min\big\{2\conn \map(Z,|X|)^H,\min_{K\lneqq H}\conn \map(Z,|X|)^K\big\}\\
&\geq\min\big\{2\min_{K\leq H}\{\conn X^K-\dim Z^K\},\min_{K\lneqq H}\min_{L\leq K}\{\conn X^L-\dim Z^L\}\big\}\\
&=\min\{2(\conn X^H-\dim Z^H),\min_{K\lneqq H}\conn X^K-\dim Z^K\}\\
&\geq \min\{2\conn X^H,\min_{K\lneqq H}\conn X^K\}-2\dim Z^H
\end{array}
\]
The last inequality holds because $\dim Z^K\geq\dim Z^H$ for $K\leq H$. Similarly, the second map is $\vartheta$-connected, for
\[\begin{array}{llll}\vartheta(H)
&=\min_{K\leq H}\big\{\conn (\bigvee_JX\rightarrow \prod_JX)^K-\dim Z^K\big\}\\
&\geq \min_{K\leq H}\big\{\min\{2\conn X^K,\min_{L\lneqq K}\conn X^L
\}\big\}-\dim Z^H\\
&\geq \min\big\{2\conn X^H,\min_{K\lneqq H}\conn X^K
\big\}-\dim Z^H
\end{array}
\]
The top horizontal map is then as connected as the minimum of these two quantities, which is precisely the range of $W(2\rho,0)$.
\end{ex}

\begin{thmB} Let $G$ be a finite group, and let $F\colon G\mbox{-}\mathcal{S}_\ast\rightarrow\Sp^{\Sigma}_G$ be an enriched homotopy functor (see Definition \ref{htpyfctr}) satisfying the following conditions.
\begin{enumerate}[label=\emph{\arabic*)}]
\item $F$ is $G$-$\rho$-analytic for a function $\rho\colon\{H\leq G\}\rightarrow \mathbb{Z}$,
\item $\widetilde{F}_\ast$ preserves connectivity above $\rho+1$ (cf. Definition \ref{presconn}),
\item $F$ is relatively additive (cf. Definition \ref{reladd}),
\item the spectrum  $D_BF(B\vee S^0)$ is weakly $G$-contractible for every $B$ in $G\mbox{-}\mathcal{S}_\ast$ (cf. \ref{reldiff}).
\end{enumerate}
Then for every split-surjective map $f\colon X\rightarrow B$ of finite pointed simplicial $G$-sets satisfying $\rho(H)+1\leq\conn f^H$, the induced map
\[F(f)\colon F(X)\longrightarrow F(B)\]
 is a $\pi_\ast$-equivalence of symmetric $G$-spectra.
\end{thmB}
Theorem $B$ for the projection map $X\rightarrow\ast$ immediately gives the following.
\begin{cor}\label{corA}
Suppose that $F\colon G\mbox{-}\mathcal{S}_\ast\rightarrow\Sp^{\Sigma}_G$ satisfies the conditions of Theorem $B$, and that it is reduced. Then it sends every finite pointed simplicial $G$-set $X$ with $\rho(H)\leq\conn X^H$ to a weakly $G$-contractible spectrum.
\end{cor}

The proof of Theorem $B$ uses a lemma relating the differential of $F$ and its derivative. If $F$ is stably $G$-linear and relatively additive, the functor $D_BF(B\vee(-))\colon G\mbox{-}\mathcal{S}_\ast\rightarrow\Sp^{\Sigma}_G$ is $G$-linear for every $B$ in $G\mbox{-}\mathcal{S}_\ast$, and it is therefore determined by its value at $S^0$ (cf. Corollary \ref{corlinearsmashspectrum}). The spectrum $D_BF(B\vee S^0)$ is called the derivative of $F$ at $B$ in Goodwillie calculus (cf. \cite{calcI}) and sometimes denoted $\partial_BF$.

\begin{lemma}\label{inductionderarpoint}
Let $F\colon G\mbox{-}\mathcal{S}_\ast\rightarrow\Sp^{\Sigma}_G$ be a relatively additive $G$-$\rho$-analytic homotopy functor. If the spectrum  $D_BF(B\vee S^0)$ is weakly $G$-contractible, then  $D_BF(X)$ is  weakly $G$-contractible for every $(X,p,s)$ in $ G\mbox{-}\mathcal{S}_B$. 
\end{lemma}

\begin{proof}
We recall that under these hypotheses $D_BF=D\widetilde{F}_B$ is $G$-linear (see Proposition \ref{forgetstablylinear}).
We start by proving the lemma for equivariant spheres.
By hypothesis $D\widetilde{F}_B(B\vee S^0)$ is weakly $G$-contractible. By induction, $D\widetilde{F}_B(B\vee S^n)$ is also weakly $G$-contractible since by $G$-linearity of $D\widetilde{F}_B$ the diagram
\[\xymatrix{D\widetilde{F}_B(B\vee S^{n-1})\ar[r]\ar[d]&D\widetilde{F}_B(B\vee D^n)\simeq \ast\ar[d]\\
\ast\simeq D\widetilde{F}_B(B\vee D^n)\ar[r]&D\widetilde{F}_B(B\vee S^{n})
}\]
is homotopy (co)cartesian. 
If $X=B\vee(S^{n}\wedge J_+)$ for a finite $G$-set $J$, the map 
\[D\widetilde{F}_B(B\vee(S^{n}\wedge J_+))=D\widetilde{F}_B\big((B\vee S^n)\wedge_B(J_+\times B)\big)\longrightarrow \prod_JD\widetilde{F}_B(B\vee S^n)\]
is a  $G$-equivalence by $G$-linearity of $D\widetilde{F}_B$, with weakly $G$-contractible target. Therefore $D\widetilde{F}_B(B\vee(S^{n}\wedge J_+))$ is also weakly $G$-contractible.

Now that the lemma is proved for spheres, we prove that $D\widetilde{F}_B(X)$ is weakly $G$-contractible by induction on the relative $G$-skeleton of the pair $(X,B)$.
The base induction step is when $X=B\vee J_+$ for a finite $G$-set $J$, which is the $n=0$ case already proved above.
Now suppose inductively that the image of the $n$-skeleton is weakly $G$-contractible. By $G$-linearity of $D\widetilde{F}_B$ the sequence
\[D\widetilde{F}_B(X^{(n)})\longrightarrow D\widetilde{F}_B(X^{(n+1)}) \longrightarrow D\widetilde{F}_B\big(B\vee(S^{(n+1)}\wedge J_+)\big) \]
induced by the homotopy cocartesian square
\[\xymatrix{X^{(n)}\ar[r]\ar[d]& X^{(n+1)}\ar[d]\\
B\ar[r]&B\vee(S^{(n+1)}\wedge J_+)
}\]
is a fiber sequence with first and last terms weakly $G$-contractible.
\end{proof}

\begin{rem}
In the classical theory of calculus of functors one needs to require $F$ to satisfy the ``limit axiom'', in order to carry out the induction argument of Lemma \ref{inductionderarpoint}. Proving this axiom for explicit examples can require a considerable amount of work, see e.g. \cite[\S 2]{McCarthy}. Here we do not need to worry about this condition, since our theory considers only finite simplicial sets. The analogous equivariant condition would be that $F$ commutes with filtered homotopy colimits of pointed simplicial $G$-sets.
\end{rem}

\begin{proof}[Proof of Theorem $B$]
We generalize Goodwillie's proof of \cite[5.4]{calcII} to our equivariant setting.
As $F$ is $G$-$\rho$-analytic, it satisfies $W(v,\rho+1)$ and $E_{n}^{G}(\rho-\frac{q}{n+1},\rho+1)$ for certain functions $q,v\colon \{H\leq G\}\rightarrow \mathbb{Z}$ and for every $n\geq 1$.
We prove the following statement $\mathcal{I}(L)$ by induction on the size of the subgroups $L$ of $G$.
\begin{quote} $\mathcal{I}(L)$: For every $n\geq 1$ the functor $F$ satisfies $E_{n}^G(\rho-\frac{r}{n+1},\rho+1)$ for the function
\[r(H)=\left\{\begin{array}{ll}q(H)&\mbox{for }H\not\le L\\
\infty&\mbox{for }H\leq L
\end{array}\right.\]
Moreover $F(f)$ is an $L$-equivalence for any split-surjective $G$-map $f$ satisfying $\rho(H)+1\leq \conn f^H$ for subgroups $H\leq L$. 
\end{quote}
The statement $\mathcal{I}(G)$ contains in particular our theorem.

The base induction step $\mathcal{I}(\{1\})$ is essentially the  proof of Corollary \cite[5.4]{calcII}. One needs to make sure that all the constructions  of  \cite{calcII} carry a $G$-action, but the final statement is about a $\{1\}$-equivalence (the induction step below also goes through the proof of \cite{calcII} once more).

Now assume inductively that $\mathcal{I}(K)$ holds for subgroups $K$ of $G$ of size $|K|\leq l$, and let $L$ be a subgroup of $G$ with $l+1$ elements. For proper subgroups $K\lneqq L$ the value $r(K)$ is already infinite as $\mathcal{I}(K)$ is satisfied. Hence we need to improve the value of $r$ at the group $L$ itself. We do it inductively by proving that if $F$ satisfies  $E_{n}^G(\rho-\frac{r}{n+1},\rho+1)$ for all $n\geq 1$ and a function $r$ with $r(H)=\infty$ on groups $H\lneqq L$, then it satisfies $E_{n}^G(\rho-\frac{r'}{n+1},\rho+1)$ for all $n\geq 1$ with
\[r'(H)=\left\{\begin{array}{ll}r(H)&\mbox{for }H\neq L\\
r(H)+1&\mbox{for }H= L
\end{array}\right.\]
This will prove the first part of $\mathcal{I}(L)$.
Let $\chi\colon \mathcal{P}(\underline{n+1})\rightarrow  G\mbox{-}\mathcal{S}_\ast$ be a strongly cocartesian $(n+1)$-cube with initial maps $e_i$ satisfying $\rho(H)+1\leq\conn e_{i}^H$. Since $F$ is a homotopy functor we can assume that all the $e_i$'s are $G$-cofibrations.
We need to show that the cube $F(\chi)$ is at least
\[c(L):=\sum_{i=1}^{n+1}\min_{K\leq L}\big(\conn e_{i}^K-(\rho(K)-\frac{r'(K)}{n+1})\big)\]
cartesian in $\pi_{\ast}^L$. Since $r'(H)=r(H)=\infty$ for subgroups $H\lneqq L$ this is
\[c(L)=\sum_{i=1}^{n+1}\big(\conn e_{i}^L-(\rho(L)-\frac{r'(L)}{n+1})\big)=-(n+1)\rho(L)+r'(L)+\sum_{i=1}^{n+1}\conn e_{i}^L\]
Let $Z\colon\!\mathcal{P}(\underline{n\!+\!2})\rightarrow  G\mbox{-}\mathcal{S}_\ast$ be the strongly cocartesian $(n+2)$-cube with initial maps $e_1,e_1,e_2,\dots,e_{n+1}$, defined by taking iterated pushouts along $e_1$. The cube $Z$ defines a map of $(n+1)$-cubes $Z\colon\chi\rightarrow \overline{\chi}$ in the direction of the repeated map $e_1$. By \cite[1.6 i)]{calcII} the cube $F(\chi)$ is $c$-cartesian if both $F(Z)$ and $F(\overline{\chi})$ are. By assumption $F$ satisfies $E_{n+1}^G(\rho-\frac{r}{n+2},\rho+1)$, and therefore $F(Z)$ is $\nu$-cartesian, where the value of $\nu$ at $L$ is
\[\begin{array}{lll}\nu(L)&=\min_{K\leq L}\big(\conn e_{1}^K-(\rho(K)-\frac{r(K)}{n+2})\big)+\\
&+\sum_{i=1}^{n+1}\min_{K\leq L}\big(\conn e_{i}^K-(\rho(K)-\frac{r(K)}{n+2})\big)\\
&=\conn e_{1}^L-(\rho(L)-\frac{r(L)}{n+2})+\sum_{i=1}^{n+1}\big(\conn e_{i}^L-(\rho(L)-\frac{r(L)}{n+2})\big)\\
&=\sum_{i=1}^{n+1}(\conn e_{i}^L)+\conn e_{1}^L-(n+2)\rho(L)+r(L)\\
&\geq\sum_{i=1}^{n+1}(\conn e_{i}^L)-(n+1)\rho(L)+r(L)+1=c(L)
\end{array}\]
The inequality holds since  $\conn e_{1}^L\geq\rho(L)+1$. Therefore it is enough to show that $F(\overline{\chi})$ is $c$-cartesian. Since $\overline{\chi}$ is defined from $\chi$ by iterating pushouts, its initial maps $\overline{e_i}$ are $G$-cofibrations and they satisfy
\[\rho(H)+1\leq\conn e_{i}^H\leq\conn \overline{e}_{i}^H\]
Moreover $\overline{e}_1$ has a canonical $G$-equivariant retraction $p^{(1)}$ defined by the fold map, as illustrated by the following diagram for the case $n=1$:
\[\begin{tabular}{>{$}l<{$}}\chi\left\{\begin{tabular}{l}\\ \\ \\ \\ \\
\end{tabular}\right.
\\
\\
\\
\end{tabular}
\vcenter{\hbox{\xymatrix@=16pt{A\ \ \ \ar@{>->}[rr]^{e_1}
\ar[dd]_{e_2}
\ar@{>->}[dr]^*!/d1pt/{\labelstyle e_1}&&B\ \ \ar@<.5ex>@{>->}[dr]\ar@{-}[d]\\
&B\ \ \ar@{>->}[rr]_<<<<<<<<<<<<<<<{\overline{e}_{1}}\ar[dd]_<<<<<{\overline{e}_2}&\ar[d]&
B\amalg_AB\ar@<.2ex>@/_1.5pc/[ll]_>>>>*!/d1pt/{\labelstyle  p^{(1)}}\ar[dd]&\\
C\ar@{-}[r]\ar[dr]&\ar[r]&D\ar@<.5ex>[dr]\\
&D\ar[rr]&&D\amalg_CD\ar@<.2ex>@/_1.5pc/[ll]
}}}\hspace{-.7cm}
\begin{tabular}{>{$}l<{$}}\\
\\
\left.\begin{tabular}{l}\\ \\ \\ \\ \\
\end{tabular}\right\}\overline{\chi}
\end{tabular}
\]
The retraction is natural, and hence it defines a retraction for the map of $n$-cubes defined by $\overline{e}_1$.
Call the $(n+1)$-cube defined by the retractions $\mathcal{Y}^{(1)}$:
\[\xymatrix{B\ar[d]_-{\overline{e}_2}\ar[r]^-{\overline{e}_1}&B\amalg_A B\ar[r]^-{p^{(1)}}\ar[d]&B\ar[d]\\
D\ar[r]_{\underbrace{\phantom{aaaaaaaaa}}_{\mathlarger{\mathlarger{\overline{\chi}}}}}&D\amalg_CD
\ar[r]_{\underbrace{\phantom{aaaaaaaaa}}_{\mathlarger{\mathlarger{\mathcal{Y}^{(1)}}}}}&D
}
\]
The cube $\mathcal{Y}^{(1)}$ is strongly cocartesian by \cite[1.8 iv)]{calcII}.
Moreover $F(\overline{\chi})$ (and hence $F(\chi)$) is $c$-cartesian if we can prove that $F(\mathcal{Y}^{(1)})$ is $(c+1)$-cartesian. This is because as maps of $n$-cubes $\mathcal{Y}^{(1)}\circ \overline{\chi}=\id$ (see \cite[1.8 iii)]{calcII}). What we gain by replacing $F(\chi)$ with $F(\mathcal{Y}^{(1)})$ is that the connectivity of its first initial map $p^{(1)}$ is
\[\conn (p^{(1)})^H=\conn \overline{e}_{1}^H+1\geq \conn e_{1}^H+1\]
since $p^{(1)}$ is a retraction for $\overline{e}_{1}$.
For the other indices $i>1$, the connectivity of the $i$-th initial map of $F(\mathcal{Y}^{(1)})$ is greater or equal to the connectivity of $\overline{e}_i$, and hence greater or equal to the connectivity of $e_i$. Let us calculate how cartesian $F(\mathcal{Y}^{(1)})$ by exploiting that $p^{(1)}$ is more connected than $e_1$.
As a map of $n$-cubes, $F(\mathcal{Y}^{(1)})$ is pointwise as connected as the map $ F(p^{(1)})$.
Recall that if the map $F(p^{(1)})$ is $\nu$-connected, then $F(\mathcal{Y}^{(1)})$ is $(\nu-n)$-cartesian. Hence we need to determine the connectivity of $F(p^{(1)})$. For every proper subgroup $K\lneqq L$ the map $F(p^{(1)})$ is a $K$-equivalence by the property $\mathcal{I}(K)$, which is satisfied by the inductive hypothesis ($p^{(1)}$ is split surjective). We use the linear approximation to determine how connected $F(p^{(1)})$ is at the group $L$. Let us denote by $A$ and $B$ the source and the target of $p^{(1)}$ respectively, and we remark that $(A,p^{(1)},\overline{e}_1)$ defines an object of $G$-$\mathcal{S}_B$. Remark \ref{rightestimate} gives an estimate for the connectivity of the map
\[\hofib F(p^{(1)})=\widetilde{F}_B(A)\longrightarrow D\widetilde{F}_B(A)=D_BF(A)\]
and by Lemma \ref{inductionderarpoint} the target is weakly $G$-contractible. Hence the map $F(p^{(1)})$ is (by \ref{rightestimate})
\[\nu^{(1)}(L)= \min\left\{2\big(\conn (p^{(1)})^L-\frac{(\rho-r)(L)}{2}\big),\min_{K\lneqq H}\conn_K\widetilde{F}_B(S^{k\rho}_B p^{(1)})-k|G/K|\right\}\]
connected in $\pi_{\ast}^L$, for a sufficiently large choice of $k$. By the inductive hypothesis the map $F(S^{k\rho}_B p^{(1)})$ is a $K$-equivalence for subgroups $K\lneqq L$. The second term of the minimum above is then infinite, and $ F(p^{(1)})$ is
\[\nu^{(1)}(L)= 2\conn (p^{(1)})^L-(\rho-r)(L)=2\big(\conn \overline{e}_{1}^L+1\big)-(\rho-r)(L)\]
connected in $\pi^{L}_\ast$. This shows that $F(\mathcal{Y}^{(1)})$ is $(\nu^{(1)}-n)$-cartesian, where we define $\nu^{(1)}(K)=\infty$ on proper subgroups $K\lneqq L$. This is not necessarily larger than $c+1$, which is what we were trying to show. However, one can repeat this whole construction by replacing $\chi$ with $\mathcal{Y}^{(1)}$. If we keep iterating this procedure we obtain a sequence of $(n+1)$-cubes $\mathcal{Y}^{(m)}$ with the property that $F(\chi)$ is $c$-cartesian if $F(\mathcal{Y}^{(m)})$ is $(c+m)$-cartesian, and with $F(\mathcal{Y}^{(m)})$ at least $(\nu^{(m)}-n)$-cartesian for the function
\[\nu^{(m)}(K)=\left\{\begin{array}{ll} 2\big(\conn \overline{e}_{1}^K+m\big)-(\rho-r)(L)&\mbox{for }K=L\\
\infty&\mbox{for }K\lneqq L
\end{array}\right.\]
This quantity can be made bigger than $c(K)+m$ by choosing $m$ sufficiently large, and therefore $F(\chi)$ is $c$-cartesian. This proves the first part of the condition $\mathcal{I}(L)$.

It remains to show that $F(f)$ is an $L$-equivalence for a highly connected split-surjective map $f$. By the first part of $\mathcal{I}(L)$ the functor $F$ satisfies $E_{1}^G(\rho-\frac{r}{2},\rho+1)$ with $r$ infinite on all subgroups $H\leq L$. Since $D\widetilde{F}_B$ is weakly $G$-contractible, we find from Remark \ref{rightestimate} that $F(f)$ is $\nu$-connected, for
\[\nu(H)=\min_{K\lneqq H}\conn_K\widetilde{F}_B(S^{k\rho}_Bf)-k|G/K|
\]
on subgroups $H\leq L$. By the condition $\mathcal{I}(K)$, the map $\widetilde{F}_B(S^{k\rho}_Bf)$ is a $K$-equivalence, and therefore $\nu(H)$ is infinite for every  $H\leq L$. This shows that $F(f)$ is an $L$-equivalence, ending the proof.
\end{proof}


\section{$\mathbb{Z}/2$-equivariant calculus and Real algebraic $K$-theory}\label{analKR}

The goal of this section is to construct our main example of $\mathbb{Z}/2$-analytic functor from Real algebraic $K$-theory, and calculate its $\mathbb{Z}/2$-equivariant derivative.
We recall from \cite{Wall} that a Wall antistructure is a triple $(A,w,\epsilon)$ where $A$ is a ring, $w\colon A^{op}\rightarrow A$ is a ring map and $\epsilon\in A^\times$ is a unit, with the property that $w^2$ is conjugation by $ \epsilon$. The Real $K$-theory of $(A,w,\epsilon)$ is a symmetric $\mathbb{Z}/2$-spectrum $\KR(A,w,\epsilon)$ defined by Hesselholt and Madsen in \cite{IbLars}, with underlying spectrum equivalent to $\K(A)$. We recall its construction in details in Section \ref{KRsec} below.

\begin{defn}\label{bimodule}
A bimodule over a Wall antistructure $(A,w,\epsilon)$ is an $A$-bimodule $M$ together with an additive map $h\colon M\rightarrow M$ which satisfies the following conditions
\[\begin{array}{lll}
h(a\cdot m)=h(m)\cdot w(a)\\
h(m\cdot a)=w(a)\cdot h(m)\\
h^{2}(m)=\epsilon\cdot m\cdot\epsilon^{-1}
\end{array}\]
for every $a$ in $A$ and every $m$ in $M$.
\end{defn}

From a bimodule $(M,h)$ over $(A,w,\epsilon)$ we define a Wall antistructure $\big(A\ltimes M,w\ltimes h,(\epsilon,0)\big)$, where the semi-direct product ring $A\ltimes M$ has underlying Abelian group $A\oplus M$ and multiplication
\[(a,m)\cdot (b,n)=(a\cdot b,a\cdot n+m\cdot b)\]
and the map $w\ltimes h$ is the direct sum of $w$ and $h$.
Given a finite pointed set $X$, define a bimodule $M(X)=\big(\bigoplus_{x\in X}M\cdot x\big)/_{M\cdot\ast}$ over $(A,w,\epsilon)$ with involution $h(X)$ induced diagonally by $h\colon M\rightarrow M$. This gives a new antistructure $\big(A\ltimes M(X),w\ltimes h(X),(\epsilon,0)\big)$, and by letting $X$ vary a functor
\[\widetilde{\KR}\big(A\ltimes M(-)\big):=\hofib\Big(\KR\big(A\ltimes M(-),w\ltimes h(-),(\epsilon,0)\big)\longrightarrow \KR(A,w,\epsilon)\Big)\]
from pointed finite sets to $\Sp_{\mathbb{Z}/2}^{\Sigma}$. Extending this levelwise to pointed finite simplicial $\mathbb{Z}/2$-sets as in Example \ref{extendfctrs} we obtain a functor $\widetilde{\KR}\big(A\ltimes M(-)\big)\colon\mathbb{Z}/2\mbox{-}\mathcal{S}_\ast\rightarrow
\Sp^{\Sigma}_{\mathbb{Z}/2}$.
The aim of this section is to prove the following theorems. 

\begin{thmA}\label{KRanal}
For every pointed finite $\mathbb{Z}/2$-set $X$, the equivariant derivative $D_\ast\widetilde{\KR}\big(A\ltimes M(X)\big)$ is equivalent to the Real MacLane homology of $A$ with coefficients in the Dold-Thom construction $M(X\wedge S^{1,1})$, as defined in \ref{defMacLane}.
\end{thmA}

Here $S^{1,1}=\Delta[1]/\partial$ is the simplicial circle with levelwise involutions $\big(0\leq i_0\leq\dots\leq i_p\leq 1\big)\mapsto \big(0\leq 1-i_p\leq\dots\leq 1-i_0\leq 1\big)$
The simplicial set $M(S^{1,1})$ is isomorphic to the nerve of $M$, and the levelwise involution sends $(m_1,\dots,m_p)$ to $(h(m_p),\dots,h(m_1))$. We remark that this involution is not simplicial.

\begin{thmC}
The functor $\widetilde{\KR}\big(A\ltimes M(-)\big)\colon\mathbb{Z}/2\mbox{-}\mathcal{S}_\ast\rightarrow\Sp^{\Sigma}_{\mathbb{Z}/2}$ is a $\mathbb{Z}/2$-$\rho$-analytic (see Definition \ref{Ganalytic}) reduced enriched homotopy functor, where $\rho$ is the function
\[\rho(H)=\left\{\begin{array}{cc}-1&\mbox{ for }H=\{1\}\\
0&\mbox{ for }H=\mathbb{Z}/2
\end{array}\right.\]
Moreover it is relatively additive and it preserves connectivity, in the sense of Definitions \ref{reladd} and \ref{presconn} respectively.
\end{thmC}

The section is organized as follows. Section \ref{KRsec} is a recollection of constructions from \cite{IbLars}, containing in particular the definition of the Real algebraic $K$-theory spectrum of $(A,w,\epsilon)$. In Section \ref{MacLane} we define Real algebraic $K$-theory with coefficients and Real MacLane homology. In Section \ref{relKRanal} we prove that Real $K$-theory with coefficients defines a $\mathbb{Z}/2$-analytic functor (\ref{KRanalytic}), whose derivative is Real MacLane homology (\ref{Drel}). Finally in Section \ref{splitting} we finish the proofs of Theorems $A$ and $C$, by showing that $\widetilde{\KR}(A\ltimes M)$ is equivalent to the Real $K$-theory of $A$ with coefficients in $M(S^{1,1})$.

\subsection{Hesselholt and Madsen's Real algebraic $K$-theory functor}\label{KRsec}

This section is a recollection of constructions from \cite{IbLars}.
Let $\Delta R$ be the smallest subcategory of finite sets containing $\Delta$ and the maps $\omega_p\colon [p]\rightarrow [p]$ defined by
\[\omega_p(i)=p-i\]
for every $p\geq 0$.
Here we denoted $[p]=\{0,\dots,p\}$.
A pointed Real $n$-simplicial set is a functor $(\Delta R^{op})^{\times n}\rightarrow Set_\ast$. This is the same as a pointed $n$-simplicial set with levelwise involutions $w_p$ that ``reverse the order of the structure maps''. If $Z$ is a Real simplicial set, we let $|Z|$ be the geometric realization of the underlying simplicial set, defined as the quotient
\[|Z|=\big(\coprod_{p\geq 0}Z_p\times\Delta^p\big)/_\sim\]
where $\sim$ is the standard equivalence relation (see e.g. \cite[I-\S 2]{GJ}). The space $|Z|$ inherits a $\mathbb{Z}/2$-action from the Real structure, defined by the involution $|Z|\rightarrow |Z|$
\[[z\in Z_p;(t_0,\dots,t_p)\in\Delta^p]\longmapsto[w_p(z);(t_p,\dots,t_0)]\]

\begin{rem}\label{sub}
The levelwise involution on a Real simplicial set $Z$ induces a simplicial involution on Segal's edgewise subdivision $sd_e Z$ of the underlying simplicial set (see \cite{Segal}). The realization $|sd_e Z|$ inherits an involution, and the canonical homeomorphism $|Z|\cong |sd_e Z|$ is $\mathbb{Z}/2$-equivariant.
\end{rem}
The realization of a Real $n$-simplicial set $Z\colon (\Delta R^{op})^{\times n}\rightarrow Set_\ast$ is defined as the realization of the diagonal Real simplicial set
\[\delta(Z)\colon  \Delta R^{op}\stackrel{\delta}{\longrightarrow}(\Delta R^{op})^{\times n}\stackrel{Z}{\longrightarrow} Set_\ast\]
with the induced $\mathbb{Z}/2$-action.

\begin{ex}
Let $D\colon C^{op}\to C$ be a functor which satisfies $D^2=\id_C$. The nerve of $C$ equipped with the levelwise involutions $N_pD\colon N_pC\to N_pC$ is a Real simplicial set.
\end{ex}

The natural input of the Real $K$-theory functor is an exact category with duality, as defined by Schlichting in \cite{Sch}.

\begin{defn}[\cite{Sch}]\label{exactwithduality}
Let $C$ be an exact category. A duality on $C$ is an exact functor $D\colon C^{op}\rightarrow C$, together with a natural isomorphism $\eta\colon\id\Rightarrow D^2$ which satisfies $D(\eta_c)\circ\eta_{Dc}=\id_c$ for every object $c$ of $C$. If $\eta$ is the identity natural transformation we say that the duality is strict.
\end{defn}

We are particularly interested in the dualities on the category of finitely generated projective modules over a ring.

\begin{ex} Let $(A,w,\epsilon)$ be a Wall antistructure. We let $A_s$ be the right $A$-module with the same underlying Abelian group as $A$, and with right module structure $b\cdot a:=w(a)\epsilon\cdot b$. Given a right $A$-module $P$,  the Abelian group of module maps $\hom_A(P,A_s)$ has a right module structure defined by $(\lambda\cdot a)(p)=\lambda(p)\cdot a$. This defines a duality on the category $\mathcal{P}_{A}$ of finitely generated projective right $A$-modules
\[D=\hom_A(-,A_s)\colon \mathcal{P}_{A}^{op}\longrightarrow \mathcal{P}_{A}\]
The natural transformation $\eta_P\colon P\to D^2P$ is defined by the formula $(\eta_P(p))\big(\lambda\colon P\to A_s\big)=w(\lambda(p))\cdot \epsilon$.
\end{ex}

Let $C$ be an exact category with duality, and let us assume for the moment that the duality is strict. In \cite{IbLars} the authors define a simplicial category  $S^{2,1}_{\sbt}C$ similar to a two-fold Waldhausen $S_{\sbt}$-construction, in the following way. Let $Cat([2],[p])$ be the category of functors of posets $[2]\rightarrow [p]$. In each simplicial degree $p$, the category $S^{2,1}_{p}C$ is the full subcategory of the category of functors $X\colon Cat([2],[p])\rightarrow C$ which satisfy
\begin{itemize}
\item $X(\sigma)= 0$ if $\sigma\colon[2]\rightarrow [p]$ is not injective,
\item for every $\psi\colon[3]\rightarrow [p]$ the sequence
\[0\longrightarrow X(d_3\psi)\longrightarrow X(d_2\psi)\longrightarrow X(d_1\psi)\longrightarrow X(d_0\psi)\longrightarrow 0\]
is exact.
\end{itemize}
The category $S^{2,1}_{p}C$ is canonically an exact category (pointwise), and conjugation with the canonical strict duality on $Cat([2],[p])$  induces a strict duality $D\colon (S^{2,1}_pC)^{op}\rightarrow S^{2,1}_pC$. The construction can therefore be iterated to define a Real $n$-simplicial set $S^{(n)}(C)$ with $\underline{p}=(p_1,\dots,p_n)$-simplices
\[S^{(n)}(C)_{\underline{p}}:=\Ob (S^{2,1}_{\sbt})^{(n)}_{\underline{p}}C =\Ob S^{2,1}_{p_1}S^{2,1}_{p_2}\dots S^{2,1}_{p_n}C\]
Its geometric realization is the pointed $\mathbb{Z}/2$-space $|S^{(n)}(C)|=:\KR(C)_n$.
There are natural transformations $\chi_\ast\colon S^{(n)}(C)\rightarrow S^{(n)}(C)\circ\chi$ associated to a permutation $\chi$ in $\Sigma_n$ defined by permuting the entries of the functors (cf. \cite{IbLars}),  and a canonical isomorphism $S^{2,1}_2C=C$. These induce respectively a $\Sigma_n$-action on $\KR(C)_n$ and structure maps $\KR(C)_n\wedge S^{m\rho}\rightarrow \KR(C)_{n+m}$, where $\rho$ is the regular representation of $\mathbb{Z}/2$. The definitions of the $\Sigma_n$-actions and of the structure maps are analogous to the definitions for the $\mathbb{Z}/2$-spectrum $\widetilde{\KR}(S;N)$ defined in details in \S\ref{MacLane}. This structure combines into a symmetric $\mathbb{Z}/2$-spectrum $\KR(C)$.

\begin{defn}[\cite{IbLars}] The symmetric $\mathbb{Z}/2$-spectrum $\KR(C)$ is called the Real $K$-theory of the exact category with strict duality $C$.
\end{defn}

In case the duality on $C$ is not strict, there is a formal construction that replaces $C$ with an equivalent category $\mathcal{D}C$ which has a strict duality. The definition of $\mathcal{D}C$ first appeared in Vogell's thesis \cite{Vogell}, and it was later generalized in by Weiss and Williams in \cite{WW}. The objects of $\mathcal{D}C$ are the triples
\[\Ob\mathcal{D}C=\big\{\varphi=(c,d,\phi)\ |\ c,d\in \Ob C,\  \phi\colon d\stackrel{\cong}{\longrightarrow}Dc\big\}\]
A morphism $(c,d,\phi)\rightarrow (c',d',\phi')$ is a pair of morphisms $(a\colon c\to c',b\colon d'\to d)$ of $C$ which satisfy $\phi\circ b=D(a)\circ \phi'$.
 The functor $(\mathcal{D}C)^{op}\rightarrow \mathcal{D}C$ that sends $(c,d,\phi)$ to $(d,c,D(\phi)\circ\eta_c)$, and $(a,b)$ to $(b,a)$, is a strict duality on $\mathcal{D}C$. The projection functor $\mathcal{D}C\rightarrow C$ that projects onto the first component on both objects and morphisms is an equivalence of categories. 
If $C$ is an exact category, the category $\mathcal{D}C$ inherits an exact structure through the equivalence $\mathcal{D}C\rightarrow C$. In particular $\mathcal{D}\mathcal{P}_A$ is an exact category with strict duality, equivalent to $\mathcal{P}_A$.
\begin{defn}[\cite{IbLars}] The Real $K$-theory of a Wall antistructure $(A,\epsilon,w)$ is the $\mathbb{Z}/2$-symmetric spectrum
\[\KR(A,w,\epsilon)=\KR(\mathcal{D}\mathcal{P}_A)\]
\end{defn}
For completeness, we mention how $\KR(\mathcal{DC})$ relates to other constructions present in the literature. We will not use any of these result in the present paper, and we refer the proofs to \cite{IbLars}.
The underlying spectrum of $\KR(\mathcal{D}C)$ is equivalent to Waldhausen's algebraic $K$-theory $\K(C)$. The equivalence is induced by a simplicial functor $S^{2,1}_{\sbt}\mathcal{D}C\to S_{\sbt}S_{\sbt}\mathcal{D}C$ to the diagonal of the two-fold $S_{\sbt}$-construction. This gives an equivalence between the underlying spectrum of $\KR(\mathcal{D}C)$ and $\K(\mathcal{D}C)$, and the latter is equivalent to $\K(C)$ by the equivalence of categories $\mathcal{D}C\to C$.
Moreover the homotopy groups of the fixed points spectrum $\KR(\mathcal{D}C)^{\mathbb{Z}/2}$ are isomorphic to the Hermitian $K$-theory groups (or Grothendieck-Witt groups) of the category with duality $C$, as defined in \cite{Sch} and \cite{Horn}. This is because the Grothendieck-Witt space $\GW(C)$ of \cite{Sch} is weakly equivalent to the space of pointed equivariant maps from the representation sphere $S^{\rho}$ to $|\Ob S^{2,1}_{\sbt}\mathcal{D}C|$. A key ingredient in establishing this equivalence is the homotopy equivalence between the fixed points $|i\mathcal{D}C|^{\mathbb{Z}/2}$ of the subcategory of isomorphisms of $\mathcal{D}C$ and the realization of the category of symmetric spaces $|i\sym C|$ from \cite{Sch}.

\subsection{Real algebraic $K$-theory with coefficients and Real MacLane homology}\label{MacLane}

In this section we define and compare the Real $K$-theory and the Real MacLane homology of a coefficients system. These constructions will provide a rich supply of $\mathbb{Z}/2$-analytic functors. Their comparison is analogous to the relationship between $\widetilde{\K}(A;M)$ and $\THH(A;M)$ of \cite{DM}.

We say that a Real $n$-simplicial set $Z\colon (\Delta R^{op})^{\times n}\rightarrow Set$ is $1$-reduced if $Z_{\underline{p}}=\ast$ whenever at least one of the components of $\underline{p}=(p_1,\dots,p_n)$ satisfies $p_i\leq 1$.

\begin{defn} A Real $S$-construction is a collection of $1$-reduced pointed Real $n$-simplicial sets $ S^{(n)}\colon(\Delta R^{op})^{\times n}\rightarrow Set_\ast$, one for every integer $n\geq 0$, together with the following structure:
\begin{itemize}[leftmargin=0cm]
\item An isomorphism of Real $n$-simplicial sets $\chi_\ast\colon S^{(n)}\rightarrow S^{(n)}\circ\chi$ for every permutation $\chi$ in $\Sigma^n$, where $\chi\colon (\Delta R)^{\times n}\rightarrow (\Delta R)^{\times n}$ denotes the automorphism that permutes the product factors. For every $\chi$ and $\xi$ in $\Sigma_n$ we require that the diagram
\[\xymatrix{S^{(n)}\ar[r]^{\chi_\ast}\ar[dr]_{(\xi\circ\chi)_\ast}&S^{(n)}\circ\chi\ar[d]^{\xi_\ast|_{\chi}}\\
&S^{(n)}\circ\xi\circ\chi
}\]
commutes. Here $\xi_\ast|_{\chi}$ is the natural transformation $\xi_\ast$ restricted along the functor $\chi$.
\item Maps of Real $n$-simplicial sets 
\[\kappa\colon S^{(n)}\longrightarrow S^{(n+1)}\circ\iota\]
where $\iota\colon (\Delta R^{op})^{\times n}\rightarrow (\Delta R^{op})^{\times (n+1)}$ is the inclusion $\iota(p_1,\dots, p_n)=(p_1,\dots,p_n,2)$. We require that for every $(\chi,\xi)$ in $\Sigma_n\times\Sigma_m$ the diagram
\[\xymatrix{S^{(n)}\ar[d]_{\chi_\ast}\ar[r]^-{\kappa^{m}}&S^{(n+m)}\circ\iota^{m}\ar[d]^{(\chi\times\xi)_\ast|_{\iota^m}}\\
S^{(n)}\circ\chi\ar[r]_-{\kappa^m|_\chi}&S^{(n+m)}\circ\iota^{m}\circ\chi
}\]
commutes. Notice that $\iota^{m}\circ\chi=(\chi\times\xi)\circ\iota^{m}$ for any $\xi$ in $\Sigma_m$.
\end{itemize}
\end{defn}

The main example of a Real $S$-construction is defined from Hesselholt and Madsen's $S^{2,1}_{\sbt}$-construction. The extra generality of the previous definition will simplify the notation in later constructions.

\begin{ex}\label{exSconstr}
If $C$ is an exact category with strict duality, the collection of Real $n$-simplicial sets $S^{(n)}(C)=\Ob(S^{2,1}_{\sbt})^{(n)}C$ of Section \ref{KRsec} defines a Real $S$-construction $S(C)$. In case the duality on $C$ is not strict, we replace $C$ with the equivalent category with strict duality $\mathcal{D}C$ of Section \ref{KRsec}.
\end{ex}

We recall that the simplex category of an $n$-simplicial set $Z$ is the category $\Simp(Z)$ with objects
\[\Ob\Simp(Z)=\coprod_{\underline{p}\in\mathbb{N}^{n}}Z_{\underline{p}}\]
A morphism $(z,\underline{p})\rightarrow (y,\underline{q})$ is a morphism $\sigma\colon \underline{p}\rightarrow \underline{q}$ in $\Delta^{\times n}$ such that $\sigma^{\ast}y=z$. If $Z$ is a Real $n$-simplicial set, the simplex category $\Simp(Z)$ of the underlying $n$-simplicial set inherits an involution $\Simp(w)$ that sends $(z,\underline{p})$ to $(w(z),\underline{p})$ and $\sigma\colon (z,\underline{p})\rightarrow (y,\underline{q})$ to 
\[\overline{\sigma}\colon \underline{p}\stackrel{\omega}{\longrightarrow} \underline{p}\stackrel{\sigma}{\longrightarrow} \underline{q}\stackrel{\omega}{\longrightarrow} \underline{q}\]
where $\omega$ is the involution on $\underline{r}=[r_1]\times\dots\times [r_n]$ defined as the product of the involutions $\omega(i)=r_j-i$.

If $S$ is a Real $S$-construction, a permutation $\chi$ in $\Sigma_n$ induces an automorphism $\phi_\chi$ of $\Simp(S^{(n)})$
\[\phi_\chi\colon \Simp(S^{(n)})\stackrel{\Simp(\chi_\ast)}{\xrightarrow{\hspace*{1.2cm}}}\Simp(S^{(n)}\circ\chi)\cong \Simp(S^{(n)})\]
where $\Simp(\chi_\ast)$ is the functor induced by functoriality of $\Simp$ on the natural transformation $\chi_\ast\colon S^{(n)}\Rightarrow S^{(n)}\circ\chi$, and the second map is the canonical isomorphism $\Simp(S^{(n)}\circ\chi)\cong \Simp(S^{(n)})$ that reindexes the disjoint union summands.

\begin{defn}
A coefficients system for a Real $S$-construction is a family of Abelian group valued functors $N\colon \Simp(S^{(n)})^{op}\rightarrow Ab$ for every $n\geq 0$, sending the basepoint $\ast\in S^{(n)}_{\underline{p}}$ to the trivial Abelian group, together with:
\begin{itemize}[leftmargin=0cm]
\item  Natural transformations $w\colon N\rightarrow N\circ \Simp(w)^{op}$, where $\Simp(w)$ is the involution on $\Simp(S^{(n)})$ above. These have to satisfy $w^2=\id\colon N\rightarrow N\circ (\Simp(w)^{op})^2=N$.
\item  Natural transformations $\chi_\ast\colon N\rightarrow N\circ \phi_{\chi}^{op}$, such that for every $\chi$ and $\xi$ in $\Sigma_n$ the diagram
\[
\xymatrix{N\ar[r]^{\chi_\ast}\ar[d]_{(\xi\circ\chi)_\ast}&N\circ \phi_{\chi}^{op}\ar[d]^{\xi_\ast|_{\phi_{\chi}^{op}}}\\
N\circ \phi_{\xi\circ\chi}^{op}\ar@{=}[r]&N\circ \phi_{\xi}^{op}\circ\phi_{\chi}^{op}
}
\]
commutes.
\item As compatibility between this structure and $\kappa\colon S^{(n)}\rightarrow S^{(n+1)}\circ\iota$ we require that the diagrams
\[\xymatrix{\Simp(S^{(n)})^{op}\ar[r]^-N\ar[d]_{\Simp(\kappa)^{op}}& Ab\\
\Simp(S^{(n)}\circ\iota)^{op}\ar@{^(->}[r]&\Simp(S^{(n+1)})^{op}\ar[u]_N
}
\ \ \ \ \ \ \ \ 
\xymatrix{N\circ (\Simp(\kappa)^{op})^m\ar[d]_{\chi_\ast}\ar[r]^-{(\chi\times\xi)_\ast}&N\circ \phi_{\chi\times\xi}\circ \Simp((\kappa)^{op})^m\ar@{=}[d]\\
N\circ (\Simp(\kappa)^{op})^m\circ\phi_\chi\ar@{=}[r]&
\Simp((\chi\times\xi)\circ \kappa^m)^{op}
}
\]
commute for all  $(\chi,\xi)$ in $\Sigma_n\times\Sigma_m$, where the bottom horizontal map of the left diagram is the canonical inclusion.
\end{itemize}
\end{defn}

Let us see how a bimodule over an exact category with duality $C$ induces a coefficient system for the Real $S$-construction $S(C)$ of Example \ref{exSconstr}.

\begin{defn}\label{defbimoddual} A bimodule with duality on an exact category with duality $(C,D,\eta)$ (see Definition \ref{exactwithduality}) is an additive functor $M\colon C^{op}\otimes C\rightarrow Ab$ together with a natural isomorphism $J\colon M\Rightarrow M\circ D_\gamma$ making the diagram
\[\xymatrix{M(c,d)\ar[dr]_{(\eta_{c}^{-1}\otimes\eta_d)_\ast}\ar[r]^{J}&M(Dd,Dc)\ar[d]^{J}\\
&M(D^2c,D^2d)
}\]
commutative for every pair of objects $c$ and $d$ of $C$. If the duality on $C$ is strict, the natural transformation $J$ automatically satisfies $J^2=\id$.
\end{defn}

There is a strictification construction for the duality on $M$ as well.
A bimodule $M$ with duality $J$ on an exact category with duality $C$ induces a bimodule $\mathcal{D}M\colon (\mathcal{D}C)^{op}\otimes \mathcal{D}C\rightarrow Ab$ with duality $\mathcal{D}J$ on the category with strict duality $\mathcal{D}C$ of \S\ref{KRsec}. The bimodule is defined at a pair of objects $\varphi=(c,d,\phi)$ and $\varphi'=(c',d',\phi')$ of $\mathcal{D}C$ by the Abelian group $\mathcal{D}M(\varphi,\varphi')=M(c,c')$, and the duality is the natural transformation
\[\mathcal{D}J_{(\varphi,\varphi')}\colon\mathcal{D}M(\varphi,\varphi')=M(c,c')\stackrel{J_{c,c'}}{\xrightarrow{\hspace*{.7cm}}}M(Dc',Dc)\stackrel{(\varphi'\otimes\varphi^{-1})_\ast}{\xrightarrow{\hspace*{1.2cm}}}M(d',d)=\mathcal{D}M(D\varphi',D\varphi)\]

\begin{ex}\label{coeffsystbimod} Let $M\colon C^{op}\otimes C\rightarrow Ab$ be a bimodule with duality over an exact category with duality $C$. 
Upon making the suitable strictifications we can assume that the duality on $C$ is strict. We define a coefficients system $N_M$ for the Real $S$-construction $S^{(n)}(C):=\Ob (S^{2,1}_{\sbt})^{(n)}C$ of Example \ref{exSconstr}.
The bimodule $M$ extends canonically to a simplicial bimodule \[M_{\sbt}\colon (S_{\sbt}^{2,1}C)^{op}\otimes  S_{\sbt}^{2,1}C\longrightarrow Ab\]
on the simplicial category $S_{\sbt}^{2,1}C$, with a duality $J_{\sbt}\colon M_{\sbt}\Rightarrow M_{\sbt}\circ D_\gamma$. The Abelian group $M_{p}(X,Y)$ for diagrams $X$ and $Y$ in $ S_{\sbt}^{2,1}C$ is defined as the subgroup
\[M_{p}(X,Y)\leq\bigoplus_{\theta\in Cat([2], [p])}M(X_{\theta},Y_{\theta})\]
of collections $\{m_\theta\}$ with the property that for every natural transformation  $\psi\colon \rho\to\theta$ the relation
\[X(\psi)^{\ast}(m_\theta)=Y(\psi)_\ast (m_\rho)\]
holds in $M(X_\rho,Y_\theta)$.
The duality $J_{k}\colon M_k\Rightarrow M_k\circ D_{\gamma}$ is the restriction of $\bigoplus_{\theta}J$. Iterations of this construction give an $n$-simplicial bimodule
\[M_{\sbt}^{(n)}\colon ((S_{\sbt}^{2,1})^{(n)}C)^{op}\otimes  (S_{\sbt}^{2,1})^{(n)}C\longrightarrow Ab\]
for every $n\geq 0$.
This is analogous to the extension for the $S_{\sbt}$-construction of \cite[I-\S 3.3]{DGM}.

The functor $N_M\colon \Simp(S^{(n)}(C))^{op}\rightarrow Ab$ of the coefficients system is defined on objects by $N_M(s,\underline{p})=M_{\underline{p}}^{(n)}(s,s)$, and by sending a morphism $\sigma\colon \underline{q}\to\underline{p}$ from $\sigma^\ast s$ to $s$ in $\Simp(S^{(n)}(C))$ to the natural transformation of the simplicial bimodule structure
\[N_M(s,\underline{p})=M_{\underline{p}}^{(n)}(s, s)\stackrel{\sigma}{\longrightarrow}M_{\underline{q}}^{(n)}(\sigma^\ast s,\sigma^\ast s)=N_M(\sigma^\ast s,\underline{q})\]
The functor $N_M$ also inherits a natural transformation
\[w\colon N_M(s,\underline{p})=M_{\underline{p}}^{(n)}(s,s)\stackrel{J_{\underline{p}}}{\longrightarrow} M_{\underline{p}}^{(n)}(Ds,Ds)=N_M(Ds,\underline{p})\]
from the duality $J$ on $M$,
and natural transformations $\chi_\ast\colon N_M\rightarrow N_M\circ\phi_{\chi}^{op}$ defined by permuting the $S^{2,1}_{\sbt}$ factors in a similar way as we do for $S(C)$. This makes $N_M$ into a coefficients system for $S(C)$.
\end{ex}

We define a symmetric $\mathbb{Z}/2$-spectrum $\widetilde{\KR}(S;N)$ from a coefficients system $(S,N)$, as follows. There is a Real $n$-simplicial set $\widetilde{\KR}(S;N)^{(n)}$ defined in degree $\underline{p}$ by
\[\widetilde{\KR}(S;N)^{(n)}_{\underline{p}}=\bigvee_{s\in S^{(n)}_{\underline{p}}}N_s\]
The simplicial structure map associated to $\sigma\colon \underline{p}\rightarrow \underline{q}$ is the wedge of the maps $\sigma^\ast\colon N_s\rightarrow N_{\sigma^\ast s}$ given by functoriality of $N$. The Real structure is given by the wedge of the maps $w\colon N_s\rightarrow N_{ws}$. Define the $n$-th space of the spectrum $\widetilde{\KR}(S;N)$ as the geometric realization of the diagonal Real simplicial set $d(\widetilde{\KR}(S;N)^{(n)})$. There is a canonical $\mathbb{Z}/2$-equivariant homeomorphism
\[\widetilde{\KR}(S;N)_n=|d(\widetilde{\KR}(S;N)^{(n)})|\cong\big(\coprod_{\underline{p}\in\mathbb{N}^n}\widetilde{\KR}(S;N)^{(n)}_{\underline{p}}\times\Delta^{p_1}\times\dots\times\Delta^{p_n}\big)/_{\sim} \]
where the equivalence relation is the same as for the classical realization of an $n$-simplicial set, and the $\mathbb{Z}/2$-action on the right-hand side is diagonal, acting on $\Delta^p$ by reversing the order of the simplex coordinates.
The $\Sigma_n$-action on $\widetilde{\KR}(S;N)_n$ is defined at a permutation $\chi$ as the composite
\[\xymatrix{\widetilde{\KR}(S;N)_n\ar[dd]_{\chi}\ar[rr]^-{\cong}&& \big(\coprod_{\underline{p}\in\mathbb{N}^n}\widetilde{\KR}(S;N)^{(n)}_{\underline{p}}\times\Delta^{p_1}\times\dots\times\Delta^{p_n}\big)/_{\sim}
\ar[d]^{\coprod\chi_\ast\times\chi}\\
&&\big(\coprod_{\underline{p}\in\mathbb{N}^n}\widetilde{\KR}(S;N)^{(n)}_{\chi(\underline{p})}\times\Delta^{p_{\chi(1)}}\times\dots\times\Delta^{p_{\chi(n)}}\big)/_{\sim}\\
*[r]{\!\!\!\!\!\!\!\!\!\!\!\!\widetilde{\KR}(S;N)_n=|d(\widetilde{\KR}(S;N)^{(n)})|}\ar@{=}[rr]&&|d(\widetilde{\KR}(S;N)^{(n)}\circ\chi)|\ar[u]_{\cong}
}\]
where the map $\chi_\ast\colon \widetilde{\KR}(S;N)^{(n)}\rightarrow \widetilde{\KR}(S;N)^{(n)}\circ\chi$ is the wedge of the natural maps $\chi_\ast\colon N_s\rightarrow N_{\phi_\chi(s)}$. Let us define the structure maps of the spectrum $\widetilde{\KR}(S;N)$.
Let $\Delta^{2,1}$ be the topological $2$-simplex $\Delta^{2}$ with the $\mathbb{Z}/2$-action that reverses the order of the coordinates. There is a canonical homeomorphism between $S^{2,1}=\Delta^{2,1}/\partial$ and the regular representation sphere $S^{\rho}$ of $\mathbb{Z}/2$. 
The structure maps of the spectrum are induced by the composite
\[\xymatrix{\widetilde{\KR}(S;N)_n\times (\Delta^{2,1})^{\times m}\ar[r]^-{\cong}\ar[ddd]&  |d(\widetilde{\KR}(S;N)^{(n)}\times(\Delta^{2,1})^{\times m})|\ar[d]^\cong\\
 &\big(\coprod_{\underline{p}\in\mathbb{N}^n}\widetilde{\KR}(S;N)^{(n)}_{\underline{p}}\times\Delta^{p_1}\times\dots\times\Delta^{p_n}\times(\Delta^{2,1})^{\times m} \big)/_{\sim}\ar[d]^{\kappa^m}\\
 & \big(\coprod_{\underline{p}\in\mathbb{N}^n}\widetilde{\KR}(S;N)^{(n+m)}_{\underline{p},2,\dots,2}\times\Delta^{p_1}\times\dots\times\Delta^{p_n}\times(\Delta^{2,1})^{\times m}\big)/_{\sim}\ar[d]\\
\widetilde{\KR}(S;N)_{n+m}\ar@{=}[r]&\big(\coprod_{\underline{q}\in\mathbb{N}^{n+m}}\widetilde{\KR}(S;N)^{(n+m)}_{\underline{q}}\times\Delta^{q_1}\times\dots\times\Delta^{q_{n+m}}\big)/_\sim
}\]
where the bottom right vertical map is the inclusion of the $\underline{p}$-component into the $\underline{q}=(\underline{p},2\dots,2)$-component. By assumption $\widetilde{\KR}(S;N)^{(n)}$ is $1$-reduced, and therefore this map descends to a map
\[\sigma_{n,m}\colon \widetilde{\KR}(S;N)_{n}\wedge S^{m\rho}\longrightarrow \widetilde{\KR}(S;N)_{n+m}\]
on the quotient $S^{m\rho}=(S^{2,1})^{\wedge m}=(\Delta^{2,1}/\partial)^{\wedge m}$.

\begin{defn}\label{defrelkr}
The $(\Sigma_n\times\mathbb{Z}/2)$-spaces $\widetilde{\KR}(S;N)_n$ together with the maps $\sigma_{n,m}$ define a symmetric $\mathbb{Z}/2$-spectrum $\widetilde{\KR}(S;N)$, called the Real $K$-theory spectrum of the coefficients system $(S,N)$.
The Real $K$-theory of an exact category with duality $C$ with coefficients in a bimodule with duality $M\colon C^{op}\otimes C\rightarrow Ab$ is the Real $K$-theory
\[\widetilde{\KR}(C;M)=\widetilde{\KR}(S(\mathcal{D}C);N_{\mathcal{D}M})\]
of the coefficients system $(S(\mathcal{D}C),N_{\mathcal{D}M})$ of Example \ref{coeffsystbimod}.
\end{defn}

The construction of the spectrum $\widetilde{\KR}(S;N)$ can be repeated verbatim with wedges replaced by direct sums of Abelian groups. This leads to our definition of Real MacLane homology.

\begin{defn}\label{defMacLane}
The Real MacLane homology of a coefficients system $(S,N)$ is the symmetric $\mathbb{Z}/2$-spectrum $\HR(S;N)$ defined by the geometric realization of the Real simplicial sets
\[\HR(S;N)^{(n)}_{\underline{p}}=\bigoplus_{s\in S^{(n)}_{\underline{p}}}N_s\]
The Real MacLane homology $\HR(C;M)$ of an exact category with duality $C$ with coefficients in a bimodule with duality $M\colon C^{op}\otimes C\rightarrow Ab$ is the Real MacLane homology of the coefficients system $(S(\mathcal{D}C),N_{\mathcal{D}M})$ of Example \ref{coeffsystbimod}.
\end{defn}

This definition is analogous to the model for $\THH$ used in \cite[3.1]{DM}, which lies in between the standard definitions of MacLane homology and of topological Hochschild homology. These three theories are all equivalent (non-equivariantly), as proved in \cite{FPSVW} and \cite{DM}. The author's thesis \cite{thesis} contains a theory of Real topological Hochschild homology, which is equivalent to the Real MacLane homology of Definition \ref{defMacLane}, at least when $2$ is invertible (see  \cite[4.12.2]{thesis}).

The inclusion of wedges into direct sums induces an equivariant map of symmetric $\mathbb{Z}/2$-spectra
\[\widetilde{\KR}(S;N)\longrightarrow\HR(S;N)\]
analogous to the trace map of \cite{DM}. Section \ref{relKRanal} studies the $\mathbb{Z}/2$-analytic properties of this map.

\subsection{Analytic properties of Real algebraic $K$-theory with coefficients}\label{relKRanal}

Given a coefficients system $(S,N)$ and a pointed set $X$, one can replace the coefficients functor $N\colon \Simp(S^{(n)})^{op}\rightarrow Ab$ with the functor $N(X)\colon \Simp(S^{(n)})^{op}\rightarrow Ab$ that sends $s$ in $S^{(n)}_{\underline{p}}$ to the Abelian group 
\[N_s(X)=\big(\bigoplus_{x\in X}N_s\cdot x\big)/_{N_s\cdot\ast}\]
of Example \ref{DoldThom}.
This gives a new coefficients system $(S,N(X))$, with an associated Real $K$-theory spectrum $\widetilde{\KR}\big(S;N(X)\big)$. This construction is functorial in the set $X$, and we extend it degreewise to an enriched functor
\[\widetilde{\KR}\big(S;N(-)\big)\colon\mathbb{Z}/2\mbox{-}\mathcal{S}_\ast\longrightarrow\Sp^{\Sigma}_{\mathbb{Z}/2}\]
as explained in Example \ref{extendfctrs}.
As the group $\mathbb{Z}/2$ has only two subgroups, we will denote a function $\nu\colon\{H\leq \mathbb{Z}/2\}\rightarrow \mathbb{Q}$ by listing its values, starting with the trivial subgroup: $\nu=(\nu(1),\nu(\mathbb{Z}/2))$.

\begin{theorem}\label{KRanalytic}
The functor $\widetilde{\KR}\big(S;N(-)\big)\colon \mathbb{Z}/2\mbox{-}\mathcal{S}_\ast\to\Sp^{\Sigma}_{\mathbb{Z}/2}$ enjoys the following properties:
\begin{enumerate}[label=\emph{\arabic*})]
\item It is a reduced homotopy functor,
\item it is $\mathbb{Z}/2$-$\rho_0$-analytic (cf. Definition \ref{Ganalytic}) where $\rho_0$ is the zero function $\rho_0=(0,0)$,
\item it is relatively additive (cf. Definition \ref{reladd}),
\item it preserves connectivity (cf. Definition \ref{presconn}).
\end{enumerate}
\end{theorem}

\begin{proof}[Proof of \ref{KRanalytic}] The functor
$\widetilde{\KR}(S;N(-))$ is clearly reduced. Since the Dold-Thom construction $N_s(-)\colon\mathbb{Z}/2\mbox{-}\mathcal{S}_\ast\rightarrow{\mathbb{Z}/2}\mbox{-}\Top_\ast $ is a homotopy functor for every $s$ in $S^{(n)}$ (see Appendix \ref{confGlin}), it follows that $\widetilde{\KR}(S;N(-))$ sends ${\mathbb{Z}/2}$-equivalences to levelwise ${\mathbb{Z}/2}$-equivalences of spectra, and it is therefore a homotopy functor.

Let us show that $\widetilde{\KR}(S;N(-))$ satisfies $E_{k}^{\mathbb{Z}/2}(0,-1)$.
Let  $\chi\colon\mathcal{P}(\underline{k+1})\rightarrow \mathbb{Z}/2\mbox{-}\mathcal{S}_\ast$ be a strongly cocartesian  $(k+1)$-cube. We need to prove that $\widetilde{\KR}(S;N(\chi))$ is $\nu=(\nu_1,\nu_2)$-cartesian, for
\[\nu=(\sum_{1\leq j\leq k+1}\conn e_j,\sum_{1\leq j\leq k+1}\min\{\conn e_j,\conn e^{{\mathbb{Z}/2}}_j\})\]
We adapt McCarthy's argument from \cite{McCarthy}.
For every fixed $n$ there is a natural equivariant homeomorphism
\[\widetilde{\KR}(S;N(\chi))_{n}^{\mathbb{Z}/2}\cong|\bigvee_{s\in (sd_e S^{(n)})^{\mathbb{Z}/2}}N_{s}(\chi)^{\mathbb{Z}/2}|\]
where $sd_e$ is the edgewise subdivision functor (see \ref{sub}).
Notice that for $s$ in $(sd_e S^{(n)})^{\mathbb{Z}/2}$ the natural transformation map $w\colon N\to N\circ \Simp(w)^{op}$ defines indeed an involution on $N_s$. By the Appendix \ref{confGlin} the squares $N_{s}(\chi)^{\mathbb{Z}/2}$ are strongly cartesian, and its final maps $N_s(f_j\colon \chi_{\underline{k+1}\backslash\{j\}}\rightarrow \chi_{\underline{k+1}})^{\mathbb{Z}/2}$ have at least connectivity
\[\conn N_s(f_j)^{\mathbb{Z}/2}\geq \min\{\conn f_j,\conn f_{j}^{\mathbb{Z}/2}\}\geq\min\{\conn e_j,\conn e_{j}^{\mathbb{Z}/2}\}\]
The second inequality holds because $\chi$ is strongly cocartesian, and the first one because $N_s(-)$ preserves connectivity.
By the dual Blakers-Massey theorem (for spaces) the cube $N_{s}(\chi)^{\mathbb{Z}/2}$ is 
\[c_2:=k+\sum_{1\leq j\leq k+1}\min\{\conn e_j,\conn e_{j}^{\mathbb{Z}/2}\}\]
cocartesian. Since wedges commute with homotopy cofibers, the homotopy cofiber $(C_{\underline{p}}^{(n)})^{\mathbb{Z}/2}$ of the map
\[\hocolim_{\mathcal{P}(\underline{k+1})\backslash\underline{k+1}}\bigvee_{s\in (sd_e S^{(n)})_{\underline{p}}^{\mathbb{Z}/2}}N_{s}(\chi)^{\mathbb{Z}/2}\longrightarrow \bigvee_{s\in (sd_e S^{(n)})_{\underline{p}}^{\mathbb{Z}/2}}N_{s}(\chi_{\underline{k+1}})^{\mathbb{Z}/2}\]
is $c_2$-connected for every $\underline{p}$ in $\mathbb{N}^n$. Since $S^{(n)}$ is $1$-reduced, the edgewise subdivision $(sd_e S^{(n)})^{\mathbb{Z}/2}$ is $0$-reduced in each of its $n$ simplicial directions. Therefore so is  $(C^{(n)})^{\mathbb{Z}/2}$, and its geometric realization is $(c_2+n)$-connected. Since homotopy cofibers commute with realizations, there is a cofiber sequence of spaces
\[\hocolim_{\mathcal{P}(\underline{k+1})\backslash\underline{k+1}}\widetilde{\KR}(S;N(\chi))_{n}^{\mathbb{Z}/2}\longrightarrow \widetilde{\KR}(S;N(\chi_{\underline{k+1}}))_{n}^{\mathbb{Z}/2}\longrightarrow |(C^{(n)})^{\mathbb{Z}/2}|\]
with  $(c_2+n)$-connected cofiber. This shows that the cofiber $C$ of the map of spectra
\[\hocolim_{\mathcal{P}(\underline{k+1})\backslash\underline{k+1}}\widetilde{\KR}(S;N(\chi))\longrightarrow \widetilde{\KR}(S;N(\chi_{\underline{k+1}}))\]
is levelwise $(c_2+n)$-connected on fixed points. A similar argument (cf. \cite[3.2]{McCarthy}) shows that it is non-equivariantly levelwise $(c_1+2n)$-connected, where 
\[c_1:=k+\sum_{1\leq j\leq k+1}\conn e_j\]
Therefore the $\mathbb{Z}/2$-spectrum $C$ is
\[(c_1+2n,c_2+n)-(\dim S^{n\rho},\dim (S^{n\rho})^{\mathbb{Z}/2})=(\nu_1+k,\nu_2+k)\]
connected. Since the homotopy fiber of a map of spectra is the loop of the cofiber, the map in the cofiber sequence above is $(\nu_1+k,\nu_2+k)$-connected, that is the $(k+1)$-cube $KR(S;N(\chi))$ is $(\nu_1+k,\nu_2+k)$-cocartesian. Therefore it is $\nu$-cartesian.

Let us show that $\widetilde{\KR}(S;N(-))$ is relatively additive. 
Given a finite based simplicial $\mathbb{Z}/2$-set $B$ in $\mathbb{Z}/2\mbox{-}\mathcal{S}_\ast$, a retractive $\mathbb{Z}/2$-space $(X,p,j)$ in $\mathbb{Z}/2\mbox{-}\mathcal{S}_B$ and a finite set $J$ with involution, we need to estimate the connectivity of
\[\iota\colon\widetilde{\KR}\big(S;N(X\wedge_B(J_+\times B))\big)\longrightarrow {\stackrel[\widetilde{\KR}(S;N(B))]{}{\widetilde{\prod}}}^{\!\!\!\!\!\!\!\!\!\!\!\!J}\widetilde{\KR}(S;N(X))\]
In spectrum level $n$ and simplicial degree $\underline{p}$ this map fits into the commutative diagram
\[\xymatrix@R=16pt{\bigvee\limits_{s\in S^{(n)}_{\underline{p}}}\!\!\!N_s(X\wedge_B(J_+\times B))\ar[rr(.8)]&& \ \ \ \ {\stackrel[\big(\!\!\!\!\!\bigvee\limits_{t\in S^{(n)}_{\underline{p}}}\!\!\!\!\!N_t(B)\big)]{}{\widetilde{\prod}}}^{\!\!\!\!\!\!\!\!\!\!\!\!J}\ \ \Big(\bigvee\limits_{s\in S^{(n)}_{\underline{p}}}N_s(X)\Big)
\\
\bigvee\limits_{s\in S^{(n)}_{\underline{p}}}\!\!\Big({\stackrel[N_s(B)]{}{\widetilde{\prod}}}^{\!\!\!\!\!J}N_s(X)\Big)\ar@{<-}[u(.8)]^-{\simeq}&\!\!\!\!\bigvee\limits_{s\in S^{(n)}_{\underline{p}}}\Big({\bigvee\limits_{N_s(B)}}^{\!\!\!\!\!J}N_s(X)\Big)\ar[l]\ar[r]^-{\cong}\ar[ul(.8)]&{\bigvee\limits
_{\big(\!\!\!\!\!\bigvee\limits_{t\in S^{(n)}_{\underline{p}}}\!\!\!\!\!N_t(B)\big)}}^{\!\!\!\!\!\!\!\!\!\!\!\!\!J}\ \ \Big(\bigvee\limits_{ s\in S^{(n)}_{\underline{p}}}N_s(X)\Big)\ar[u(.7)]_-{\simeq}
}\]
The left-hand vertical map is an equivalence, since $N_s(-)$ is $\mathbb{Z}/2$-linear for every $s$ in $(S^{(n)}_{\underline{p}})^{\mathbb{Z}/2}$ (cf. Remark \ref{linreladd} and Appendix \ref{confGlin}). The right-hand vertical map defines a $\mathbb{Z}/2$-equivalence of symmetric $\mathbb{Z}/2$-spectra by Proposition \ref{wedgesintoprod}. Therefore the connectivity of the top horizontal map is determined by the left-hand bottom map. The calculation of Lemma \ref{connwedgesitoprodtop} expresses its connectivity in terms of the connectivity of the maps $N_s(p)\colon N_s(X)\rightarrow N_s(B)$. By linearity of $N_s(-)$ its splitting $N_s(j)$ fits into a fiber sequence $N_s(B)\rightarrow N_s(X)\rightarrow N_s(X/B)$, and therefore
\[\begin{array}{ll}\conn N_s(p)^H&= \conn N_s(j)^H+1=\conn N_s(X/B)^H\\
&\geq\min_{K\leq H}\conn X/B\geq \min_{K\leq H}\conn p^K\end{array}\]
By Lemma \ref{connwedgesitoprodtop} our map is
\[\nu(H)=\min\{2\conn p^H,\min_{K\lneqq H}\conn p^K\}-1\]
connected in $\pi^{H}_\ast$.
Arguing like before, the fact that $(S,N)$ is $1$-reduced shows that the connectivity of our map in spectrum degree $n$ is $\nu+(2n,n)$, and therefore $\nu$-connected on the homotopy colimit.

It remains to show that $\widetilde{\KR}(S;N(-))$ preserves connectivity. The argument is similar to the calculation of the connectivity of the cofiber $C^{(n)}$ above, using that $N(-)$ preserves connectivity.
\end{proof}

We also extend MacLane homology to a functor $\HR\big(S;N(-)\big)\colon\mathbb{Z}/2\mbox{-}\mathcal{S}_\ast
\to\Sp^{\Sigma}_{\mathbb{Z}/2}$ in a similar way.
The inclusion of wedges into direct sums defines a natural transformation $\widetilde{\KR}\big(S;N(-)\big)\to\HR\big(S;N(-)\big)$.

\begin{prop}\label{Drel}
The functor $\HR\big(S;N(-)\big)\colon\mathbb{Z}/2\mbox{-}\mathcal{S}_\ast
\to\Sp^{\Sigma}_{\mathbb{Z}/2}$ is $\mathbb{Z}/2$-linear, and the inclusion of wedges into direct sums induces a natural $\pi_\ast$-equivalence of symmetric $\mathbb{Z}/2$-spectra,
\[D_\ast\widetilde{\KR}\big(S;N(X)\big)\stackrel{\simeq}{\longrightarrow}D_\ast\HR\big(S;N(X)\big)\stackrel{\simeq}{\longleftarrow}\HR\big(S;N(X)\big)\stackrel{\simeq}{\longleftarrow}\HR\big(S;N\big)\wedge|X|\]
for every $X$ in $\mathbb{Z}/2\mbox{-}\mathcal{S}_\ast$.
\end{prop}

\begin{proof}
To show that $\HR\big(S;N(-)\big)$ sends homotopy cocartesian squares to homotopy cartesian squares, it is enough to show that the map
\[\HR\big(S;N(X)\big)_n\to\Omega\HR\big(S;N(X\wedge S^1)\big)_n\]
is an equivalence of pointed $\mathbb{Z}/2$-spaces for every $X$ in $\mathbb{Z}/2\mbox{-}\mathcal{S}_\ast$ (e.g. by \cite[1.8]{calcIII}). Since both $\HR\big(S;N(X)\big)_n$ and $\HR\big(S;N(X)\big)_{n}^{\mathbb{Z}/2}$ are geometric realizations of simplicial Abelian groups, loop spaces and realizations commute, and the assembly map above factors as
\[|\!\!\!\bigoplus_{s\in sd_e S^{(n)}}\!\!\!N_s(X)|\longrightarrow|\Omega\!\!\!\bigoplus_{s\in sd_e S^{(n)}}\!\!\!N_s(X\wedge S^1)|\stackrel{\simeq}{\longrightarrow}\Omega|\!\!\!\bigoplus_{s\in sd_e S^{(n)}}\!\!\!N_s(X\wedge S^1)|
\]
Since loops commute with indexed direct sums, it is enough to show that
\[\bigoplus_{s\in S^{(n)}_{\underline{p}}}N_s(X)\longrightarrow\bigoplus_{s\in S^{(n)}_{\underline{p}}}\Omega N_s(X\wedge S^1)
\]
is an equivalence of simplicial $\mathbb{Z}/2$-sets for every $\underline{p}$. Non-equivariantly this is just linearity of the Dold-Thom construction $N_s(-)$. On $\mathbb{Z}/2$-fixed points the map is, up to isomorphism, the direct sum of the assembly maps
\[\bigoplus_{s\in (S^{(n)}_{\underline{p}})^{\mathbb{Z}/2}}\!\!\! N_s(X)^{\mathbb{Z}/2}\oplus \bigoplus_{s\in F^{(n)}_{\underline{p}}}N_s(X) \longrightarrow\bigoplus_{s\in (S^{(n)}_{\underline{p}})^{\mathbb{Z}/2}}\!\!\!\Omega N_s(X\wedge S^1)^{\mathbb{Z}/2}\oplus \bigoplus_{s\in F^{(n)}_{\underline{p}}}\Omega N_s(X\wedge S^1)
\]
where $F^{(n)}_{\underline{p}}$ is a set of representatives for the free orbits of $S^{(n)}_{\underline{p}}$.
This is an equivalence by linearity of $N_s(-)$ on the second summands, and by $\mathbb{Z}/2$-linearity of $N_s(-)$, for $s$ fixed in $S^{(n)}_{\underline{p}}$, on the first summands (see Appendix \ref{confGlin}). To see that $\HR\big(S;N(-)\big)$ sends indexed wedges to indexed products, we show in Appendix \ref{confGlin} that for every finite $\mathbb{Z}/2$-set $J$ and Abelian group with $\mathbb{Z}/2$-action $N$, the canonical map from wedges to products is an equivariant isomorphism $N(\bigvee_JX)\cong\prod_JN(X)$. Thus the canonical map for $\HR\big(S;N(-)\big)$ is also an isomorphism
\[|\bigoplus_{s\in S^{(n)}_{\underline{p}}}N_s(\bigvee_J X)|\cong|\bigoplus_{s\in S^{(n)}_{\underline{p}}}\prod_JN_s(X)|\cong \prod_J|\bigoplus_{s\in S^{(n)}_{\underline{p}}}N_s(X)|\]
This shows that $\HR\big(S;N(-)\big)$ is $\mathbb{Z}/2$-linear, and therefore that it is equivalent to its differential (see \ref{rightestimate} and \ref{corlinearsmashspectrum}).

It remains to show that $D_\ast\widetilde{\KR}\big(S;N(-)\big)\to D_\ast\HR\big(S;N(-)\big)$ is an equivalence. The calculation of Lemma \ref{connwedgesitoprodtop} shows that the equivariant connectivity of $\widetilde{\KR}\big(S;N(X)\big)\to \HR\big(S;N(X)\big)$ is
\[\big(2\conn X+1,\min\{2\conn X^{\mathbb{Z}/2},\conn X\}+1\big)\]
Therefore $\Omega^{n\rho}\widetilde{\KR}\big(S;N(X\wedge S^{n\rho})\big)\to \Omega^{n\rho}\HR\big(S;N(X\wedge S^{n\rho})\big)$ is non-equivariantly $\nu_1$-connected, for
\[\nu_1=2(\conn X+2n)+1-2n=2n + 2\conn X+1\]
and $\nu_2$-connected on fixed points, for
\[\nu_2=\min\Big\{2(\conn X+2n)+1-2n,\min\big\{2(\conn X^{\mathbb{Z}/2}+n),\conn X+2n\big\}+1-n\Big\}\]
For $n$ sufficiently large, the second term of the outer minimum in $\nu_2$ is smaller than the first, and $\nu_2$ becomes
\[\nu_2=\min\big\{2(\conn X^{\mathbb{Z}/2}+n),\conn X+2n\big\}+1-n=n+\min\big\{2\conn X^{\mathbb{Z}/2},\conn X\big\}+1\]
The equivariant connectivity $(\nu_1,\nu_2)$ diverges with $n$, and the map on differentials is an equivalence.
\end{proof}

\subsection{Real $K$-theory of semidirect products as Real $K$-theory with coefficients}\label{splitting}

Let $(A,w,\epsilon)$ be a Wall antistructure and $(M,h)$ a bimodule over $(A,\epsilon,w)$ as defined in \ref{bimodule}. The goal of this section is to prove that the Real $K$-theory $\widetilde{\KR}(A\ltimes M)$ of Section \ref{KRsec} is equivalent to the Real $K$-theory of a certain coefficients system, in the sense of Definition \ref{defrelkr}. This is a Real analogue of Theorem \cite[4.1]{DM}. We use this comparison to finish the proofs of Theorem $A$ and Theorem $C$, by means of Theorems \ref{Drel} and \ref{KRanalytic} respectively.

The coefficients system whose $K$-theory compares to $\widetilde{\KR}(A\ltimes M)$ is constructed from a bimodule with duality on $\mathcal{P}_A$ induced by $M$, via the construction of Example \ref{coeffsystbimod}.
Let us define a bimodule $H^M\colon\mathcal{P}_A^{op}\otimes \mathcal{P}_A\rightarrow Ab$ by sending two $A$-modules $P$ and $Q$ to the Abelian group of module maps
\[H^M(P,Q)=\hom_A(P,Q\otimes_AM)\]
There is a duality $J\colon H^M\Rightarrow H^M\circ D_\gamma$ on $H^M$ (in the sense of Definition \ref{defbimoddual}) defined as
\[J_{P,Q}\colon\hom_A(P,Q\otimes_AM)\stackrel{\widehat{J}_{P,Q}}{\xrightarrow{\hspace*{.7cm}}}
\hom_A\big(DQ,\hom_A(P,M_w)\big)\stackrel{\cong}{\longrightarrow}\hom_A\big(DQ,(DP)\otimes_AM\big)\]
Here $M_w$ is the right $A$-module structure on $M$ defined by $m\cdot a:=w(a)\epsilon\cdot m$. The second map  is induced by the canonical isomorphism $ \hom_A(P,A_s)\otimes_A M\rightarrow\hom_A(P,M_w)$. The map $\widehat{J}$ sends a module map $f$ to
\[\xymatrix{\widehat{J}(f)(\lambda)\colon P\ar[r]^-{f}& Q\otimes_AM\ar[r]^{\lambda\otimes\id}& 
A_s\otimes_AM\ar[rr]^-{(w(-)\cdot\epsilon)\otimes h}&& A\otimes_AM_w\cong M_w }\]
By strictifying the dualities this construction induces a bimodule $\mathcal{D}H^M$ over the exact category with strict duality $\mathcal{D}\mathcal{P}_A$, with associated coefficients system $(S(\mathcal{D}\mathcal{P}_A),N_{\mathcal{D}H^M})$, as in Example \ref{coeffsystbimod}. Its associated Real $K$-theory and Real MacLane homology are denoted respectively by $\widetilde{\KR}(A;M)$ and $\HR(A;M)$.

\begin{theorem}\label{relKRandKR}
Let $(M,h)$ be a bimodule over a Wall antistructure $(A,\epsilon,w)$. There is a natural zig-zag of $\pi_\ast$-equivalences of symmetric $\mathbb{Z}/2$-spectra
\[\widetilde{\KR}\big(A\ltimes M\big)\simeq \widetilde{\KR}\big(A;M(S^{1,1})\big)\]
where $M(S^{1,1})$ is the Dold-Thom construction of the Real circle $S^{1,1}=\Delta^{1}/\partial$, with involution $(m_1,\dots,m_p)\mapsto (h(m_p),\dots,h(m_1))$.
\end{theorem}

The proof of Theorem \ref{relKRandKR} is technical, and it is given at the end of the section. We first end the proofs of Theorems $A$ and $C$ assuming Theorem \ref{relKRandKR}.

\begin{proof}[Proof of Theorem $A$]
We need show that $D_\ast \widetilde{\KR}\big(A\ltimes M(-)\big)$ is equivalent to $\HR(A;{M(S^{1,1})})$. This immediate by \ref{relKRandKR} and \ref{Drel}.
\end{proof}

\begin{proof}[Proof of Theorem $C$]
By Theorem \ref{relKRandKR} it is enough to show that the functor
\[\widetilde{\KR}\big(S(A),N_{M(S^{1,1}) (-)}\big)\colon\mathbb{Z}/2\mbox{-}\mathcal{S}_\ast\longrightarrow\Sp^{\Sigma}_{\mathbb{Z}/2}\]
is $\mathbb{Z}/2$-$\rho$-analytic, for the function $\rho=(-1,0)$.
Here $M(S^{1,1})(X)$ is the iterated configuration bimodule. It is isomorphic to $M(S^{1,1}\wedge X)$ where  $S^{1,1}\wedge X$ is the smash product of the pointed Real simplicial set $S^{1,1}$ and the pointed $\mathbb{Z}/2$-simplicial set $X$. This is the bisimplicial set obtained by taking the smash product levelwise. The $\mathbb{Z}/2$-action is simplicial in the $X$ factor, and Real in the $S^{1,1}$ factor. As usual we realize all the simplicial directions after taking Real $K$-theory, and we take the diagonal of the $\mathbb{Z}/2$-actions coming from each simplicial direction.

The functor $\widetilde{\KR}\big(S(A);N_{M(S^{1,1}\wedge (-))}\big)$  is of the form $\widetilde{\KR}\big(S;N(-)\big)$ precomposed with the functor $S^{1,1}\wedge(-)$.
By Theorem \ref{KRanalytic} we know that given a strongly cocartesian $(n+1)$-cube in $\mathbb{Z}/2\mbox{-}\mathcal{S}_\ast$ with initial maps $e_i$, its image under the composite functor is $\nu=(\nu_1,\nu_2)$-cartesian, for
\[\begin{array}{ll}\nu_1&=\sum_{i=1}^{n+1}\conn(e_i\wedge S^{1,1})\\
&=\sum_{i=1}^{n+1}(\conn e_i+1)
\end{array}\]
and
\[\begin{array}{ll}\nu_2&=\sum_{i=1}^{n+1}\min\{\conn(e_i\wedge S^{1,1}),\conn(e_i\wedge S^{1,1})^{\mathbb{Z}/2} \}\\
&=\sum_{i=1}^{n+1}\min\{\conn e_i+1,\conn e_{i}^{\mathbb{Z}/2}\}
\end{array}\]
This shows that our functor satisfies $E^{\mathbb{Z}/2}_n(\rho,0)$,
and it is therefore $\mathbb{Z}/2$-$\rho$-analytic (with function $q=0$).
\end{proof}

The rest of the paper is dedicated to proving Theorem \ref{relKRandKR}. The equivalence of \ref{relKRandKR} is induced by a zig-zag of Real bisimplicial sets of the form
\[\xymatrix@C=17pt{\coprod\limits_{(S^{2,1}_{\sbt})\mathcal{D}\mathcal{P}_A}
\!\!\!\!\!\!\mathcal{D}M(S^{1,1})_{\sbt}\ar[r]^-{\simeq}& Ni\big((S^{2,1}_{\sbt})\mathcal{D}\mathcal{P}_A\big)\ltimes
\big(\mathcal{D}(H^M)_{\sbt}\big)\ar[r]^-{\simeq}&N i(S^{2,1}_{\sbt})\mathcal{D}\mathcal{P}_{A\ltimes M}&\Ob (S^{2,1}_{\sbt})\mathcal{D}\mathcal{P}_{A\ltimes M}\ar[l]_{\simeq}}\]
The symbol $\ltimes$ denotes the semi-direct product of categories and bimodules, defined in \ref{semi} below. We define these maps and we prove that they are $\mathbb{Z}/2$-equivalences, in Proposition \ref{splitPA} for the middle map, and in Proposition \ref{swallow} for the left and right-hand maps.

Let $C$ be an exact category and $M\colon C^{op}\otimes C\rightarrow Ab$ an additive functor. For a morphism $f\colon c\to d$ in $C$ and an object $b$ in $C$, we denote
\[f_{\ast}=M(\id_b\otimes f)\colon M(b,c)\to M(b,d) \ \ \ \ \ \ \mbox{and}\ \ \ \ \ \ f^{\ast}=M(f\otimes \id_b)\colon M(d,b)\to M(c,b)\]
the induced maps.
\begin{defn}\label{semi}
The semi-direct product $C\ltimes M$ is the category with the same objects of $C$, and with morphism sets $(C\ltimes M)(c,d)=C(c,d)\times M(c,d)$. Composition is defined by
\[\big(f\colon c\rightarrow d,m\in M(c,d)\big)\circ \big(g\colon b\rightarrow c,n\in M(b,c)\big)=\big(f\circ g,f_\ast n+g^\ast m\big)\]
Strict dualities $D$ on $C$ and $J$ on $M$ induce a strict duality $D\ltimes J$ on $C\ltimes M$, defined by $D$ on objects and by $D\times J$ on morphisms. The projection onto the morphisms of $C$ defines a functor $p\colon C\ltimes M\longrightarrow C$, with a section $s\colon C\longrightarrow C\ltimes M$ defined by the inclusion at the zeros of $M$.
\end{defn}

We show that this semi-direct product construction models split square zero extensions of exact categories.
Let $p\colon B\rightarrow C$ and $s\colon C\rightarrow B$ be exact functors, and $U\colon p\circ s\Rightarrow\id$ a natural isomorphism. Define a bimodule $\ker p\colon C^{op}\otimes C\rightarrow Ab$ by
\[(\ker p)(c,d)=\ker\big(B(s(c),s(d))\stackrel{p}{\longrightarrow} C(ps(c),ps(d))\big)\]
Suppose additionally that for every $f$ in $(\ker p)(c,d)$ and $g$ in $(\ker p)(b,c)$ the composite $f\circ g$ is zero in the Abelian group $B(s(b),s(d))$. Then there is a functor
\[F\colon C\ltimes (\ker p)\longrightarrow B\]
that sends an object $c$ to $s(c)$, and a morphism $\big(f\colon c\rightarrow c',m\in (\ker p)(c,c')\big)$ to $s(f)+m$.

\begin{lemma}\label{classsplit}
Let $(p,s,U)$ be as above, and suppose additionally that the section $s$ is essentially surjective. Then the functor $F\colon C\ltimes (\ker p)\longrightarrow B$
 is an equivalence of categories over $C$.
\end{lemma}

\begin{proof}
The argument is analogous to the classification of split square zero extensions of rings. The functor $F$ is obviously essentially surjective, since $s$ is by assumption. To see that it is fully faithful, define an inverse for
\[F\colon C(c,d)\times (\ker p)(c,d)\longrightarrow B(s(c),s(d))\]
by sending $f\colon s(c)\rightarrow s(d)$ to the pair $\big(U_d\circ p(f)\circ U^{-1}_c,f-s(U_d\circ p(f)\circ U_{c}^{-1})\big)$ where $U_c\colon ps(c)\rightarrow c$ is the natural isomorphism.
\end{proof}

\begin{ex}
Let $M$ be a bimodule over a ring $A$, and $p\colon A\ltimes M\to A$ the projection with zero section $s\colon A\to A\ltimes M$. These ring maps induce functors 
\[p=(-)\otimes_{A\ltimes M}A\colon \mathcal{P}_{A\ltimes M}\to \mathcal{P}_{A} \ \ \ \ \ \mbox{and}\ \ \ \ \ \  s=(-)\otimes_{A}(A\ltimes M)\colon \mathcal{P}_{A}\rightarrow \mathcal{P}_{A\ltimes M}\]
and a natural isomorphism $U\colon P\otimes_{A}(A\ltimes M)\otimes_{A\ltimes M}A\rightarrow P$.
In \cite[1.2.5.1]{DGM} the authors show that the bimodule $\ker p\colon\mathcal{P}_{A}^{op}\otimes \mathcal{P}_{A}\to Ab$ is canonically isomorphic to the bimodule $H^M=\hom_A(-,-\otimes_AM)$, and that the section functor $s$ is essentially surjective. Hence the lemma above describes an equivalence of categories $\mathcal{P}_A\ltimes H^M\stackrel{\simeq}{\longrightarrow}\mathcal{P}_{A\ltimes M}$.
\end{ex}

The next proposition extends this equivalence to the $S^{2,1}_{\sbt}$-construction.

\begin{prop}\label{splitPA} For every $n\geq 0$, the functor
\[F\colon i\big((S^{2,1}_{\sbt})^{(n)}\mathcal{D}\mathcal{P}_A\big)\ltimes
\big(\mathcal{D}(H^M)^{(n)}_{\sbt}\big)\stackrel{}{\longrightarrow} i(S^{2,1}_{\sbt})^{(n)}\mathcal{D}\mathcal{P}_{A\ltimes M}\]
induced by the split functor $(S^{2,1}_{\sbt})^{(n)}p\colon i(S^{2,1}_{\sbt})^{(n)}\mathcal{D}\mathcal{P}_{A\ltimes M}\rightarrow i(S^{2,1}_{\sbt})^{(n)}\mathcal{D}\mathcal{P}_{A}$ commutes with the dualities and it is levelwise an equivalence of categories. In particular it induces a $\mathbb{Z}/2$-equivalence on geometric realizations.
\end{prop}

\begin{proof}
Generally, if $B$ and $C$ carry dualities and $p$ and $s$ commute with the dualities, so does $F\colon C\ltimes (\ker p)\to B$.

We prove that $F$ is a levelwise equivalence of categories. The extension of the bimodule $H^M$ to $(S_{\sbt}^{2,1})^{(n)}\mathcal{P}_A$ defined in \ref{coeffsystbimod} sends a pair of diagrams $X$ and $Y$ in $(S_{\underline{p}}^{2,1})^{(n)}\mathcal{P}_A$ to the Abelian group of natural transformations
\[(H^M)^{(n)}_{\underline{p}}=\hom_A(X,Y\otimes_AM)\]
where the tensor product and the sum of maps of diagrams are taken objectwise.
The identification between $\ker p$ and $M^H$ of \cite[1.2.5.1]{DGM} extends to an isomorphism $\ker (S^{2,1}_{\sbt})^{(n)}p\cong \mathcal{D}(H^M)^{(n)}_{\sbt}$. 
By Lemma \ref{classsplit} it is enough to show that the section $ i(S^{2,1}_{\sbt})^{(n)}\mathcal{D}\mathcal{P}_{A}\to i(S^{2,1}_{\sbt})^{(n)}\mathcal{D}\mathcal{P}_{A\ltimes M}$ is essentially surjective.  As $\mathcal{D}\mathcal{P}_{A}$ and $\mathcal{P}_{A}$ are naturally equivalent categories, it suffices to show that the section $i(S^{2,1}_{\sbt})^{(n)}\mathcal{P}_{A}\to i(S^{2,1}_{\sbt})^{(n)}\mathcal{P}_{A\ltimes M}$ is essentially surjective. For $n=0$ this is the section 
\[s=(-)\otimes_{A}(A\ltimes M)\colon \mathcal{P}_A\longrightarrow \mathcal{P}_{A\ltimes M}\]
which is essentially surjective by \cite[1.2.5.4]{DGM}. For $n\geq 1$ we show more generally that if $s\colon C\to B$ is an essentially surjective exact functor, and $B$ is a split-exact category, then $S^{2,1}_{\sbt} s\colon S^{2,1}_{\sbt}C\to S^{2,1}_{\sbt}B$ is also essentially surjective. The result follows by induction on $n$ as $(S^{2,1}_{\sbt})^{(n)}\mathcal{P}_{A\ltimes M}$ is split-exact.
Since $B$ is split exact, a diagram $X$ in $S^{2,1}_pB$ is (non-canonically) isomorphic to the diagram $Y$ in $S^{2,1}_pB$ with vertices
\[Y_\theta=\bigoplus_{\rho=(0^i1^j2^{p-i-j+1})\in r(\theta)}\ker\big(X_{i-1<i+j-1<i+j}\longrightarrow X_{i\leq i+j-1<i+j}\big)\]
Here $r(\theta)$ is the set of retractions for the map $\theta\colon [2]\longrightarrow [p]$, and the maps of the diagram $Y$ are inclusions and projections of the direct summands. This splitting is the $S^{2,1}_p$-analogue of the result that the objects of a diagram in $S_pB$ decompose as direct sums of the diagonal objects. We refer to \cite{IbLars} for a proof.
By the above splitting, it is enough to find a diagram in $S^{2,1}_kC$ whose image by $s$ is isomorphic to $Y$.
We denote the kernel in the above splitting corresponding to a retraction $\rho$ by $b_\rho$. As $s$ is essentially surjective one can choose isomorphisms $\epsilon_\rho\colon b_\rho{\rightarrow} s(c_\rho)$
for some objects $c_\rho$ in $C$. Since the maps in the diagram $Y$ are all projections and inclusions, these arbitrary choices of isomorphisms fit together into an isomorphism of diagrams $Y\cong \bigoplus_{\rho\in r(\theta)}s(c_\rho)$.
Define $Z$ in $S^{2,1}_kC$ to have vertices $Z_\theta=\bigoplus_{\rho\in r(\theta)}c_\rho$
and projections and inclusions as maps.
Since $s$ is exact it is in particular additive, and there are isomorphisms $s(Z)\cong \bigoplus_{\rho\in r(\theta)}s(c_\rho)\cong Y\cong X$.

We are left with showing that $F$ induces a $\mathbb{Z}/2$-equivalence on realizations. This is a general property of equivalences of categories, proved in Lemma \ref{eqeqofcat} below.
\end{proof}

\begin{lemma}\label{eqeqofcat}
Let $A$ and $B$ be categories with strict duality, and let $F\colon A\to B$ be a fully faithful and essentially surjective functor which commutes strictly with the dualities. Then $F$ induces a $\mathbb{Z}/2$-equivalence on geometric realizations.
\end{lemma}

\begin{proof} 
Let us denote by $D_A$ and $D_B$ the respective dualities on $A$ and $B$.
A choice of objects $a_b$ in $A$ and of isomorphisms $\epsilon_b\colon F(a_b)\to b$, for every object $b$ in $B$, gives an inverse functor $F'\colon B\to A$. It is defined on objects by $F'(b)=a_b$, and on morphisms by
\[F'\big(b\stackrel{f}{\longrightarrow} b'\big)=F^{-1}\big(F(a_b)\stackrel{\epsilon_b}{\longrightarrow}b\stackrel{f}{\longrightarrow} b'\stackrel{\epsilon_{b'}^{-1}}{\longrightarrow} F(a_b')\big)\]
The functor $F'$ commutes with the dualities up to a natural isomorphism $\xi\colon F'D_B\Rightarrow D_AF'$ defined by
\vspace{-.7cm}

\[\xi_b=F^{-1}\big(FF'D_B(b)=F(a_{D_Bb})\stackrel{\epsilon_{D_Bb}}{\xrightarrow{\hspace*{.7cm}}}D_Bb\stackrel{D_B(\epsilon_{b})}{\xrightarrow{\hspace*{.7cm}}}D_BF(a_b)=FD_A(a_b)=FD_AF'(b)\big)\]
The pair $(F',\xi)$ induces an inverse $\mathcal{D}(F',\xi)$ for the functor $\mathcal{D}F\colon \mathcal{D}A\to\mathcal{D}B$ that commutes strictly with the dualities. It sends an object $\big(b,d,\phi\colon d\stackrel{\cong}{\to} D_B(b)\big)$ of $\mathcal{D}B$ to
\[\big(F'(b),F'(d), F'(d)\stackrel{F'(\phi)}{\longrightarrow}F'D_B(b)\stackrel{\xi_b}{\longrightarrow}D_AF'(b)\big)\]
and a morphism $(f,g)$ in $\mathcal{D}B$ to $(F'(f),F'(g))$.
This induces a $\mathbb{Z}/2$-homotopy inverse for $\mathcal{D}F$ on geometric realizations. There is a natural equivalence of categories $A\to \mathcal{D}A$ that sends $a$ to $(D_Aa,a,\id_{D_Aa})$, giving a commutative square of $\mathbb{Z}/2$-spaces
\[\xymatrix{|A|\ar[d]\ar[r]^-F&|B|\ar[d]\\
|\mathcal{D}A|\ar@<1ex>[r]^-{\mathcal{D}F}_-\simeq&|\mathcal{D}B|\ar@<1ex>[l]^-{\mathcal{D}(F',\xi)}
}\]
It remains to show that the vertical maps are $\mathbb{Z}/2$-equivalences. We show this for $A$, and we drop the subscript $A$ from the strict duality $D:=D_A$. Since $A\to \mathcal{D}A$ is an equivalence of categories it induces a non-equivariant equivalence on realizations. Therefore we need to show that $|A|^{\mathbb{Z}/2}\to |\mathcal{D}A|^{\mathbb{Z}/2}$ is an equivalence. Taking the edgewise subdivision of the nerve of $A$, one can see that the $\mathbb{Z}/2$-fixed points of $|A|$ are naturally equivalent to the geometric realization of the category $\sym A$, with objects self-dual isomorphisms
\[\Ob \sym A=\left\{\big(a\in A, k\colon a\stackrel{k}{\to}Da\big)\ | \ D(k)=k\right\}\]
and morphisms $(a,k)\to (a',k')$ maps $f\colon a\to a'$ in $A$ which satisfy $k=D(f)k'f$. It remains to show that $\sym A$ is equivalent to $\sym \mathcal{D}A$. There are mutually inverses equivalences of categories $p\colon \sym \mathcal{D}A\to \sym A$  and $s\colon \sym A\to \sym \mathcal{D}A$ defined on objects by
\[\vcenter{\hbox{\xymatrix@=17pt{d\ar[d]_{\phi}^{\cong}&c\ar@{-->}[dl]\ar[d]_{\cong}^{D\phi}\ar[l]_{k}\\
Dc&Dd\ar[l]^{D(k)}
}}}\ \ \stackrel{p}{\longmapsto} \ \ \ 
(c,\phi\circ k)\ \ \ \ \ \ \ \ \ \ \mbox{and} \ \ \ \ \ \ \ \ \ (a,k\colon a\to Da)\ \ \stackrel{s}{\longmapsto} \ \ \ \vcenter{\hbox{\xymatrix@=17pt{Da\ar@{=}[d]&a\ar@{=}[d]\ar[l]_{k}\\
Da&DDa\ar[l]^{D(k)}
}}}\]
and on morphisms by $p(f\colon c\to c',g\colon d'\to d)=f$ and $s(f)=(f,D(f))$.
\end{proof}

Given an exact category $C$ and an additive functor $M\colon C^{op}\otimes C\rightarrow Ab$ let $\coprod_CM$ be the groupoid defined as the disjoint union of  the groups $M(c,c)$. Its objects are the objects of $C$, and it has only endomorphisms, defined by
\[(\coprod_CM)(c,c)=M(c,c)\]
The composition of endomorphisms is the addition in the Abelian groups $M(c,c)$.
Strict dualities $D$ on $C$ and $J$ on $M$ induce a strict duality on $\coprod_CM$ defined by $D$ on objects and by $J\colon M(c,c)\rightarrow M(Dc,Dc)$ on morphisms. There is an embedding $e\colon \coprod_CM\to i(C\ltimes M)$ which is the identity on objects, and that sends a morphism $m$ in $M(c,c)$ to the morphism $(\id_c,m)$. Here $i(C\ltimes M)$ denotes the subcategory of isomorphisms of $C\ltimes M$.

\begin{prop}\label{swallow}
For every integer $n\geq 1$, the embedding
\[e\colon \coprod_{(S^{2,1}_{\sbt})^{(n)}\mathcal{D}C}\mathcal{D}M_{\sbt}^{(n)}{\longrightarrow} i\big((S^{2,1}_{\sbt})^{(n)}\mathcal{D}C\big)\ltimes
\mathcal{D}M^{(n)}_{\sbt}\]
induces a weak $\mathbb{Z}/2$-equivalence on geometric realizations.
\end{prop}

\begin{rem}\label{Obincl}
For the trivial bimodule $M=0$, Proposition \ref{swallow} states that the inclusion $\Ob S^{2,1}_{\sbt}\mathcal{D}C\rightarrow iS^{2,1}_{\sbt}\mathcal{D}C$ induces a $\mathbb{Z}/2$-equivalence on geometric realizations. The proof of Proposition \ref{swallow} uses a swallowing argument completely analogous to the argument of \cite[1.4-Cor(2)]{Wald}, which shows that the inclusion as a discrete simplicial category  $\Ob S_{\sbt}C\rightarrow iS_{\sbt}C$ induces an equivalence of classifying spaces.
\end{rem}

\begin{proof}[Proof of \ref{swallow}]
By induction, it is enough to prove the proposition for $n=1$.
We are going to show that for every fixed integer $k$ the map 
\[N_{2k+1}e\colon N_{2k+1}\coprod_{S^{2,1}_{\sbt}\mathcal{D}C}\mathcal{D}M_{\sbt}\to N_{2k+1}(i\big(S^{2,1}_{\sbt}\mathcal{D}C\big)\ltimes
\mathcal{D}M_{\sbt})\]
admits a $\mathbb{Z}/2$-equivariant homotopy inverse (which is not simplicial in $k$). This will show that the edgewise subdivision of $N_{\sbt}e$  is a levelwise $\mathbb{Z}/2$-equivalence of $\mathbb{Z}/2$-bisimplicial sets, and hence it induces a $\mathbb{Z}/2$-equivalence on geometric realizations.

Since $\coprod_{S^{2,1}_{\sbt}\mathcal{D}C}\mathcal{D}M_{\sbt}$ is a disjoint union of groups, there is a natural isomorphism
\[N_{2k+1}\coprod_{S^{2,1}_{\sbt}\mathcal{D}C}\mathcal{D}M_{\sbt}\cong \coprod_{S^{2,1}_{\sbt}\mathcal{D}C}N_{2k+1}\mathcal{D}M_{\sbt}\cong
\coprod_{S^{2,1}_{\sbt}\mathcal{D}C}\mathcal{D}M^{\oplus (2k+1)}_{\sbt} \]
where the direct sums are taken objectwise.
An element of $N_{2k+1}(i\big(S^{2,1}_{p}\mathcal{D}C\big)\ltimes
\mathcal{D}M_{p})$ is a pair $(\underline{\phi},\underline{m})$ of a diagram of isomorphisms
\[\underline{\phi}=\left(\vcenter{\hbox{\xymatrix{Y_0\ar[d]_{\phi_0}&
Y_1\ar[d]_{\phi_1}\ar[l]_-{b_1}&
\ar[l]_-{b_2}\dots&Y_k\ar[d]_{\phi_k}\ar[l]_{b_k}&Y_{k+1}\ar@{-->}[dl]\ar[l]_{b_{k+1}}\ar[d]^{\phi_{k+1}}&
\dots\ar[l]_{b_{k+2}} &&Y_{2k+1}\ar[ll]_-{b_{2k+1}}\ar[d]^{\phi_{2k+1}}
\\
D(X_0)&D(X_1)\ar[l]^-{Da_1}&\ar[l]^-{Da_2}\dots&D(X_k)\ar[l]^-{Da_{k}}&
D(X_{k+1})\ar[l]^{Da_{k+1}}&\dots\ar[l]^-{Da_{k+2}}&&
D(X_{2k+1})\ar[ll]^-{Da_{2k+1}}}}}\right)\]
in $S^{2,1}_pC$, and a collection of elements $m_i$ in the Abelian groups $M_p(X_{i-1},X_i)$. 
We define a homotopy inverse $r\colon N_{2p+1}(i\big(S^{2,1}_{\sbt}\mathcal{D}C\big)\ltimes
\mathcal{D}M_{\sbt})\rightarrow \coprod_{S^{2,1}_{\sbt}\mathcal{D}C}\mathcal{D}M_{\sbt}^{\oplus(2k+1)}$ by contracting the bimodule components of $(\underline{\phi},\underline{m})$ onto the diagonal of the middle square in $\underline{\phi}$. Precisely, $r$ is defined by
\[r(\underline{\phi},\underline{m})=\big(Y_{k+1}\stackrel{b_{k+1}}{\longrightarrow}Y_k\stackrel{\phi_k}{\longrightarrow} D(X_k),r(\underline{m})\in M_{\sbt}(X_{k},X_{k})^{\oplus(2k+1)}\big)\]
where $r(\underline{m})$ has $i$-component 
\[r(\underline{m})_i=\left\{\begin{array}{ll}
(a^{-1}_{k})^{\ast}\dots(a^{-1}_i)^{\ast}(a_{k})_\ast \dots (a_{i+1})_\ast m_i
&,1\leq i\leq k\\
a_{k+1}^{\ast}\dots a_{i-1}^{\ast}(a^{-1}_{k+1})_\ast\dots (a^{-1}_{i})_\ast m_i
& ,k+1\leq i\leq 2k+1
\end{array}\right.\]
This defines a retraction of $N_{2k+1}e$. Moreover $r$ commutes strictly with the dualities, since we are contracting onto the middle square.
We need to define a simplicial homotopy
\[H\colon N_{2k+1}(i\big(S^{2,1}_{\sbt}\mathcal{D}C\big)\ltimes
\mathcal{D}M_{\sbt})\times\Delta[1]\longrightarrow N_{2k+1}(i\big(S^{2,1}_{\sbt}\mathcal{D}C\big)\ltimes
\mathcal{D}M_{\sbt})\]
between $(N_{2k+1}e)\circ r$ and the identity, which commutes with the dualities. Let us forget for a moment the bimodule component.
Consider $N_{2k+1}iS^{2,1}_{2}\mathcal{D}C=N_{2k+1}i\mathcal{D}C$
as a category with strict duality, with natural transformations of diagrams as morphisms.
The $\mathcal{D}C$ component of $(N_{2k+1} e_2)\circ r_2$ extends to a functor $\lambda\colon \mathcal{N}_{2k+1}i\mathcal{D}C\rightarrow \mathcal{N}_{2k+1}i\mathcal{D}C$ by sending all the morphisms to identities. There is a natural isomorphism $U\colon \id\Rightarrow\lambda$ defined at an object $\underline{\phi}$ by the diagram
\[\xymatrix@1@=6pt@C4pt@R35pt{
&Dc_0&&Dc_1\ar[ll]&&\dots\ar[ll]&&Dc_k\ar[ll]&&
Dc_{k+1}\ar[ll]&&
Dc_{k+2}\ar[ll]&&
\dots\ar[ll]&&
Dc_{2k+1}\ar[ll]\\
d_0\ar[ur]&&
d_1\ar[ur]\ar[ll]&&
\dots\ar[ll]&&
d_k\ar[ur]\ar[ll]&&
d_{k+1}\ar[ur]\ar[ll]&&
d_{k+2}\ar[ur]\ar[ll]&&
\dots\ar[ll]&&
d_{2k+1}\ar[ur]\ar[ll]\\
&Dc_k\ar[uu]|<<<<<<<<<{D(a_k\dots a_1)}&&
Dc_k\ar[uu]|<<<<<<<<<{D(a_k\dots a_2)}\ar@{=}[ll]&&
\dots\ar@{=}[ll]&&
Dc_k\ar@{=}[uu]\ar@{=}[ll]&&
Dc_{k}\ar[uu]|<<<<<<<<<{D(a_{k+1}^{-1})}\ar@{=}[ll]&&
Dc_{k}\ar[uu]|<<<<<<<<<{\ \ D(a_{k+1}a_{k+2})^{-1}}\ar@{=}[ll]&&
\dots\ar[ll]&&
Dc_{k}\ar[uu]|<<<<<<<<<{\ \ D(a_{k+1}\dots a_{2k+1})^{-1}}\ar@{=}[ll]\\
d_{k+1}\ar[uu]|>>>>>>>>>>>{\ b_1\dots b_{k+1}}\ar[ur]|-{\ \phi_k b_{k+1}}&&
d_{k+1}\ar[uu]|>>>>>>>>>>>{\ b_2\dots b_{k+1}}\ar[ur]|-{\ \phi_k b_{k+1}}\ar@{=}[ll]
&&\dots\ar@{=}[ll]&&
d_{k+1}\ar[uu]|>>>>>>>>>>>{b_{k+1}}\ar[ur]|-{\ \phi_k b_{k+1}}\ar@{=}[ll]&&
d_{k+1}\ar@{=}[uu]\ar[ur]|-{\ \phi_k b_{k+1}}\ar@{=}[ll]&&
d_{k+1}\ar[uu]|>>>>>>>>>>>{b_{k+2}^{-1}}\ar[ur]|-{\ \phi_k b_{k+1}}\ar@{=}[ll]&&
\dots\ar@{=}[ll]&&
d_{k+1}\ar[uu]|>>>>>>>>>>>{(b_{k+2}\dots b_{2k+1})^{-1}}\ar[ur]|-{\ \phi_k b_{k+1}}\ar@{=}[ll]
}\]
This natural transformation respects the dualities, in the sense that $DU_{\underline{\phi}}\circ U_{D\underline{\phi}}=\id_{D\underline{\phi}}$.
Write this natural transformation as a functor $U\colon N_{2k+1}i\mathcal{D}C\times[1]\longrightarrow N_{2k+1}i\mathcal{D}C$.
In simplicial degree $p$, define a homotopy 
\[K_p\colon N_{2k+1}iS^{2,1}_{p}\mathcal{D}C\times\Delta[1]_p\longrightarrow N_{2k+1}iS^{2,1}_{p}\mathcal{D}C\]
by sending $(\underline{\phi},\sigma\colon [p]\rightarrow [1])$ to the composite
\[\xymatrix{Cat([2],[p])\ar[r]^-{(\id, ev_{1})}&Cat([2],[p])\times [p]\ar[r]^-{\underline{\phi}\times \sigma}&N_{2k+1}i\mathcal{D}C\times [1]\ar[r]^-{U}& N_{2k+1}i\mathcal{D}C}\]
Here we used the identification $N_{2k+1}iS^{2,1}_{p}\mathcal{D}C\cong iS^{2,1}_{p}N_{2k+1}\mathcal{D}C$,
and $ev_1\colon Cat([2],[p])\rightarrow [p]$ is the evaluation functor that sends $\theta\colon [2]\to[p]$ to $\theta(1)$. The construction of this homotopy is similar to the one defined in \cite[1.2.3.2]{DGM}. The map $K$ defines a simplicial homotopy between the $iS^{2,1}_{\sbt}\mathcal{D}C$ component of $(N_{2k+1}e)\circ r$ and the identity. Now we reintroduce the bimodule components. We label the objects and the maps in the diagram $K(\underline{\phi},\sigma)$ by
\[K(\underline{\phi},\sigma)=\left(\vcenter{\hbox{\xymatrix{\sigma (Y_0)\ar[d]_{\sigma(\phi_0)}&
\sigma(Y_1)\ar[d]_{\sigma(\phi_1)}\ar[l]_-{\sigma(b_1)}&
\ar[l]_-{\sigma(b_2)}\dots &&\sigma(Y_{2k+1})\ar[ll]_-{\sigma(b_{2k+1})}\ar[d]^{\sigma(\phi_{2k+1})}
\\
D(\sigma(X_0))&D(\sigma(X_1))\ar[l]^-{D\sigma(a_1)}&\ar[l]^-{D\sigma(a_2)}\dots&&
D(\sigma(X_{2k+1}))\ar[ll]^-{D\sigma(a_{2k+1})}}}}\right)\]
In order to define the homotopy on the bimodule component, we need isomorphisms $\sigma(X_i)\to X_i$ to push forward and pull back the $m_i$'s.
By definition, $\sigma(X_i)$ is the diagram of $S^{2,1}_{p}C$ whose vertex at $\theta\colon [2]\to [p]$ is
\[\sigma(X_i)_{\theta}=\left\{\begin{array}{ll}(X_i)_{\theta}&,\ \sigma\theta (1)=0\\
(X_k)_{\theta}&,\ \sigma\theta (1)=1
\end{array}\right.\]
Let $f_{\sigma}^i\colon \sigma (X_i)\to X_i$ be the isomorphism with components
\[(f_{\sigma}^i)_\theta=\left\{\begin{array}{ll}\id_{(X_i)_\theta}&,\ \sigma\theta (1)=0\\
(a_k\dots a_{i+1})^{-1}_{\theta}&,\ \sigma\theta (1)=1,\ 1\leq i\leq k\\
(a_i\dots a_{k+1})_{\theta}&,\ \sigma\theta (1)=1,\ k<i\leq 2k+1
\end{array}\right.\]
We define the simplicial homotopy $H\colon N_{2k+1}(i\big(S^{2,1}_{\sbt}\mathcal{D}C\big)\ltimes
\mathcal{D}M_{\sbt})\times\Delta[1]\longrightarrow N_{2k+1}(i\big(S^{2,1}_{\sbt}\mathcal{D}C\big)\ltimes
\mathcal{D}M_{\sbt})$ by sending
\[(\underline{\phi},\underline{m},\sigma\colon [p]\rightarrow [1])\longmapsto(K(\underline{\phi},\sigma),K(\underline{\phi},\sigma)_\ast \underline{m})\] where $K(\underline{\phi},\sigma)_\ast \underline{m}$ has $i$-component
\[(K(\underline{\phi},\sigma)_\ast \underline{m})_i= (f_{\sigma}^{i-1})^\ast(f_{\sigma}^{i})_{\ast}^{-1} m_i \ \ \ \ \ \in M(\sigma(X_{i-1}),\sigma(X_i))\]
\end{proof}

\begin{proof}[Proof of \ref{relKRandKR}]
For every $n\geq 0$, there is a commutative diagram of Real $(n+1)$-simplicial sets
\[\xymatrix{\coprod\limits_{(S^{2,1}_{\sbt})^{(n)}\mathcal{D}\mathcal{P}_A}
\!\!\!\!\!\!\!\!\!\mathcal{D}M(S^{1,1})^{(n)}_{\sbt}\ar[r]^-{\cong}\ar[dr]& N_{\sbt}\!\!\!\!\!\!\coprod\limits_{(S^{2,1}_{\sbt})^{(n)}\mathcal{D}\mathcal{P}_A}
\!\!\!\!\!\!\!\!\mathcal{D}M_{\sbt}^{(n)}\ar[r]^-{F\circ e}\ar[d]&N_{\sbt} i(S^{2,1}_{\sbt})^{(n)}\mathcal{D}\mathcal{P}_{A\ltimes M}\ar[d]_{p}&\Ob (S^{2,1}_{\sbt})^{(n)}\mathcal{D}\mathcal{P}_{A\ltimes M}\ar[d]_{p}\ar[l]_{\simeq}\\
&\Ob(S^{2,1}_{\sbt})^{(n)}\mathcal{D}\mathcal{P}_A
\ar[r]^-{\simeq}&N_{\sbt}i(S^{2,1}_{\sbt})^{(n)}\mathcal{D}\mathcal{P}_A&\Ob(S^{2,1}_{\sbt})^{(n)}\mathcal{D}\mathcal{P}_A\ar[l]_{\simeq}
}\]
The three inclusions of objects are equivalences by Remark \ref{Obincl}.
The collection of geometric realizations
\[\KR_n(A;M(S^{1,1}))=|\coprod\limits_{(S^{2,1}_{\sbt})^{(n)}\mathcal{D}\mathcal{P}_A}
\!\!\!\!\!\!\!\!\!\mathcal{D}M(S^{1,1})^{(n)}_{\sbt}|\] has the structure of a symmetric $\mathbb{Z}/2$-spectrum analogous to the one for $\widetilde{\KR}(A;M(S^{1,1}))$, and the diagram above is a diagram of symmetric $\mathbb{Z}/2$-spectra.
Moreover the composite $F\circ e$ is an equivalence for $n\geq 1$ by Propositions \ref{splitPA} and \ref{swallow} (see also Lemma \ref{eqeqofcat}), and hence it induces a $\pi_\ast$-equivalence of $\mathbb{Z}/2$-spectra.

As the Real K-theory spectrum $\widetilde{\KR}(A\ltimes M,w\ltimes h,(\epsilon,0))$ is defined as the homotopy fiber of the right-hand vertical map,
it is enough to show that the homotopy fiber of the left-hand map $\KR(A;M(S^{1,1}))\to \KR(A)$ is $\pi_\ast$-equivalent to $\widetilde{\KR}(A;M(S^{1,1}))$. This is a consequence of the following general argument, which uses only that the map $\KR(A;M(S^{1,1}))\to \KR(A)$ is split. Let $p\colon E\to W$ be a map of $\mathbb{Z}/2$-spectra, with a section $s\colon W\to E$. The sequence $W\to E\to \hoc(s)$ is a fiber sequence, and therefore the horizontal homotopy fibers in the square
\[\xymatrix{
E\ar[r]\ar[d]_p&\hoc(s)\ar[d]\\
W\ar[r]&\ast
}\]
are equivalent, showing that the square is homotopy cartesian. Thus the vertical fibers are also equivalent, that is the homotopy fiber of $p$ is equivalent to the homotopy cofiber of $s$. In our case $s\colon  \KR(A)\to  \KR(A;M(S^{1,1}))$ is a levelwise cofibration, and in particular its homotopy cofiber is equivalent to the strict cofiber, which is $\widetilde{\KR}(A;M(S^{1,1}))$ by definition.
\end{proof}

\section{Appendix}

\subsection{Equivariant Dold-Thom construction and $G$-linearity}\label{app1}

Let $G$ be a finite group, and let $M$ be a simplicial $\mathbb{Z}[G]$-module. Define a functor $M(-)\colon G\mbox{-}\mathcal{S}_\ast\to G\mbox{-}\Top_\ast$
by sending a based simplicial $G$-set to the realization of the based bisimplicial $G$-set with horizontal $n$-simplices
\[M(X)_n=M(X_n)=\big(\bigoplus_{x\in X_n}(M \cdot x)\big)/_{M\cdot\ast}=\bigoplus_{x\in X_n\backslash\ast}(M \cdot x)\]
The simplicial maps are defined as in Example \ref{DoldThom}. The group $G$ acts both on $M$ and $X$, by conjugation. We prove the following proposition, which we used extensively in \S\ref{analKR}.

\begin{prop}\label{confGlin}
The functor $M(-)\colon G\mbox{-}\mathcal{S}_\ast\rightarrow G\mbox{-}\Top_\ast$ is a $G$-linear reduced homotopy functor (see \ref{Glin} and \ref{DoldThom}). Moreover it preserves connectivity, in the sense that
\[\conn M(X)^H\geq\min_{K\leq H}\conn X^K\]
for every subgroup $H$ of $G$.
\end{prop}

\begin{lemma} An inclusion of pointed $G$-simplicial sets $X\subset Y$ in  $G\mbox{-}\mathcal{S}_\ast$ induces a fiber sequence of $G$-spaces $M(X)\to M(Y)\to M(Y/X)$.
\end{lemma}

\begin{proof}
The $H$-fixed points of the sequence in simplicial degree $n$ is isomorphic to the projection
\[\bigoplus_{[x]\in (X_n\backslash\ast)/H}M^{H_x}x\longrightarrow \bigoplus_{[y]\in (Y_n\backslash\ast)/H}M^{H_y}y\longrightarrow \bigoplus_{[y]\in (Y_n\backslash X_n)/H}M^{H_y}y \]
where $H_x$ is the stabilizer group of $x$ in $H$.
The projection of this sequence is a Kan fibration of simplicial Abelian groups for every $n$. The bisimplicial Abelian groups $M(X)^H$ and $M(X/Y)^H$ satisfy the conditions for the Bousfield-Friedlander theorem of \cite[IV-4.9]{GJ}, and therefore the realization is a fiber sequence as well.
\end{proof}

\begin{proof}[Proof of \ref{confGlin}]
To prove that $M(-)$ is a homotopy functor, suppose first that $M$ is discrete.
We use a topological model for the functor $M(-)$. Given a $G$-CW-complex $Y$ define
\[M(Y)=\big(\coprod_{n\geq 0}(M^n\times Y^n)/_{\Sigma_n}\big)/\sim\]
where $\sim$ is the standard equivalence relation of the Dold-Thom construction, that collapses the basepoint and identifies the zero labels. The space $M(Y)$ has the quotient topology and its $G$-action is induced by the diagonal action on $M^n\times Y^n$. Notice that $M(-)$ defines a continuous functor from $G$-CW-complex to $G$-spaces, and therefore it preserves $H$-homotopy equivalences.
 If $X$ is a simplicial set, the canonical map $M(X)\rightarrow M(|X|)$
is a natural $G$-homeomorphism. If $f\colon X\rightarrow Y$ is a weak $H$-equivalence of simplicial pointed $G$-sets, the induced map $|X|\rightarrow |Y|$ is a weak  $H$-equivalence of $G$-CW-complexes, and therefore an $H$-homotopy equivalence. By continuity of the functor $M$, the map $M(X)\cong M(|X|)\rightarrow M(|Y|)\cong M(Y)$
is an $H$-homotopy equivalence.
Now let $M$ be a simplicial $\mathbb{Z}[G]$-module, and let $f\colon X\rightarrow Y$ be a weak $H$-equivalence of simplicial pointed $G$-sets. The argument above shows that in every simplicial degree $k$ the map of simplicial $G$-sets
\[M_k(X)\longrightarrow M_k(Y)\]
is an $H$-homotopy equivalence. Therefore its realization $M(X)\rightarrow M(Y)$ is a weak $H$-equivalence, and $M(-)$ is a homotopy functor.

To see that the functor $M(-)$ preserves connectivity,
fix a subgroup $H$ of $G$ and define $c_H:=\min_{K\leq H}\conn X^K$. Suppose first that $X$ is $c_H$-reduced. In this case  $M(X)$ is also $c_H$-reduced, and therefore so is $M(X)^H$. In particular $M(X)^H$ is $c_H$-connected. For the general case, we claim that every pointed $c_H$-connected simplicial $G$-set $X$ is $H$-equivalent to a $c_H$-reduced pointed simplicial $G$-set. Non-equivariantly, one can collapse the simplices of $X$ up to the connectivity of $X$ without changing the weak homotopy type. Collapsing simplices up to simplicial degree $c_H$ gives a $c_H$-reduced pointed simplicial $G$-set that has the same $H$-homotopy type of $X$. Since $M(-)$ is a homotopy functor, $M(X)^H$ is $c_H$-connected. 

Now we prove $G$-linearity. Given a homotopy cocartesian square $X\colon\mathcal{P}(\underline{2})\rightarrow G\mbox{-}\mathcal{S}_\ast$ of pointed simplicial $G$-sets, we can replace the horizontal maps by $G$-cofibrations since $M(-)$ is a homotopy functor. The rows of the diagram
\[\xymatrix{M(X_\emptyset)\ar[r]^{M(f_1)}\ar[d]_{M(f_2)}&N(X_1)\ar[d]^{M(\overline{f}_{2})}\ar[r]&
M\big(\hoc(f_1)\big)\ar[d]_{c}^{\simeq}\\
M(X_2)\ar[r]_{M(\overline{f}_{1})}&M(X_{12})\ar[r]&M\big(\hoc(\overline{f}_{1})\big)
}\]
are fiber sequences by the lemma above, and the map $c$ induced by the map between cofibers is a $G$-equivalence. Therefore the homotopy fibers of $M(f_{2})$ and $M(\overline{f}_{2})$ fit into a fiber sequence
\[\hofib M(f_{2})\longrightarrow\hofib M(\overline{f}_{2})\longrightarrow \hofib(c)\simeq \ast\]
where $\hofib(c)$ is a weakly $G$-contractible space. To finish the proof it remains to show that $M(-)$ sends indexed wedges to indexed products.
Given a pointed simplicial $G$-set $X$ and a finite $G$-set $J$, the canonical map
\[M(\bigvee_JX)\longrightarrow \prod_JM(X)\]
is a $G$-homeomorphism.
\end{proof}

\bibliographystyle{amsalpha}
\bibliography{Gcalculusbib}

\end{document}